\documentclass[a4paper, 10pt, reqno, final]{amsart}

\usepackage{ifdraft}
\usepackage[T1]{fontenc}
\usepackage[utf8]{inputenc}
\usepackage[english]{babel}
\usepackage{amsfonts,amssymb,amsmath,amsthm}
\usepackage{etoolbox}
\usepackage{paralist}
\usepackage{subfigure}

\usepackage[obeyDraft]{todonotes}

\usepackage[top=3cm, bottom=3cm, left=2.5cm, right=2.5cm]{geometry}
\usepackage{varioref}
\usepackage{hyperref}
\usepackage{cleveref}

  \theoremstyle{definition} %% keine Kursivschrift
  \newtheorem{defi}{Definition}[section]
	\crefname{defi}{definition}{definitions}
	
  \newtheorem{bem}[defi]{Remark}
	\crefname{bem}{remark}{remarks}
	
  \newtheorem{bsp}[defi]{Example}
  \newtheorem{examples}[defi]{Examples}

\theoremstyle{plain} %% kursive Schrift
  \newtheorem{lem}[defi]{Lemma}
	\crefname{lem}{lemma}{lemmata}
	
  \newtheorem{thm}[defi]{Theorem}
	\crefname{thm}{theorem}{theorems}
	
  \newtheorem{cor}[defi]{Corollary}
	\crefname{cor}{corollary}{corollaries}
	
	\newtheorem{prop}[defi]{Proposition}
	\crefname{prop}{proposition}{propositions}

	\crefname{propdef}{proposition}{propositions}
	
\newcommand{\N}{\mathbb{N}}

\newcommand{\R}{\mathbb{R}}

\renewcommand{\H}{\mathbb{H}}
\renewcommand{\S}{\mathbb{S}}

\newcommand{\spn}[1]{\langle#1\rangle}
\newcommand{\jac}[2]{\operatorname{#1}_{#2}}
\newcommand{\Pii}[3]{\Pi\!\left(#1;\left.#3\right|#2\right)}
\newcommand{\sgn}{\operatorname{sgn}}

\newcommand{\ro}[2]{\rho_{#2}\!\left(#1\right)}
\newcommand{\FAC}{\Xi}

\newcommand{\f}{\mathfrak{f}}
\renewcommand{\t}{\mathfrak{t}}
\renewcommand{\v}{\mathfrak{v}}
\newcommand{\e}{\mathfrak{e}}
\newcommand{\p}{\mathfrak{p}}
\newcommand{\q}{\mathfrak{q}}
\newcommand{\s}{\mathfrak{s}}
\newcommand{\m}{\mathfrak{m}}
\newcommand{\x}{\mathfrak{x}}

\renewcommand{\c}{\mathfrak{c}}

\newcommand{\Q}{\mathfrak{R}_{\p, \q}}
\newcommand{\T}{\mathfrak{T}_{\p, \q}}

\renewcommand{\L}{\mathcal{L}}
\renewcommand{\pounds}{\Lambda}

\begin{document}
\title{Channel linear Weingarten surfaces in space forms}
\author{Udo Hertrich-Jeromin}
\address{Institute of Discrete Mathematics and Geometry\\
         TU Wien \\
         Wiedner Hauptstra{\ss}e 8-10/104 \\
         1040 Wien, Austria}
\email{udo.hertrich-jeromin@tuwien.ac.at}
\author{Mason Pember}
\address{Department of Mathematical Sciences\\
         University of Bath\\
         Bath \\
         BA2 7AY, UK}
\email{mason.j.w.pember@bath.edu}
\author{Denis Polly}
\address{Institute of Discrete Mathematics and Geometry\\
	TU Wien \\
	Wiedner Hauptstra{\ss}e 8-10/104 \\
	1040 Wien, Austria}
\email{dpolly@geometrie.tuwien.ac.at}
\date{\today}

%%%%%%%%%%%%%%%%%%%%%%%%%%%%%%%%%%%%%%%%%%%%%%%%%%%%%%%%%%%%%%%%%%%%%%%%
%%%Abstract

\begin{abstract}
Channel linear Weingarten surfaces in space forms
are investigated in a Lie sphere geometric setting,
which allows for a uniform treatment of different
ambient geometries.
We show that any channel linear Weingarten surface in
a space form is isothermic and, in particular, a surface
of revolution in its ambient space form.

We obtain explicit parametrisations for channel surfaces
of constant Gauss curvature in space forms,
and thereby for a large class of linear Weingarten
surfaces up to parallel transformation.
\end{abstract}

\subjclass[2020]{53A10 (primary); 53A40, 53C42, 37K35, 37K25 (secondary)}

\keywords{Lie sphere geometry; linear Weingarten surface; channel surface; isothermic surface; isothermic sphere congruence; Omega surface; constant Gauss curvature; Jacobi elliptic function}

\maketitle

\section{Introduction}
%%%%%%%%%%%%%%%%%%%%%%%%%%%%%%%%%%%%%%%%%%%%%%%%%%%%%%%%%%%%%%%%%%%%%%%%
Two different ways to define linear Weingarten surfaces can be found
in the literature, either by requiring an affine relationship between the
principal curvatures or one between the mean and the Gau{ss} curvature. In 
these notes, we adopt the second definition: 
a surface in a space form is called linear Weingarten if, for some 
non-trivial triple $a, b, c\in \R$, its Gauss and mean curvatures
$K$ and $H$ satisfy
\begin{align*}
	aK+2bH+c=0.
\end{align*}
This includes surfaces of constant Gauss or mean curvature (CGC or CMC,
respectively) as well as their parallel surfaces. 

We investigate linear Weingarten surfaces that are additionally channel
surfaces, that is, envelop a $1$-parameter family of spheres.
An example is provided by surfaces that are invariant under a $1$-parameter 
subgroup of rotations in the given space form.
The study of such rotational linear Weingarten surfaces goes back
to Delaunay's investigations of CMC surfaces of revolution in Euclidean
space in \cite{delaunay1841},
but have again sparked interest in recent years from various points
of view, see
\cite{barros2012}, \cite{lopez2009}, \cite{lopez2009a} or most recently 
\cite{arroyo2019}, \cite{dursun2020}, \cite{pampano2020} and references 
therein. Since the considered surfaces arise via the action of isometries 
on a planar profile curve, all curvature notions only depend on the 
profile curve. This fact has been used to prove various classification results 
for these curves and, subsequently, rotational linear Weingarten surfaces. 
However, explicit parametrisation formulas for these surfaces seem only
to be available in special cases.

Channel linear Weingarten surfaces in Euclidean space are either tubular
or surfaces of revolution, and are parallel to either the catenoid
or a CGC surface, as was shown in \cite{hertrich-jeromin2015}
by giving explicit parametrisations in terms of Jacobi's elliptic 
functions (\cite[Chap. 16]{abramowitz1972}).
In this way, the authors obtained a complete and transparent
classification of channel linear Weingarten surfaces in Euclidean space.

In this present text, we aim to complement the existing results by
explicit parametrisations for rotational CGC surfaces, in particular,
in hyperbolic space $\H^3$ and the $3$-sphere $\S^3$ and,
in consequence, for any linear Weingarten surface that is parallel
to such a CGC surface
---
the clear advantage being that many results about such surfaces can
then be derived or verified by mere computation.
Particular attention is paid to a choice of parametrisations that
are well-behaved across singularities of the surfaces,
that necessarily occur in various cases according to Hilbert's theorem.
Our parametrisations provide a complete classification result for
non-tubular channel linear Weingarten surfaces in $\S^3$,
where every such surface is parallel to a CGC surface
(cf \cite{barros2012}),
and encompass the large class of channel linear Weingarten surfaces
in $\H^3$ that have a rotational CGC surface in their parallel 
family (cf \cite{dursun2020}, \cite{lopez2009} or \cite{pampano2020}).

Considering the ambient space form geometries as subgeometries
of Lie sphere geometry will allow for a unified treatment
of the various cases that occur:
note that channel surfaces naturally belong to the realm of sphere geometries,
see \cite{pember2018}, for example.
Further, in \cite{burstall2010}, the authors have shown that 
linear Weingarten surfaces appear as special $\Omega$-surfaces in Lie 
sphere geometry. \Cref{sec:1} will serve the reader as a brief 
introduction into the projective model of this geometry and explain how 
space form geometries may be viewed as subgeometries.
In this way, we may conveniently investigate surfaces in different space forms
in a unified manner.
A Bonnet type theorem (\Cref{prop:Bonnet}) will demonstrate the key role
of CGC surfaces within the class of linear Weingarten
surfaces and will provide for a simple generalisation of our parametrisations
to rotational linear Weingarten surfaces.

As another instance of the unifying sphere geometric treatment,
rotational surfaces will be considered in \Cref{sec:2}, where
we express the Gauss curvature of a rotational 
surface in terms of one parameter. This expression will be used to 
classify rotational CGC surfaces.

In \Cref{sec:3} we provide a Lie geometric version of Vessiot's theorem
\cite[Theorem 3.7.5]{hertrich-jeromin2003}:
as sphere geometries provide a natural ambient setting for channel
surfaces, M\"obius geometry provides for a natural ambient geometry
for isothermic surfaces, that is, surfaces that admit conformal
curvature line parameters. 
Vessiot's theorem states that any channel 
isothermic surface is, upon a suitable stereographic projection 
into Euclidean space, a surface of revolution or has straight 
curvature lines.
Similarly, we shall discover that $\Omega$-surfaces,
a Lie sphere geometric generalisation of isothermic surfaces,
that are additionally channel surfaces are isothermic upon a suitable choice of a M\"{o}bius (sub-)geometry
of Lie sphere geometry (see \Cref{thm:Vessiot}). 

Motivated by \Cref{prop:Bonnet} and \Cref{thm:CLWisRot}, we will investigate 
rotational CGC surfaces in \Cref{sec:4}: elliptic differential 
equations are obtained and used to provide constructions for families of 
such surfaces. The principal aim of this section is to provide general
strategies to solve the occurring differential equations and to discuss
reality of the produced solutions.
Our case 
analysis, depending on the relation between the constant Gauss 
curvature and the curvature of the ambient space form, is yet another 
instance of a splitting that is frequently observed when constructing 
immersions between space forms  \cite{dursun2020}, \cite{ferus1996}, 
\cite{lopez2009a}. 

Specific parametrisations, and the classification results they imply, 
are then stated in \Cref{sec:5}: for instance, all channel linear 
Weingarten surfaces in $\S^3$ are parametrised by the functions given 
in Table \ref{tab:SphericalCase} and a suitable parallel 
transformation.
The corresponding classification in $\H^3$ turns out to be
richer, partly due to the appearance of various types of
``rotations'',
see Tables \ref{tab:HyperbolicCaseElliptic},
\ref{tab:HyperbolicCaseHyperbolic}, and
\ref{tab:HyperbolicCaseParabolic}.

\section{Linear Weingarten surfaces}\label{sec:1}
%%%%%%%%%%%%%%%%%%%%%%%%%%%%%%%%%%%%%%%%%%%%%%%%%%%%%%%%%%%%%%%%%%%%%%%%
We consider parametrised surfaces in space forms. For a unified 
treatment we model the space form geometries as subgeometries of Lie 
sphere geometry. Here is a quick glance at our setup, for details see  
\cite{burstall2018}, \cite{cecil2008} or \cite{hertrich-jeromin2020}. 

\subsection{Lie sphere geometry and its subgeometries}

Consider $\R^{4,2}$, a $6$-dimensional real vector space with inner 
product $(.,.)$ of signature ${(-++++-)}$. We call a vector $\v$ 
\emph{timelike}, \emph{spacelike} or \emph{lightlike} depending on 
whether $(\v, \v)$ is negative, positive or vanishes. 

We call the \emph{projective light cone}
\begin{align*}
 \L := \mathbb{P}L^5
 = \mathbb{P}\{\v\in\R^{4,2}\,|\,(\v,\v)=0\}
 = \{\spn{\v}\subset\R^{4,2}|\v\neq 0~\textrm{is lightlike}\},
\end{align*}
the \emph{Lie quadric}, where $\spn{\v}$ denotes the linear span of 
$\v$. Points in the Lie quadric represent oriented $2$-spheres in 
$3$-dimensional space forms (here, points are spheres with vanishing 
radius). 

Lie sphere transformations are given by the action of orthogonal 
transformations of $\R^{4,2}$ on the Lie quadric. For $A\in O(4,2)$ 
we have $A\cdot \spn{\v} = \spn{A\v}$.  

Let $\p\in\R^{4,2}$ be a unit timelike vector, that is $(\p, \p)=-1$. 
The projective sub-quadric
\begin{align*}
	\mathcal{M}_\p := \{x \in \L|x\perp \p\},
\end{align*}
is a model space of $3$-dimensional M\"{o}bius geometry and 
\emph{M\"{o}bius transformations} are induced by orthogonal 
transformations that fix $\p$. We call $\p$ a \emph{point sphere 
complex}, and elements of $\mathcal{M}_\p$ \emph{point spheres}. 
More generally, every $\v \in \R^{4,2}\setminus \{0\}$ spans a 
\emph{linear sphere complex}, consisting of all spheres which are 
represented by null lines perpendicular to it. A sphere $\spn{\s}\in\L$
 contains a point $\spn{\x} \in \mathcal{M}_\p$ if and only if 
$(\x, \s) =0$. 

If we orthogonally project a sphere $s=\spn{\s}\in \L\setminus \mathcal{M}_\p$ onto $\spn{\p}^\perp$
we obtain a vector $\sigma \in \spn{\p}^\perp \cong \R^{4,1}$ with
\begin{align*}
	(\sigma, \sigma)=(\s, \p)^2>0, 
\end{align*}
i.e.,  the representative of a sphere in the projective model of
 M\"{o}bius geometry as described in \cite{hertrich-jeromin2003}. 
We call $\sigma$ a \emph{M\"{o}bius representative of $s$}.

Let $\q \in \R^{4,2}\setminus \{0\}$ be perpendicular to $\p$ and denote
$\kappa:=-(\q,\q)$. Define the affine sub-quadric
\begin{align*}
	\Q:=\{\x\in L^5|(\x,\p) = 0, ~ (\x,\q)=-1\}. 
\end{align*}
Then, $\Q$ has constant curvature $\kappa$, hence each connected component of $\Q$ 
yields a model for a space form geometry. The isometry group of the space form is then 
denoted as $Iso_{\p,\q}(3)$ and consists of all orthogonal 
transformations that fix $\p$ and $\q$. We call $\q$ the 
\emph{space form vector}.
Spheres represented by null lines orthogonal to $\q$ represent \emph{planes};
hence the set of planes in $\Q$ is identified with
\begin{align*}
	\T:=\{\x \in L^5|(\x,\p)=-1,~(\x,\q)=0\}.
\end{align*}

%%%%%%%%%%%%%%%%%%%%%%%%%%%%%%%%%%%%%%%%%%%%%%%%%%%%%%%%%%%%%%%%%%%%%%%%
\subsection{Surfaces}\label{sec:Surfaces}
Let $\f:\Sigma^2 \to \Q$ parametrise a surface in a space 
form. Its tangent plane congruence, denoted by $\t:\Sigma^2 \to \T$, is 
of the form $\t = \p + \mathfrak{n}$, where $\mathfrak{n}$ denotes the 
usual Gau{ss} map of $\f$ when $(\q,\q)\neq 0$.

The line congruence $\pounds=\spn{\f, \t}$ is called the \emph{Legendre 
lift of $\f$}. $\pounds$ is a \emph{Legendre immersion}, that is, for 
any two sections $\s_1, \s_2 \in \Gamma \pounds$ we have
\begin{align*}
	(\s_1, \s_2) = 0 ~\textrm{and}~(d\s_1, \s_2) = 0,
\end{align*}
and for all $p\in \Sigma^2$
\begin{align*}
 \forall\s\in\Gamma\pounds\colon d_p\s(X)\in\pounds(p)~
  \Rightarrow~X=0.
\end{align*}
We say that $\pounds$ \emph{envelops} a sphere congruence $\spn{\s}:\Sigma^2\to\L$
if $\s(p) \in \pounds(p)$ for all $p\in\Sigma^2$. 

Lie sphere transformations naturally act on Legendre immersions and 
Legendre lifts of surfaces in space forms: for 
$A \in O(4,2)$,
$$A\cdot \pounds:=A\cdot \spn{\f, \t} = \spn{A\f, A\t},$$
which induces a map on surfaces by mapping $\f$ to the point sphere 
congruence enveloped by $A\cdot \pounds$. The following lemma, a 
proof of which can be found in \cite[Sect 2.5]{cecil2008}, shows 
that this map on surfaces is well-defined. 

\begin{lem}\label{lem:UmiqueSurface}
	Given a point sphere complex $\p$ and a Legendre immersion 
$\Lambda$, there is precisely one point sphere congruence $\f$ 
enveloped by $\Lambda$.
\end{lem}

Apart from the point sphere and tangent plane congruences, Legendre 
lifts also envelop their \emph{curvature sphere congruences} $s_1,s_2$:
away from umbilic points, let $(u, v)$ denote curvature line coordinates. Then, the curvature sphere congruences are characterised by
\begin{align*}
	(\s_1)_u, (\s_2)_v \in \Gamma\pounds,
\end{align*}
for any lifts $\s_i$ of $s_i$. Further, we have $\pounds = \spn{\s_1, \s_2}$. 

A sphere congruence $s:\Sigma^2 \to \L$ is \emph{isothermic} 
if it allows for a \emph{Moutard lift $\s$}, that is,
\begin{align*}
	\s_{uv} ||\s, 
\end{align*}
for parameters $(u,v)$, which are then curvature line parameters. 
Equivalently, any lift $\s$ of $s$ satisfies a \emph{Laplace equation}
$0=\s_{uv}+a\s_u+b\s_v+c\s$ with equal \emph{Laplace invariants},
$a_u+ab-c=b_v+ab-c$ (\cite[Chap II]{darboux1889}, \cite{doliwa2003}).

\begin{defi}\label{def:OmegaSurf}
 A surface $\f$ in a space form is called an \emph{$\Omega$-surface},
if its Legendre lift $\pounds$ envelops a (possibly complex conjugate)
pair of isothermic sphere congruences $s^\pm$ that separate the 
curvature spheres harmonically.
We will also call $\pounds$ itself an $\Omega$-surface.
\end{defi}

\begin{bem}
	For real isothermic sphere congruences, this is the definition of
Demoulin in \cite{demoulin1911} and \cite{demoulin1911a}.
However, we will also consider $\Omega$-surfaces enveloping complex isothermic
sphere congruences (also characterised by the existence of Moutard lifts).
If the two isothermic sphere congruences coincide
(with one of the curvature sphere congruences),
the surface $\f$ is called an \emph{$\Omega_0$-surface}.
\end{bem}

In our considerations, we will use a characterisation of 
$\Omega$-surfaces in terms of special lifts of their curvature 
sphere congruences: while one can always choose lifts 
$\s_1 \in \Gamma s_1$ and $\s_2 \in \Gamma s_2$ of the curvature 
sphere congruences of a Legendre lift such that
\begin{align*}
	(\s_1)_u \in \Gamma s_2,~\textrm{and}~(\s_2)_v\in \Gamma s_1,
\end{align*}
for $\Omega$-surfaces, even more can be achieved (\cite{burstall2012}).

\begin{prop}\label{prop:OmegaChar}
	A surface $\f:\Sigma^2 \to \Q$ with curvature sphere 
congruences $s_1$ and $s_2$ is an $\Omega$-surface, if  and only if 
there exists a function $\phi$ and lifts $\s_1, \s_2$ of $s_1, s_2$
such that
\begin{align}\label{eq:OmegaCurvatureSpheres}
	\begin{split}
		(\s_1)_u &= \phi_u \s_2, \\
		(\s_2)_v &= \varepsilon^2\phi_v \s_1,
	\end{split}
\end{align}
where $\varepsilon\in \{1, i\}$. 
\end{prop}

\begin{bem}
For $\Omega_0$-surfaces, \eqref{eq:OmegaCurvatureSpheres} holds with 
$\varepsilon = 0$. 
\end{bem}

\begin{proof}
A proof for this can be found in \cite[Sect 4.3]{polly2017}. We 
summarise the argument given there briefly:

Let $\pounds$ be an $\Omega$-surface with Moutard lifts $\s^\pm$ of 
the enveloped isothermic sphere congruences. Since they separate the 
curvature sphere congruences $s_1$ and $s_2$ harmonically, there are lifts 
$\s_1 \in \Gamma s_1$ and $\s_2 \in \Gamma s_2$ such that
\begin{align*}
	\s^\pm = \s_1 \pm \varepsilon \s_2,
\end{align*}
and functions $\alpha, \overline{\alpha}, \beta, \overline{\beta}$ 
such that
\begin{align*}
	(\s_1)_u = \overline{\alpha}\s_1 + \alpha \s_2 \\
	(\s_2)_v = \beta \s_1 + \overline{\beta}\s_2.
\end{align*}
The Moutard condition $\s^\pm_{uv} || \s^\pm \in \pounds$ yields 
$\overline{\alpha}=\overline{\beta}=0$ as well as 
$\alpha_v=\varepsilon^2 \beta_u$, which is the integrability 
condition of 
\begin{align*}
	\phi_u &= \alpha \\
	\phi_v &=\epsilon^2 \beta.
\end{align*}
Hence \eqref{eq:OmegaCurvatureSpheres} is satisfied with any 
solution of that system. 

Conversely, given lifts satisfying \eqref{eq:OmegaCurvatureSpheres}, 
the two sphere congruences given by
${s^\pm = \spn{\s^\pm}}$ with ${\s^\pm := \s_1 \pm \varepsilon\s_2}$
obviously separate the curvature spheres harmonically. They are also
isothermic, as 
\begin{align*}
	(\s^\pm)_{uv} = \varepsilon^2 \left(\phi_u \phi_v 
		\pm \varepsilon \phi_{uv}\right)\s^\pm
\end{align*}
demonstrates that $\s^\pm$ are Moutard lifts. 
\end{proof}

%%%%%%%%%%%%%%%%%%%%%%%%%%%%%%%%%%%%%%%%%%%%%%%%%%%%%%%%%%%%%%%%%%%%%%%%
\subsection{Curvature}
The principal curvatures of $\f$ can be expressed by
\begin{align}\label{eq:PrincipalCurvature}
	k_i = \frac{(\s_i,\q)}{(\s_i,\p)}, 
\end{align}
which is lift-invariant. Denote by $K=k_1 k_2$ and 
$H=\frac{k_1 +k_2}{2}$ the (extrinsic) Gau{ss} and mean curvatures of 
$\f$. We call $\f$ a \emph{linear Weingarten surface} if there is a 
non-trivial triple of constants $a, b, c$ such that the linear 
Weingarten condition
\begin{align*}
	aK + 2bH + c = 0,
\end{align*}
holds. A linear Weingarten surface is called \emph{tubular} if 
$ac-b^2 =0$, that is, if one of the principal curvatures 
is constant.
In what follows we will generally assume that $\f$ is non-tubular. 

We restate a version of the linear Weingarten condition given in 
\cite{burstall2012}. Since the principal curvatures of a linear 
Weingarten surface $\f$ can be written in terms of arbitrary lifts 
$\s_i$ of the curvature sphere congruences as in 
\eqref{eq:PrincipalCurvature}, we may write the linear Weingarten 
condition as 
\begin{align}\label{eq:LinearWeingartenCondition}
	 \begin{pmatrix}
		 (\s_1,\q), (\s_1,\p)
	 \end{pmatrix} W \begin{pmatrix}
		 (\s_2,\q) \cr (\s_2,\p)
	 \end{pmatrix} = 0,~\textrm{with}~W=\begin{pmatrix}
		 a &b \cr b &c
	 \end{pmatrix}.
\end{align}

Then, $W$ induces a non-degenerate bilinear form on $\spn{\p,\q}$, 
and the linear Weingarten condition can be seen as an orthogonality condition
for two vectors with respect to that form. A change of basis in $\spn{\p,\q}$,
given by a $\operatorname{GL}(2)$ matrix $B$ as $(\q,\p)\mapsto(\q,\p)B^{-1}$,
changes the linear Weingarten condition by
\begin{align}\label{eq:BaseChange}
	W \mapsto BWB^t.
\end{align}
In \cite{burstall2012} this is used to prove the following theorem:

\begin{thm}\label{thm:LinearWeingartenInLie}
	Non-tubular linear Weingarten surfaces in space forms are those 
$\Omega$-surfaces $\pounds = \spn{\s^+, \s^-}$ that envelop a (possibly 
complex conjugate) pair of isothermic sphere congruences 
$s^\pm = \spn{\s^\pm}$, each of which takes values in a linear sphere 
complex $\p^\pm$. The plane $\spn{\p^+, \p^-}$ is the plane spanned 
by the point sphere complex $\p$ and the space form vector $\q$.
\end{thm}

\begin{bem}\label{bem:ConservedQuantities}
	The linear sphere complexes $\p^\pm$ are constant conserved quantities 
for the isothermic sphere congruences of the enveloped pair in the sense 
of \cite[Def 2.1]{burstall2012a}, for details see 
\cite{burstall2019a}.
\end{bem}

This characterisation is useful to investigate parallel families of 
linear Weingarten surfaces: Given a space form $\Q$, a 
\emph{parallel transformation} $P$ is a Lie sphere transformation 
that acts solely on $\spn{\p,\q}$. It is well known that parallel 
transformations preserve linear Weingarten surfaces
(see, for instance, \cite[Sect 2.7]{palais1988}),
a fact that follows by straightforward computations in a space form,
or from \Cref{thm:LinearWeingartenInLie}: $P\cdot\pounds = \spn{P\s^+, P\s^-}$ 
is still spanned by isothermic sphere congruences taking values in 
fixed linear sphere complexes $P\p^\pm$. The fact that $P\p^\pm$ span 
$\spn{\p, \q}$ makes $P\pounds$ a linear Weingarten surface in $\Q$ 
again.

Let $\f$ be a linear Weingarten surface in $\Q$ satisfying 
\eqref{eq:LinearWeingartenCondition}. If we interpret the action 
of $P$ in $\spn{\p,\q}$ as a change of basis in that plane, we learn
that the parallel surface $\tilde{\f}\in \Gamma P\pounds$ satisfies the linear 
Weingarten equation
\begin{align*}
	\begin{pmatrix}
		 (\tilde{\s}_1,\q), (\tilde{\s}_1,\p)
	 \end{pmatrix} \tilde{W} \begin{pmatrix}
		 (\tilde{\s}_2,\q)\\ (\tilde{\s}_2,\p)
	 \end{pmatrix} = 0,
\end{align*}
with a matrix
$\tilde{W}= P W P^t$
of the form given in \eqref{eq:BaseChange}.
In \cite{burstall2018}, the authors investigate parallel families 
of discrete linear Weingarten surfaces in this manner. For our 
purpose of classifying channel linear Weingarten surfaces, we formulate
the following proposition, the proof of which is analogous to that of the
discrete version, see \cite[Sect 4.6]{burstall2018}.

\begin{prop}\label{prop:Bonnet}
	Let $\f:\Sigma \to \Q$ be a non-tubular linear Weingarten 
surface satisfying \eqref{eq:LinearWeingartenCondition} in a space 
form of curvature $\kappa=-(\q, \q)\in\{-1,0,+1\}$. 
	\begin{enumerate}
		\item $\kappa=1$: the parallel family of $\f$ contains two 
antipodal pairs of CGC surfaces.
		\item $\kappa=0$: if $c\neq 0$, $\f$ is parallel to a CGC surface with $K\neq 0$.
		\item $\kappa=-1$: if $\left| \frac{a+c}{2}\right|>|b|$, $\f$ 
is parallel to precisely one CGC surface with 
$K\neq 0$. 
	\end{enumerate}
\end{prop}

\begin{bem}
	The parallel CGC surfaces satisfy $K\neq 0$, because surfaces 
with $K=0$ are tubular and tubularity is preserved under parallel 
transformation.
Similarly, parallel transformations in a space form 
with $\kappa \neq 0$ preserve flatness, that is, vanishing of the intrinsic 
Gauss curvature $K+\kappa$.
For $\kappa =-1$, flat surfaces satisfy $\tfrac{a+c}{2}=b=0$, 
hence $\f$ is parallel to a non-flat surface if 
$\left| \tfrac{a+c}{2}\right|>|b|$.
\end{bem}

\begin{bem}
In the remaining cases,
$\kappa=0$ and $c=0$ or
$\kappa=-1$ and $\left|\frac{a+c}{2}\right|\leq|b|\neq0$,
 that are not stated in the proposition,
the parallel family does not contain a CGC surface. 

In the
Euclidean case, parallel families like this consist of a 
minimal surface and its parallel surfaces, all of which have 
constant harmonic mean curvature, that is, constant ratio of
Gauss and mean curvature. It was mentioned in 
\cite{hertrich-jeromin2015} that the only minimal channel surface in
$\R^3$ is the catenoid.

For $\kappa=-1$, each parallel family
not containing a CGC surface contains either one CMC surface or
one surface of constant harmonic mean curvature. Classification of these
surfaces will be the subject of a future publication.
\end{bem}

\section{Rotational surfaces}\label{sec:2}
%%%%%%%%%%%%%%%%%%%%%%%%%%%%%%%%%%%%%%%%%%%%%%%%%%%%%%%%%%%%%%%%%%%%%%%%
The goal of this section is to obtain formulas for the Gauss curvature 
of rotational surfaces. We aim to achieve this in a symmetric way, so 
that our formulas are as independent as possible of the type of space form and 
rotation. 

The isometry group of a space form $\Q$ is the subgroup 
$\operatorname{Iso}_{\p,\q}(3) \subset O(4,2)$ of orthogonal 
transformations that fix the point sphere complex $\p$ and the space 
form vector $\q$. For a $2$-plane $\Pi\perp \spn{\p,\q}$, we call a
$1$-parameter subgroup of isometries $\rho_1$ a \emph{subgroup of 
rotations  (in $\Pi$)}, if it acts as the identity on $\Pi^\perp$. 
Denote by $\sgn{\Pi}$ the signature of the induced metric on $\Pi$.
 We call a subgroup of rotations
\begin{itemize}
	\item \emph{elliptic} if $\sgn{\Pi}=(++)$,
	\item \emph{parabolic} if $\sgn{\Pi}=(+0)$, or
	\item \emph{hyperbolic} if $\sgn{\Pi}=(+-)$. 
\end{itemize}
Note that these are all possible signatures because $\Pi$ is 
perpendicular to the timelike point sphere complex $\p$. The causal 
character of the space form vector $\q$ further restricts the 
possible signatures: for instance, there are no parabolic subgroups 
of rotations in $\S^3$ ($\q$ timelike) but they act as translations 
on $\R^3$ ($\q$ lightlike). Hyperbolic subgroups of rotations only 
exist in $\H^3$.

Let $\{\e_1, \v_1\}$ denote an orthogonal basis of 
$\Pi\subset \R^{4,2}$, where $(\e_1, \e_1)=1$ and
\begin{align*}
	 \kappa_1 := (\v_1, \v_1)
\end{align*}
encodes the signature of $\Pi$. Denote 
\begin{align*}
  \v(\theta) = \rho_1(\theta) \v_1,~\e(\theta) = \rho_1(\theta) \e_1,
\end{align*}
and, since $\rho_1$ is a $1$-parameter subgroup of $O(4,2)$, the $\theta$-derivative $\v'$ of $\v$ is given as
\begin{align*}
  \v'(\theta) = \rho_1(\theta) \v'(0),
\end{align*}
and similarly for $\e$. It is easy to see that
\begin{align*}
  (\v'(0), \v_1 ) = (\e'(0), \e_1) = 0,
\end{align*}
hence, upon changing the parameter $\theta$, we have $\e'(0) = \v_1$ which yields
\begin{align*}
  \v'(\theta) = -\kappa_1 \e(\theta), ~\textrm{and}~ \e'(\theta) = \v(\theta).
\end{align*}
Setting 
\begin{align*}
c_{\lambda}(\psi):=\left\{ \begin{array}{c}
	\cos(\sqrt{\lambda}~\psi) \\ 1 \\ \cosh(\sqrt{-\lambda}~\psi) 
\end{array}\right.\quad \textrm{and}\quad	s_{\lambda}(\psi):=
 \left\{ \begin{array}{cll}
 \tfrac{1}{\sqrt{\lambda}}\sin(\sqrt{\lambda}~\psi) 
	&\textrm{for} &\lambda>0 \\
 \psi &\textrm{for} &\lambda=0 \\
 \tfrac{1}{\sqrt{-\lambda}}\sinh(\sqrt{-\lambda}~\psi) 
	&\textrm{for} &\lambda<0,
\end{array}\right.
\end{align*}
we have
\begin{align*}
  \e(\theta) &= c_{\kappa_1}(\theta) \e_1 + s_{\kappa_1}(\theta) \v_1, \\
  \v(\theta) &= -\kappa_1 s_{\kappa_1}(\theta) \e_1 + 
               c_{\kappa_1}(\theta) \v_1, 
\end{align*}
compare\footnote{Note upon comparing that the author there uses a 
different sign convention for $\kappa_1$.} with 
\cite[Sect 3.7.6]{hertrich-jeromin2003}. 

\begin{bem}
The orbit of $\v_1$ under the action of an elliptic (hyperbolic)
subgroup of rotations is an ellipse (a hyperbola) in $\Pi$. If $\Pi$ is 
isotropic, the parabolic subgroup of rotations $\rho_1$ fixes the lightlike 
$\v_1$. In this case, the orbit of any lightlike $\v_2 \in \spn{\e_1}^\perp$
with $(\v_1,\v_2)=-1$ is a parabola in the affine $2$-plane $\v_2 + \Pi$, given
as
\begin{align*}
  \rho_1(\theta) \v_2 = \v_2 + \theta \e_1 + \tfrac{\theta^2}{2}\v_1.
\end{align*}
In \Cref{fig:Rotations}, 
we visualise the three types of subgroups of 
rotation that exist in the hyperbolic plane: elliptic subgroups move a 
point along a circular orbit that does not intersect the ideal 
boundary; the orbit of the same point under the action of a hyperbolic subgroup 
intersects the ideal boundary in two points. The orbit under a parabolic 
subgroup is a horocircle, i.e., touches the ideal boundary in precisely 
one point (in the half plane we choose the ideal point at infinity, hence
the orbit appears as a straight line).
\end{bem} 

\begin{figure}%
 \centering 
 \subfigure[][Poincar\'{e} disk]{%
  \includegraphics[width=4\columnwidth/10]{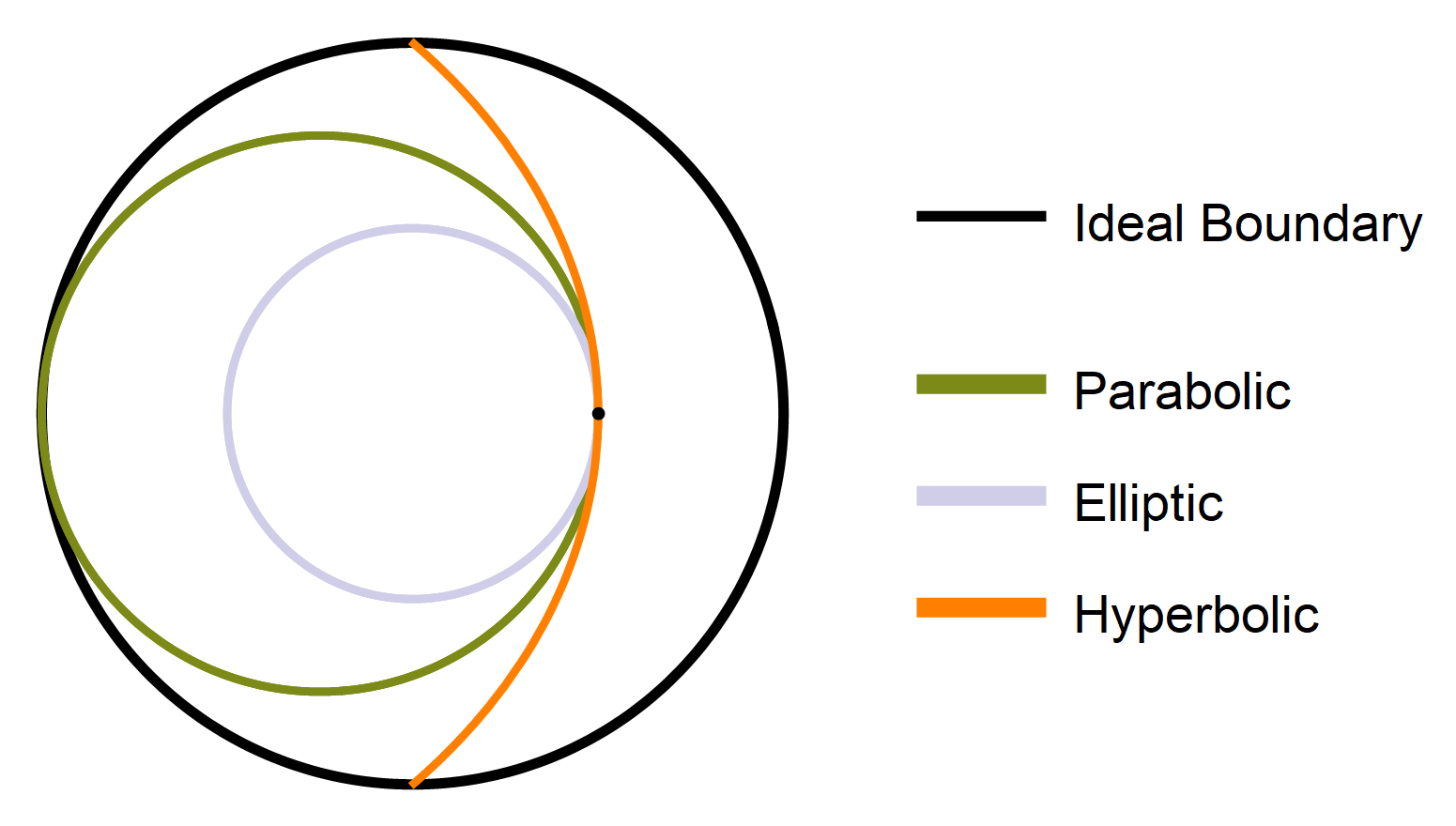}}%
 \hfill
 \subfigure[][Poincar\'{e} half plane]{%
  \includegraphics[width=4\columnwidth/10]{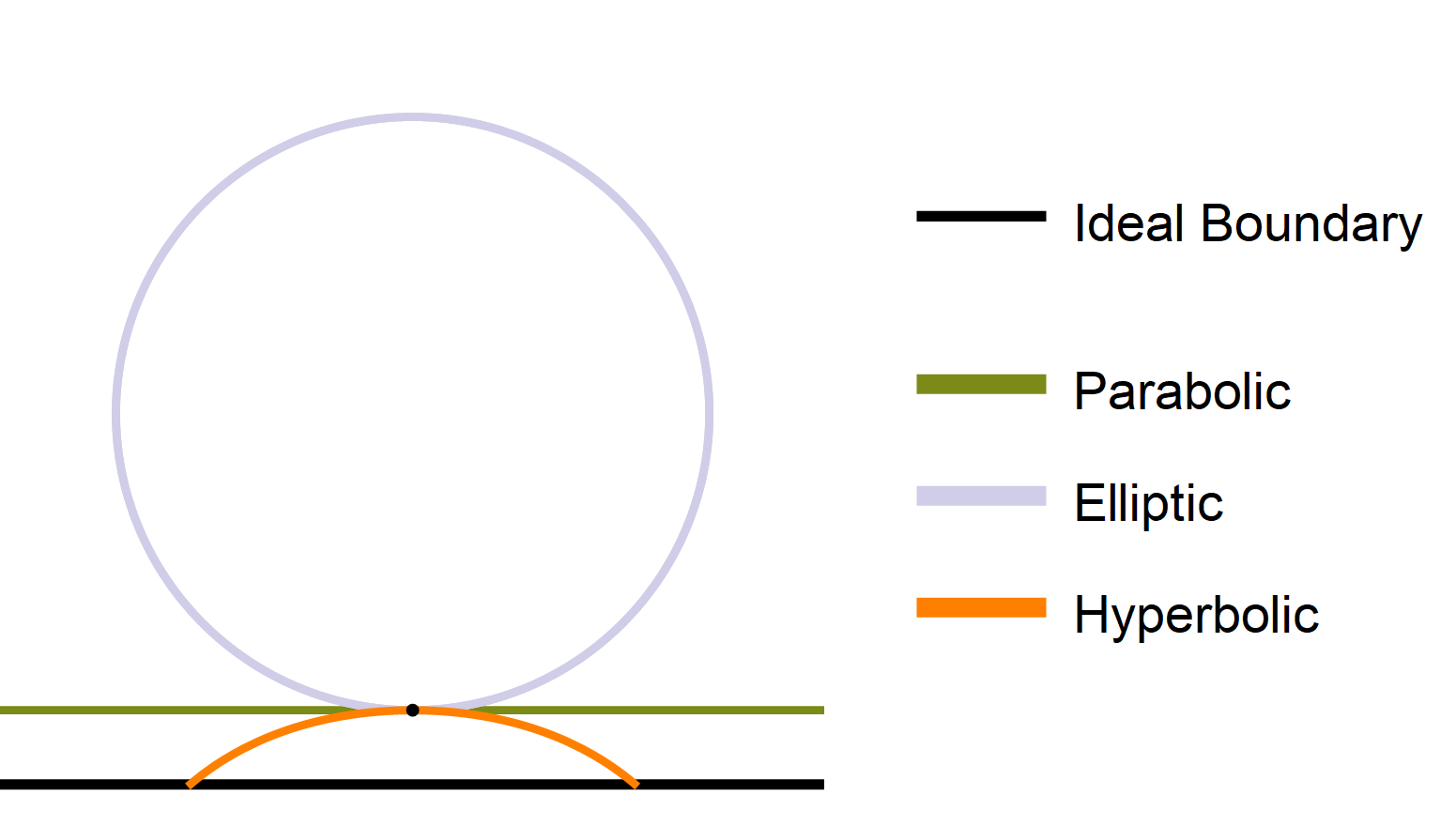}%
 }%
 \caption{Visualisation of subgroups of rotation in the hyperbolic 
  plane: The orbits of one point under all three types of subgroup 
  of rotation are shown in the Poincar\'{e} disk and the half plane 
  model.}
 \label{fig:Rotations}
\end{figure}

We call $\f:\Sigma^2 \to \Q$ a \emph{rotational surface} if there is 
a plane $\Pi$ so that $\f$ is invariant under a subgroup of rotations 
$\rho_1$ in $\Pi$.
We can parametrise
\begin{align*}
	\f(t, \theta) = \ro{\theta}{1}\c(t),
\end{align*}
with a \emph{profile curve} 
$\c=\f(\cdot,0): I \to \spn{\e_1}^\perp\cap \Q$ that is orthogonal to 
$\e_1$. The profile curve takes values in the sphere $\e_1+\p\in \T$, 
and is in this sense \emph{planar}. We employ the parametrisation 
\eqref{eq:MoutardOfF} below of rotational surfaces, see also
\cite[Sect 3.7.6]{hertrich-jeromin2003}.

We first consider the cases where $\v_1$ is non-isotropic, hence, 
$\kappa_1 \neq 0$. Let 
\begin{align*}
  \gamma(t) := \tfrac{\c(t)}{\kappa_1 r(t)},~\textrm{with}~
    r(t) = -\tfrac{(\c(t),\v_1)}{\kappa_1},
\end{align*}
denote the lift of the profile curve such that $(\gamma, \v_1)=-1$. Note that $r(t)$ is the $\v_1$-coordinate function of $\c$. For this 
new lift $\gamma$,  
\begin{align*}
  \gamma + \tfrac{1}{\kappa_1} \v_1 \in \Pi^\perp,
\end{align*}
is fixed by $\rho_1$ and we obtain the following Moutard lift of $\f$:
\begin{align}\label{eq:MoutardOfF}
  \begin{split}
    \m_\f(t, \theta) &= \gamma(t) + 
      \tfrac{1}{\kappa_1}\left(\v_1-\ro{\theta}{1}\v_1\right) \\
      &= \gamma(t)  + s_{\kappa_1}(\theta_1) \e_1+ \tfrac{1}{\kappa_1}\left(1- c_{\kappa_1}(\theta)\right)\v_1.
  \end{split}
\end{align} 

If, on the other hand, $\v_1$ is isotropic, we choose a lightlike 
$\v_2$ perpendicular to $\p, \q$ and $\e_1$ such that 
$(\v_1, \v_2) = -1$. Then we define
\begin{align*}
  \gamma(t) := \tfrac{\c(t)}{r(t)} ~\textrm{with}~r(t) = -(\c(t), \v_1),
\end{align*}
which yields again $(\gamma, \v_1) = -1$, but this time $r(t)$ is
the $\v_2$-coordinate function of $\c(t)$. A Moutard lift of $\f$ is then 
\begin{align}
  \begin{split}
    \m_\f(t, \theta) &= \rho_1(\theta) \v_2 + \gamma(t)- \v_2 \\
      &=\gamma(t) + \theta \e_1 + \tfrac{\theta^2}{2}\v_1.
  \end{split}
\end{align}
Note that this is indeed \eqref{eq:MoutardOfF} as $\tfrac{1}{\kappa_1}\left(1- c_{\kappa_1}(\theta)\right)$ converges to $\tfrac{\theta^2}{2}$ as $\kappa_1$ approaches $0$.

The tangent plane congruence $\t: \Sigma^2 \to \T$ of $\f$ is also 
invariant under the subgroup of rotations in $\Pi$, hence also has a 
Moutard lift $\m_\t(t, \theta) = \ro{\theta}{1} \v_1 + \nu(t)$ with 
a suitable $\nu: I\to \spn{\e_1}^\perp$. As with the surface $\f$, 
there is a function $\tilde{r}$ such that $\m_\t = \t/\tilde{r}$. 
The conditions $(\m_\t, \m_\f) = (\m_\t, (\m_\f)_t)=0$ and 
$(\m_\t, \m_\t)=0$ determine the map $\nu$. 

The curvatures of the surface are the curvatures of its lift $\f$ in 
$\R^{4,2}$ with respect to $\t$, which represents the Gauss map of $\f$ 
as a hypersurface in $\Q$, as was mentioned at the beginning of 
\Cref{sec:Surfaces}. To obtain the Moutard lifts $\m_\f$ and $\m_\t$, 
we rescaled by $r$ and $\tilde{r}$ respectively, hence
\begin{align}\label{eq:ParaGau}
 (\f_t, \f_t)=r^2((\m_\f)_t, (\m_\f)_t),~
 (\f_t,\t_t)=r\tilde{r}((\m_\f)_t,(\m_\t)_t), 
\end{align}
because $\f$ and $\t$ are lightlike. We will denote the speed of the 
profile curve of the Moutard lift $\m_\f$ by $v$, that is,
\begin{align*}
	v^2 =  ((\m_\f)_t, (\m_\f)_t).
\end{align*}
We proceed to provide formulas for the Gauss curvature of $\f$,
to be be used in \Cref{sec:4}. The cases where either $\kappa$ 
or $\kappa_1$ vanish need to be treated separately\footnote{If 
$\kappa = \kappa_1 =0$ then $\f$ is a cylinder in Euclidean space and 
therefore satisfies $K \equiv 0$.}.

%%%%%%%%%%%%%%%%%%%%%%%%%%%%%%%%%%%%%%%%%%%%%%%%%%%%%%%%%%%%%%%%%%%%%%%%
\subsection{Non-isotropic cases}
First assume that $\kappa_1 \neq 0$ in a non-Euclidean space form 
($\kappa \neq 0$). Then $\R^{4,2}$ splits as
\begin{align*}
 \R^{4,2}=\spn{\p,\q} \oplus_\perp \Pi \oplus_\perp \spn{\v_2, \e_2},
\end{align*}
where we choose $\v_2, \e_2\in\R^{4,2}$ orthogonal such that $(\e_2, \e_2)=1$ 
and $(\v_2, \v_2)=\kappa_2$ with\footnote{Note that the causal 
character of $\v_2$ is determined by $\v_1$ and $\q$.} 
$|\kappa_1|=|\kappa_2|$. 

With suitable functions $A, B$ on $I$ and $R=\kappa_1/(\kappa r)$, we write
\begin{align*}
 \m_\f(t, \theta)&=\ro{\theta}{1}\v_1 +R(t)\q +A(t)\v_2 +B(t) \e_2, \\
 \m_\t(t, \theta) &= \ro{\theta}{1} \v_1 + \tilde{R}(t)\p - 
	\tfrac{\kappa_1 B'(t)}{\kappa_2(B'(t)A(t)-B(t)A'(t))}\v_2 + 
	\tfrac{\kappa_1 A'(t)}{B'(t)A(t)-B(t)A'(t)}\e_2,
\end{align*}
with $\tilde{R}=1/\tilde{r}$ satisfying
\begin{align}\label{eq:tildeR}
 \tilde{R}^2 = \frac{\kappa_1 v^2 \kappa R^2}{\kappa_2 (AB'-A'B)^2}.
\end{align}
The principal curvatures of $\m_\f$ with respect to $\m_\t$ are
\begin{align*}
 \widetilde{k}_1=\frac{((\m_\f)_t, (\m_\t)_t)}{((\m_\f)_t,(\m_\f)_t)}=
 \kappa_1\frac{A''B'-A'B''}{v^2(AB'-A'B)},\quad \widetilde{k}_2=
 \frac{((\m_\f)_\theta, (\m_\t)_\theta)}
   {((\m_\f)_\theta,(\m_\f)_\theta)}=
 \frac{\kappa_1}{\kappa_1}=1.
\end{align*}
Using \eqref{eq:ParaGau} and \eqref{eq:tildeR} we obtain
\begin{align*}
 k_1 =\kappa_1~\sqrt{\frac{\kappa\kappa_2}{\kappa_1}}~
 \frac{A''B'-A'B''}{v^3},\quad k_2 = 
 \sqrt{\frac{\kappa\kappa_2}{\kappa_1}}~\frac{AB'-A'B}{v},
\end{align*}
hence the Gauss curvature of $\f$ is
\begin{align*}
	K=\kappa\kappa_2\frac{(AB'-A'B)(A''B'-A'B'')}{v^4}.
\end{align*}

It is beneficial to rewrite this using \emph{polar coordinates}:
define $\Pi_2 := \spn{\v_2, \e_2}$ and denote the rotation in $\Pi_2$ 
as $\rho_2$, analogous to $\rho_1$. Since $\kappa_2 \neq 0$, we have
\begin{align*}
	\ro{\psi}{2}  \v_2= \kappa_2s_{\kappa_2}(\psi)\e_2 + 
c_{\kappa_2}(\psi) \v_2,
\end{align*}
hence we write 
\begin{align*}
	A(t)\v_2 + B(t) \e_2 = D(t) \ro{\psi(t)}{2} \v_2 
\end{align*}
with suitable functions $D$ and $\psi$ of $t$:
if $\v_2$ is spacelike $\Pi_2$ is a Euclidean plane and
 these are the usual polar coordinates;
for timelike $\v_2$ we have $\kappa_2 A^2 + B^2 = \kappa R^2 - \kappa_1< 0$
 showing that $A\v_2 + B\e_2$ is timelike and thus in the orbit of $D\v_2$
 for a suitably chosen function $D$.
In terms of these 
new coordinate functions, we obtain
\begin{align}\label{eq:MoutardOfFPolar}
	\m_\f(t,\theta)=\ro{\theta}{1} \v_1 +
  D(t) \ro{\psi(t)}{2} \v_2 + R(t) \q
\end{align}
and state the following lemma:

\begin{lem}\label{lem:PCPol}
For a rotational surface, given in terms of the Moutard lift 
\eqref{eq:MoutardOfFPolar}, its Gauss curvature is given by
\begin{align}\label{eq:PCPol}
\begin{split}
&K = \kappa\kappa_2 \psi'~\frac{\kappa\kappa_2 D^2 R(\psi'R''-\psi''R')
 -2\psi'(\kappa R R')^2-\psi'(\kappa_1\kappa R'^2 + 
 \kappa_2^3\psi'^2 D^4)}{v^4},\\
&\textrm{with} ~ v^2 = \frac{\kappa_1\kappa R'^2 + 
 \kappa_2^3\psi'^2 D^4}{\kappa_2 D^2}. 
\end{split}
\end{align}
\end{lem}

\begin{bem}
	The rotations $\rho_1$ and $\rho_2$ commute, hence changing the 
initial value of $\psi$ results in an isometry applied to $\f$ that 
does not change the curvature. This is reflected by the fact that 
only derivatives of $\psi$ appear in \eqref{eq:PCPol}.
\end{bem}

%%%%%%%%%%%%%%%%%%%%%%%%%%%%%%%%%%%%%%%%%%%%%%%%%%%%%%%%%%%%%%%%%%%%%%%%
\subsection{Isotropic cases}
We now turn to the cases where either one of $\v_1$ or $\q$ is 
lightlike\footnote{If $\v_1$ and $\q$ are both lightlike, $\f$ is 
a cylinder in Euclidean space and thus $K=0$ and $\f$ is tubular.},
in other words, $\f$ 
is parabolic rotational in $\H^3$ or elliptic rotational in $\R^3$. 
Then, we choose $\v_2$ lightlike such that either $(\v_1, \v_2)=-1$ or 
$(\v_2, \q)=-1$. In either case $\spn{\v_1, \v_2, \p, \q}$ spans a 
$4$-dimensional subspace orthogonal to a spacelike plane. Let 
$\e_2$ denote a vector such that $\{\e_1, \e_2\}$ is an orthonormal 
basis of that plane.

%%%%%%%%%%%%%%%%%%%%%%%%%%%%%%%%%%%%%%%%%%%%%%%%%%%%%%%%%%%%%%%%%%%%%%%%
\subsubsection{Parabolic rotational surfaces in \texorpdfstring{$\H^3$}{H3}}
Assume $\kappa_1 = 0$. With the notations of the previous section we obtain
\begin{align*}
 \m_\f(t, \theta) = \ro{\theta}{1} \v_2 + B(t)\e_2 + A(t)\v_1+ R(t)\q,
\end{align*}
as the Moutard lift of $\f$. The principal curvatures are
\begin{align*}
	k_1 =\sqrt{-\kappa}~\frac{B'A''-B''A'}{v^3}, 
	\quad k_2=\sqrt{-\kappa}~\frac{B'}{v},
\end{align*}
hence the Gauss curvature of $\f$ is
\begin{align*}
	K = -\kappa\frac{B'(B'A''-B''A')}{v^4}.
\end{align*}
As polar coordinates we use 
$\v_2 + B(t)\e_2+A(t)\v_1 = \ro{\psi(t)}{2}\v_2 + D(t)\v_1$, 
where $\rho_2$ denotes the parabolic rotation in $\spn{\v_1, \e_2}$. 
We arrive at the following parametrisation of the Moutard lift $\m_\f$
\begin{align}\label{eq:MoutardOfFPolarPara}
 \m_\f(t,\theta)=\ro{\psi(t)}{2}\ro{\theta}{1}\v_2 +D(t)\v_1 +R(t)\q,
\end{align}
and obtain the following lemma by a straightforward computation:

\begin{lem}\label{lem:GCParaRot}
For a rotational surface, given in terms of the Moutard lift 
\eqref{eq:MoutardOfFPolarPara}, its Gauss curvature is given by
\begin{align}\label{eq:GCParaRot}
\begin{split}
&K=\frac{\psi'}{\kappa}\frac{\kappa R(\psi'R''-\psi''R')-
	\psi'(\psi'^2-\kappa R'^2)}{v^4}, \\
&\textrm{with} ~ v^2:=((\m_\f)_t, (\m_\f)_t)=\psi'^2-\kappa R'^2.
\end{split}
\end{align}
\end{lem}

%%%%%%%%%%%%%%%%%%%%%%%%%%%%%%%%%%%%%%%%%%%%%%%%%%%%%%%%%%%%%%%%%%%%%%%%
\subsubsection{Surfaces of revolution in \texorpdfstring{$\R^3$}{R3}}
Lastly, for $(\q,\q)=0$, we choose $\v_2$ so that $(\q, \v_2) =-1$ 
and we may proceed similarly to the parabolic rotation case. 
Note, however, that in this case the roles of $A$ and $R$ are 
interchanged: the $\v_2$-coefficient of the space form lift $\f$ 
equals $1$, so $A=1/r$ if we parametrise $\m_\f$ as before. Also, 
$(\m_\t,\q)=0$ translates to the $\v_2$-part of the tangent plane 
congruence vanishing. 

The principal curvatures are now
\begin{align*}
k_1=\sqrt{\kappa_1}~\frac{A'B''-B'A''}{v^3},\quad k_2 = 
 \frac{1}{\sqrt{\kappa_1}}~\frac{BA'-AB'}{v}.
\end{align*}
The polar coordinates in $\spn{\v_2, \e_2}$ are given by parabolic 
rotations in $\spn{\q, \e_2}$. This results in 
\begin{align}\label{eq:MoutardOfFPolarEuclid}
\m_\f(t,\theta)=\ro{\theta}{1} \v_1 +R(t)\q +D(t)\ro{\psi(t)}{2}\v_2
\end{align}
as a Moutard lift, hence $B=D\psi$ and $A=D$.
We arrive at the expression for the Gauss 
curvature stated in the following lemma: 

\begin{lem}\label{lem:GCEuclid}
For a rotational surface, given in terms of the Moutard lift 
\eqref{eq:MoutardOfFPolarEuclid}, the Gauss curvature is given by
	\begin{align}\label{eq:GCEuclid}
		\begin{split}
		&K = \frac{D^2\psi'(D(\psi'D''-\psi''D')-2D'^2\psi')}{v^4}, \\
		&\textrm{with}~v^2 = \frac{\kappa_1 D'^2 + \psi'^2D^4}{D^2}.
		\end{split}
	\end{align}
\end{lem}

\section{Channel linear Weingarten surfaces}\label{sec:3}
%%%%%%%%%%%%%%%%%%%%%%%%%%%%%%%%%%%%%%%%%%%%%%%%%%%%%%%%%%%%%%%%%%%%%%%%
Channel surfaces can be characterised by a number of equivalent 
properties (see \cite{blaschke1929}). We give a definition in terms of
Legendre lifts. 

\begin{defi}\label{def:ChannelSurface}
	A surface is called a \emph{channel surface} if its Legendre lift 
$\Lambda$ envelops a $1$-parameter family of spheres $s$. 
\end{defi}

\begin{bem}\label{rem:ChannelCurvatureSphere}
	Since $s$ only depends on one parameter, it is a curvature sphere 
congruence of $\Lambda$. Given curvature parameters $(u,v)$, there 
exists a lift $\s$ of $s$ such that, wlog,
	\begin{align*}
		\s_v = 0.
	\end{align*}
	Take another lift $\tilde{\s}=\lambda \s$ of $s$, then 
$\tilde{\s}_v = \lambda_v \s$. Thereby, a channel surface has a 
curvature sphere congruence $s$ such that $\s_v \in\Gamma s$ for all 
lifts $\s\in\Gamma s$.
\end{bem}

Channel surfaces are examples of $\Omega_0$-surfaces \cite{pember2018}.
We now consider umbilic-free channel $\Omega$-surfaces, that is, 
Legendre maps that are channel and have an additional 
$\Omega$-structure. 

Let $s$ be a one parameter family of (curvature) spheres, enveloped by 
$\Lambda$, and denote the other curvature sphere congruence by $s_1$. As 
we stated in \Cref{prop:OmegaChar}, the curvature sphere congruences of an 
$\Omega$-surface $\Lambda$ admit lifts $\s_1, \s$ such that
\begin{align}\label{eq:ThisSect}
	(\s_1)_u = \phi_u \s, \quad \s_v = \varepsilon^2\phi_v \s_1. 
\end{align}
for curvature line coordinates $(u,v)$ and $\varepsilon\in\{i,1\}$. 
As noted in \Cref{rem:ChannelCurvatureSphere}, all lifts of $s$ 
satisfy $\s_v \in\Gamma s$. Together with \eqref{eq:ThisSect} this implies 
$\phi_v =0$ and $\phi=\phi(u)$ is a function of $u$ only. Furthermore, 
$(\s_1)_{uv} = 0$, so $\s_1$ and $\s$ are Moutard lifts of isothermic 
sphere congruences, as are all maps of the form 
\begin{align}\label{eq:MultibleMoutard}
	\s_1 + U(u) \s.
\end{align}

The original version of Vessiot's theorem \cite{vessiot1926} states 
that any channel isothermic surface in the conformal $3$-sphere
(the space of M\"{o}bius geometry) is either a surface of revolution,
a cylinder, or a cone in a suitably chosen Euclidean subgeometry. We 
will now prove the following Lie geometric version of this theorem. 

\begin{thm}\label{thm:Vessiot}
	A channel $\Omega$-surface is either a Dupin cyclide or
a cone, a cylinder, or a surface of revolution in a 
suitably chosen Euclidean subgeometry,
\end{thm}

\begin{proof}
Let $\Pi$ denote the subspace spanned by $\s$ and its $u$-derivatives. 
Then, because $(\s_1)_{uv}=0$, we get
\begin{align*}
	\Pi(u) \perp \spn{(\s_1)_v, (\s_1)_{vv}}=:\Pi_1(v).
\end{align*}
Since $\Pi_1$ is of dimension $2$ (for non-degeneracy we assume 
$(\s_1)_v$ to be spacelike), $\Pi$ is an at most $4$-dimensional subspace 
of $\R^{4,2}$ for every $u$. 

Assume $\Pi$ is $3$-dimensional at one point $u_0$. Then 
$\Pi(u_0)^\perp$ is a $(2,1)$-plane in which $\s_1$ takes its 
values. Thereby, $s_1$ is a curvature sphere congruence of a 
Dupin cyclide given by the splitting $\Pi(u_0)\oplus \Pi(u_0)^\perp$ 
(see, e.g., \cite[Def 4.4]{pember2018}).

Now assume $\Pi$ is never $3$-dimensional. Then, because $\Pi_1$ only 
depends on $v$ and $\Pi$ only on $u$, $\Pi$ is a constant 
$4$-dimensional space, including at least one timelike direction. 
Choose any point sphere complex $\p \in \Pi$ and consider the 
M\"{o}bius geometry modeled on $\p$. After projection onto 
$\spn{\p}^\perp$, $\Pi_1$ is unchanged, so the M\"{o}bius 
representative of $s$ moves in a $3$-space 
$\pi\subset \spn{\p}^\perp \cong \R^{4,1}$. The following three cases occur 
(for details see \cite[Sect 3.7.7]{hertrich-jeromin2003}).
\begin{itemize}
	\item If $\pi$ does not intersect the light cone 
$\L_\p\subset\spn{\p}^\perp$, then $\L_\p\cap \pi^\perp$ 
contains exactly two points which are then contained in all 
spheres of the enveloped sphere curve $s$. Map one of these 
points to infinity via stereographic projection to see that 
the envelope is a cone. 
	\item If $\pi$ touches $\L_\p$ in one point, then all 
spheres of $s$ touch in precisely this point. Upon a stereographic 
projection, $s$ consists of planes and the envelope becomes a 
cylinder. 
	\item If, finally, $\pi$ intersects $\L_\p$ in two points, 
then $\pi^\perp$ consists of all spheres that share a common 
circle $\gamma$. Accordingly, all spheres in $s$ are perpendicular 
to that circle. Under a stereographic projection that maps $\gamma$ 
to a straight line, $s$ becomes a curvature sphere congruence of a 
surface of revolution.
\end{itemize}%
\end{proof}

\begin{bem}\label{rem:Vessiot}
	If~$\Pi$ is three-dimensional, we can choose $\p$ so that $s$ 
consists of the spheres in an elliptic sphere pencil. Then the 
surface, that is, the point sphere congruence, degenerates to a 
circle and $s_1$ consists of point spheres, which furthers the 
analogy between \Cref{thm:Vessiot} and Vessiot's theorem: In the 
original theorem, it is stated that an isothermic channel surface 
is either rotational or its curvature lines are straight lines 
(i.e., the distinguished circles in the Euclidean subgeometry), 
whereas in \Cref{thm:Vessiot} the channel $\Omega$-surface is either 
isothermic or the curvature spheres are points (the distinguished 
spheres in the M\"{o}bius subgeometry).
\end{bem}

Now we turn to the main objects of interest for this paper: a channel 
linear Weingarten surface in a space form $\Q$ is an $\Omega$-surface 
with a pair of constant conserved quantities $\p^\pm$ spanning 
$\spn{\p, \q}$ (see \Cref{bem:ConservedQuantities}) such that one 
curvature sphere congruence is constant along the corresponding 
curvature direction (recall that we assume all linear Weingarten
surfaces to be non-tubular). Thus, let $\pounds=\spn{\s^+, \s^-}$ 
be linear Weingarten with conserved quantities $\p^\pm$ and let 
$\tilde{\p}\in \spn{\p^+, \p^-}$ be any linear sphere complex. 
Consider the enveloped sphere congruence $\tilde{s}$ given by the lift
\begin{align*}
	\tilde{\s}
	= \s_1 - \frac{(\s_1,\tilde{\p})}{(\s,\tilde{\p})}\s 
	= \frac{1}{(\s,\tilde{\p})}(
	 (\tilde{\p},\s)\s_1 - (\tilde{\p},\s_1)\s),
\end{align*}
where $s=\spn{\s}$ is the enveloped $1$-parameter family of curvature
spheres.
Clearly $(\tilde{\s},\tilde{\p})=0$, hence $\tilde{s}$ takes values 
in the sphere complex $\tilde{\p}$. As we saw in the proof of 
Proposition \ref{prop:OmegaChar}, the isothermic sphere congruences 
enveloped by $\Lambda$ are given as $\s^\pm = \s_1 \pm \varepsilon \s$ 
with lifts of the curvature spheres satisfying \eqref{eq:ThisSect}. Thereby 
we have $(\s_1,\p^\pm) = \mp\varepsilon(\s,\p^\pm)$, hence, 
$(\s_1,\p^\pm)_v=0$. It is now straightforward to see that 
$\tilde{\s}_{uv} =0$, hence $\tilde{\s}$ is a Moutard lift of 
$\tilde{s}$ (given in the form \eqref{eq:MultibleMoutard}). Since 
the plane spanned by $\p^\pm$ is also spanned by the space form 
vector and the point sphere complex of the space form, we have proved 
the following proposition.
 
\begin{prop}\label{prop:ClWIsIsothermic}
	Let $\pounds$ be a non-tubular channel linear Weingarten surface 
in a space form with point sphere complex $\p$ and space form vector 
$\q$. Then every sphere congruence that takes values in a linear 
sphere complex in $\spn{\p, \q}$ is isothermic. 
\end{prop}

\begin{cor}\label{cor:ClWIsIsothermic}
	Non-tubular channel linear Weingarten surfaces in space forms 
(and their tangent plane congruences) are isothermic.
\end{cor}

In \Cref{thm:Vessiot}, we have constructed a space form lift of a channel 
$\Omega$-surface in a suitably chosen Euclidean subgeometry that was 
invariant under a $1$-parameter family of Lie sphere transformations, namely,
translations, dilations, or rotations. For channel linear Weingarten 
surfaces in a space form, we wish to show that these Lie sphere
transformations are always rotations in the respective space form.

\begin{thm}\label{thm:CLWisRot}
	Every non-tubular channel linear Weingarten surface in a space 
form is a rotational surface. 
\end{thm}

\begin{proof}
	First, as stated in \Cref{rem:Vessiot}, if $\Pi$ is 
$3$-dimensional at any one point then $\pounds$ is a Dupin cyclide. 
However, linear Weingarten Dupin cyclides are always tubular. Thus, 
for a non-tubular channel linear Weingarten surface, $\Pi$ is constant 
and $4$-dimensional. 
	
	The sphere curve $s$ is invariant under rotations in the plane 
$\Pi_1$. Since $(\s_1)_v$ (hence $(\s_1)_{vv}$) is perpendicular to 
$\spn{\p^+, \p^-}=\spn{\p,\q}$, we have that $\p, \q \perp \Pi_1$ and thus 
the rotations in $\Pi_1$ are indeed isometries of the considered space form.
\end{proof}

\begin{bem}
 It should be noted that the isometries that appear in the Euclidean 
case may be translations. However, the resulting surfaces, 
cylinders, are tubular. Consequently, non-tubular linear Weingarten 
channel surfaces in that case are always surfaces of revolution, as 
stated in \cite[Thm 2.1]{hertrich-jeromin2015}.
\end{bem}

\section{Constant Gau{ss} curvature}\label{sec:4}
%%%%%%%%%%%%%%%%%%%%%%%%%%%%%%%%%%%%%%%%%%%%%%%%%%%%%%%%%%%%%%%%%%%%%%%%
Let $\f:\Sigma \to \Q$ be a channel surface with constant Gauss 
curvature; then it is a rotational surface by \Cref{thm:CLWisRot}. Let 
$\Pi_1$ denote the corresponding plane of rotations. We computed 
the Gauss curvature $K$ of a rotational surface in \Cref{sec:2}, and 
we will now use the fact that $K$ is constant to obtain a differential equation
that determines the coordinate functions of $\f$. 

As in \Cref{sec:2}, we will also split this section into two 
subsections, attending to the non-isotropic and isotropic cases separately.
In either subsection we will obtain a system of differential equations
by choosing 
an appropriate parametrisation for the profile curve of the rotational surface.

\subsection{Non-isotropic cases}\label{sec:NonIsoCases}
Let $\f$ be a surface of non-parabolic rotation in a non-Euclidean space
 form. Then $\R^{4,2}$ orthogonally splits into the $\p, \q$-plane, 
the rotation plane $\Pi_1 = \spn{\v_1, \e_1}$ and the orthogonal 
complement plane $\Pi_2=\spn{\v_2, \e_2}$. We have 
\begin{align}\label{eq:MultiTable}
	\begin{array}{c|cccccc}
		(~,~) &\p &\q &\v_1 &\e_1 &\v_2 &\e_2 \cr
		\hline
		\p &-1 &0 &0 &0 &0 &0 \cr
		\q &0 &-\kappa &0 &0 &0 &0 \cr
		\v_1 &0 &0 &\kappa_1 &0 &0 &0 \cr
		\e_1 &0 &0 &0 &1 &0 &0 \cr
		\v_2 &0 &0 &0 &0 &\kappa_2 &0 \cr
		\e_2 &0 &0 &0 &0 &0 &1
	\end{array}
\end{align}
where we choose $\v_2$ such that $|\kappa_2|=|\kappa_1|$ (note that 
$\kappa\kappa_1\kappa_2 >0$). We denote the rotation in the plane 
$\Pi_i$ by $\rho_i$ and write
\begin{align}\label{eq:fParam}
	\f(t, \theta) = 
r(t)\ro{\theta}{1}\v_1 +d(t)\ro{\psi(t)}{2} \v_2 +\tfrac{1}{\kappa}\q,
\end{align}
where we have used polar coordinates in $\Pi_2$. Note at this point that 
we are still free to chose the speed of the profile curve.

\begin{prop}\label{prop:NonIsoProp}
	Let $\Q$ be a non-Euclidean space form and let 
$\f:\Sigma^2 \to \Q$ be a CGC $K\neq 0$ 
surface of non-parabolic rotation, parametrised as in 
\eqref{eq:fParam}.
	
	Then, with a suitable choice of speed for the profile curve, 
the coordinate functions $r, \psi$ and $h$ satisfy
\begin{align}
r'^2&=\tfrac{1}{\kappa\kappa_1\kappa_2}\left((1-C)+\kappa_1 Kr^2\right)
\left(C-\kappa_1(K+\kappa)r^2\right), \label{eq:JEE} \\
\psi'&=\tfrac{\kappa_1 K r^2 + 1 - C}{\kappa_2(1-\kappa \kappa_1 r^2)} ,
\label{eq:JEEpsi}\\
d^2&=\tfrac{1-\kappa\kappa_1 r^2}{\kappa\kappa_2}, \label{eq:JEEd}
\end{align}
	where $C$ denotes a suitable constant.
\end{prop}

\begin{proof}
With the Moutard lift 
\begin{align*}
	\m_\f(t, \theta) = 
	\ro{\theta}{1}\v_1 + D(t) \ro{\psi(t)}{2}\v_2 + R(t) \q, 
\end{align*}
we have the following expression of the Gau{ss} curvature in terms 
of $R, D$ and $\psi$
\begin{align}\label{eq:KComplicated}
K=\kappa\kappa_2^3 \psi' D^4~
	\tfrac{\kappa\kappa_2 D^2 R(R''\psi'-R'\psi'')-
	2\psi'(\kappa R R')^2}{(\kappa_1\kappa R'^2 + 
	\kappa_2^3\psi'^2 D^4)^2}-
	\tfrac{\kappa\kappa_2^3\psi'^2 D^4}{\kappa_1\kappa R'^2 + 
	\kappa_2^3\psi'^2 D^4}, 
\end{align}
which is a simple reformulation of \eqref{eq:PCPol} in \Cref{lem:PCPol}. 
We will use the fact that $K\equiv const$ to obtain a differential 
equation for $R=\tfrac{1}{\kappa r}$, which will then yield the differential 
equation for $r$. 

Define 
\begin{align*}
 g(t) :
 = \tfrac{\kappa \kappa_1 R'^2(t)}
   {\kappa\kappa_1 R'^2(t)+\kappa_2^3\psi'^2(t)D^4(t)}.
\end{align*}
Then \eqref{eq:KComplicated} yields
\begin{align*}
 K+\kappa=\tfrac{\kappa_2D^2 R}{2\kappa_1 R'}(\kappa g)'+\kappa g,
\end{align*}
which implies
\begin{align}\label{eq:g}
	\kappa g =\tfrac{\widetilde{C}R^2}{\kappa_2 D^2} + (K+\kappa).
\end{align}
Since $0\leq g \leq 1$, we have that
\begin{align}\label{eq:Bounds} 
	 D^2 \geq \tfrac{1}{\kappa\kappa_2}\left(\kappa_1 K - 
	(\kappa K+\widetilde{C})R^2\right) \geq 0.
\end{align}
Therefore, we obtain the elliptic differential equation 
\begin{align}\label{eq:EllipticR}
	\kappa^2 R'^2
	= \tfrac{1}{\kappa \kappa_1 \kappa_2}
	  \left( \kappa_1 K - (\kappa K+\widetilde{C})R^2 \right) 
          \left( R^2(\widetilde{C}+\kappa K+\kappa^2)
	   - \kappa_1(K+\kappa)
	   \right)
\end{align}
for the function $R$ from \eqref{eq:g} by reparameterising 
the $t$-coordinate so that the speed of the profile curve 
$\m_\f(t,0)$ (given in \Cref{lem:PCPol}) satisfies
\begin{align*}
	v^2 = \tfrac{1}{\kappa\kappa_2}\left(\kappa_1 K - 
(\kappa K+\widetilde{C})R^2\right);
\end{align*}
hence setting
\begin{align}\label{eq:Psi}
	\psi' = 
\tfrac{\kappa_1 K -(\kappa K+\widetilde{C})R^2}{\kappa\kappa_2 
(\kappa R^2 - \kappa_1)}, 
\end{align}
without loss of generality (note that
$\kappa R^2 - \kappa_1= \kappa_2 D^2 \neq 0$). Rewriting this 
and \eqref{eq:EllipticR} in terms of $r$, we obtain 
\eqref{eq:JEE} and \eqref{eq:JEEpsi}, for 
$C = \tfrac{1}{\kappa^2}(\widetilde{C}+\kappa K + \kappa^2)$. 
The equation for $d$ follows from the fact that $\f$ takes values 
in the light cone. This completes the proof. 
\end{proof}

Note that the constant $C$, as defined in the last proof, satisfies
\begin{align}\label{eq:BoundsD}
\tfrac{\kappa\kappa_1}{\kappa_2}(K+\kappa)r^2 \leq 
\tfrac{\kappa C}{\kappa_2}\leq 
\tfrac{\kappa}{\kappa_2}+\tfrac{\kappa\kappa_1}{\kappa_2} K r^2,
\end{align}
because of \eqref{eq:Bounds}. This imposes bounds on the sign of 
$C$ under certain circumstances: say
$\tfrac{\kappa_1 (K + \kappa)}{\kappa_2}$ is non-negative, then 
$\tfrac{\kappa}{\kappa_2}C<0$ would force $r$ to be imaginary. Similarly, 
if $\tfrac{\kappa_1 \kappa K}{\kappa_2}$ were non-positive, $r$ would be
imaginary as soon as $\tfrac{\kappa}{\kappa_2}(C-1) >0$. Hence, in 
our pursuit of real solutions, we investigate three cases:
\begin{itemize}
\item the intrinsic Gauss curvature $K + \kappa$ is positive and 
$\tfrac{\kappa}{\kappa_2}C\geq 0$;
\item the extrinsic Gauss curvature $K$ is negative and 
$\tfrac{\kappa}{\kappa_2}(C-1) \leq 0$;
\item the remaining cases, where $K+\kappa\leq 0 \leq K$ and 
$C\in \R$.
\end{itemize}
The first two of these cases might overlap and the last may not occur 
(for instance if $\kappa =1$). 

\subsubsection{Positive intrinsic Gauss curvature.}\label{punkt:PosCurv}
We start our analysis with the assumption $K+\kappa > 0$. If $C=0$ 
then $r\equiv 0$, hence assume $\tfrac{\kappa C}{\kappa_2}>0$. Then 
\begin{align*}
	y=\sqrt{\tfrac{\kappa_1 (K+\kappa)}{C}}~r
\end{align*}
is a real function and \eqref{eq:JEE} takes the form
\begin{align*}
 y'^2 = \tfrac{K+\kappa - \kappa C}{\kappa\kappa_2}
  \left(1-y^2\right) \left(q^2+p^2 y^2\right)
   ~\textrm{with}~
 p^2 = \tfrac{KC}{K+\kappa - \kappa C}
   ~\textrm{and}~
 p^2+q^2 =1, 
\end{align*}
which leads to 
\begin{align*}
	y(s) = \jac{cn}{p}\left(
	\sqrt{\tfrac{K+\kappa -\kappa C}{\kappa \kappa_2}}~s\right), 
\end{align*}
and subsequently
\begin{align}\label{eq:rSolGaussPos}
	r(s) = \sqrt{\tfrac{C}{\kappa_1(K +\kappa)}}~\jac{cn}{p}
\left(\sqrt{\tfrac{K+\kappa -\kappa C}{\kappa \kappa_2}}~s\right),
\end{align}
where $\jac{cn}{p}$ denotes a Jacobi elliptic function with modulus 
$p$. If $p\notin [0,1]$, a Jacobi transformation may be applied to 
express $r$ by another Jacobi elliptic function with modulus
$\tilde p\in[0,1]$.
For an overview of the Jacobi elliptic 
functions and their Jacobi transformations, see \Cref{app:JacobiFunctions} or 
\Cref{sec:5} for examples.

To obtain $\psi$ we need to integrate \eqref{eq:JEEpsi}, which may be 
rewritten as
\begin{align*}
 \psi'(s) = -\tfrac{K}{\kappa\kappa_2} + 
\tfrac{K+\kappa}{\kappa \kappa_2} 
\left(1+\tfrac{\kappa C}{K+\kappa -\kappa C}
\jac{sn}{p}^2\left(
\sqrt{\tfrac{K+\kappa -\kappa C}{\kappa\kappa_2}}~s\right)\right)^{-1}.
\end{align*}
We can then express $\psi$ as
\begin{align}\label{eq:psiSolGaussPos}
 \psi(s) = -\tfrac{K}{\kappa\kappa_2}~s + 
\tfrac{K+\kappa}{\kappa\kappa_2}
\sqrt{\tfrac{\kappa\kappa_2}{K+\kappa -\kappa C}}~
\Pii{-\tfrac{\kappa C}{K+\kappa -\kappa C}}{p}{
\sqrt{\tfrac{K+\kappa -\kappa C}{\kappa\kappa_2}}~s}
\end{align}
where $\Pii{k}{p}{s}$ denotes the incomplete elliptic integral of 
the third kind\footnote{Defining the incomplete integral of third 
kind as
\begin{align*}
\Pi(k;s,p)=
\int_0^{s}\tfrac{1}{1-k\sin^2(u)}\tfrac{du}{\sqrt{1-p^2\sin^2(u)}},
\end{align*}
as is often done, we obtain the relationship
\begin{align*}
	\Pii{k}{p}{s} = \Pi(k; \jac{am}{p}(s), p).
\end{align*}
} with modulus $p$ and parameter $k$ as defined in 
\cite[Sect 17.2]{abramowitz1972}, that is, 
\begin{align*}
	\Pii{k}{p}{s} = \int_0^{s}\tfrac{du}{1-k \jac{sn}{p}^2(u)}.
\end{align*}
For solutions \eqref{eq:rSolGaussPos} with $p \notin [0,1]$,
transformations of $\Pi$ can be applied (see 
\Cref{app:JacobiFunctions}).

\subsubsection{Negative extrinsic Gauss curvature.}\label{punkt:NegCurv}
Assume now that $K<0$. Then $y=\sqrt{\tfrac{\kappa_1 K}{C-1}}r$ is 
real and satisfies
\begin{align*}
 y'^2 = \tfrac{\kappa C -(K+\kappa)}{\kappa\kappa_2}
  \left(1-y^2\right) \left(q^2 + p^2 y^2\right)
  ~\textrm{with}~
 p^2 = \tfrac{(K+\kappa)(C-1)}{\kappa C -(K+\kappa)}
  ~\textrm{and}~
 p^2 + q^2 =1. 
\end{align*}
Hence we learn that
\begin{align}\label{eq:rSolGaussNeg}
r(s) = \sqrt{\tfrac{C-1}{\kappa_1 K}}~\jac{cn}{p}
\left(\sqrt{\tfrac{\kappa C- (K+\kappa)}{\kappa\kappa_2}}~s\right). 
\end{align}

Again we integrate \eqref{eq:Psi} by writing
\begin{align*}
	\psi'(s) = -\tfrac{K}{\kappa\kappa_2} + 
\tfrac{K}{\kappa \kappa_2} 
\left(1-\tfrac{\kappa(C-1)}{\kappa C-(K+\kappa)}
\jac{sn}{p}^2\left(\sqrt{\tfrac{\kappa C -(K+\kappa)}{\kappa\kappa_2}}
~s\right)\right)^{-1},
\end{align*}
and $\psi$ is given by an incomplete integral of the third kind
\begin{align}\label{eq:psiSolGaussNeg}
	\psi(s) =-\tfrac{K}{\kappa\kappa_2}~s + 
\tfrac{K}{\kappa\kappa_2}
\sqrt{\tfrac{\kappa \kappa_2}{\kappa C-(K+\kappa)}}
~\Pii{
\tfrac{\kappa (C-1)}{
\kappa C-(K+\kappa)}}{p}{
\sqrt{\tfrac{\kappa C- (K+\kappa)}{\kappa\kappa_2}}~s}).
\end{align}

\subsubsection{Remaining cases.}\label{punkt:MixCurv}
Finally, assume that $K+\kappa \leq 0 \leq K$, hence $C$ a priori ranges 
over $\R$. Note that, for particular values of $C$, the solutions given in
the previous two sections suffice: 
\begin{itemize}[-]
	\item if $\tfrac{C}{\kappa_1}$ is negative, we can still employ 
the solution given in \eqref{eq:rSolGaussPos};
	\item if $\tfrac{C-1}{\kappa_1}$ is positive, the solution given 
in \eqref{eq:rSolGaussNeg} is still real. 
\end{itemize}

However, if $C$ satisfies
$0 \leq \tfrac{C}{\kappa_1} \leq \tfrac{1}{\kappa_1}$, both solutions 
are imaginary. We consider this case next. The function 
$y=\sqrt{\tfrac{\kappa_1 K}{C-1}}~r$ now satisfies the differential 
equation
\begin{align*}
 y'^2 = \tfrac{(-K)C}{\kappa\kappa_2}(1-y^2)(1-p^2y^2)
  ~\textrm{with}~
 p^2 = \tfrac{(K+\kappa)(C-1)}{KC},
\end{align*}
which is solved by the Jacobi $\jac{sn}{}$ function 
\begin{align*}
r(s) = \sqrt{\tfrac{C-1}{\kappa_1 K}}~\jac{sn}{p}
\left(\sqrt{\tfrac{(-K)C}{\kappa\kappa_2}}~s\right).
\end{align*}
To see that this is a real solution, note that the two square roots 
are purely imaginary. Since $\jac{sn}{p}(i x)$ is imaginary, we find
that $r$ is real (see \Cref{app:JacobiFunctions}): 
\begin{align}
r(s)&= \sqrt{\tfrac{C-1}{\kappa_1 K}}~\jac{sn}{p}
	\left(\sqrt{\tfrac{KC}{\kappa\kappa_2}}~is\right)  \notag \\
	&= -\sqrt{\tfrac{1-C}{\kappa_1 K}}~\jac{sc}{q}
 \left(\sqrt{\tfrac{KC}{\kappa\kappa_2}}~s\right), \label{eq:SNSol}
\end{align}
wherein $q$ is defined by $p^2+q^2=1$, as usual. Note that this solution is 
real under our assumption 
$0 \leq \tfrac{C}{\kappa_1} \leq \tfrac{1}{\kappa_1}$, hence 
in all cases where neither \eqref{eq:rSolGaussPos} nor 
\eqref{eq:rSolGaussNeg} yield a solution.

As in the previous subsections, we may rewrite \eqref{eq:Psi} as
\begin{align*}
	\psi'(s) = \tfrac{1-C}{\kappa_2} 
\left(1-\tfrac{K+\kappa-\kappa C}{K}\jac{sn}{q}^2
\left(\sqrt{\tfrac{KC}{\kappa\kappa_2}}~s\right) \right)^{-1},
\end{align*}
to obtain
\begin{align}\label{eq:SNSolPsi}
	\psi(s) = 
\tfrac{1-C}{\kappa_2}\sqrt{\tfrac{\kappa\kappa_2}{KC}}~
\Pii{\tfrac{K+\kappa-\kappa C}{K}}{q}{
\sqrt{\tfrac{KC}{\kappa\kappa_2}}~s}.
\end{align}

\subsection{Isotropic scenarios}
The isotropic cases are those where $\kappa$ or $\kappa_1$ vanishes,
that is, the cases
of parabolic rotational surfaces in hyperbolic space forms and
of elliptic rotational surfaces in Euclidean space. 

\subsubsection{Parabolic rotational surfaces in \texorpdfstring{$\H^3$}{H3}}\label{sec:Parb}
First, we investigate parabolic rotational surfaces in hyperbolic 
space: assume that $\v_1, \v_2$ are lightlike with 
$(\v_1, \v_2)=-1$ and that $\e_1, \e_2$ are unit length orthogonal vectors perpendicular to 
$\spn{\v_1, \v_2}$ and $\spn{\p,\q}$. Consider 
the parabolic rotations $\rho_i$ in $\spn{\v_1, \e_i}$, which are
given by
\begin{align*}
\ro{\theta}{i}:~ (\v_2, \e_i, \v_1) \mapsto (\v_2 + \theta \e_i + 
\tfrac{\theta^2}{2}\v_1, \e_i + \theta \v_1, \v_1).
\end{align*} 
Then, in polar coordinates, a parabolic rotational surface $\f$ is 
given by
\begin{align}\label{eq:fParamP}
	\f(t, \theta)= \tfrac{1}{\kappa}\q + r(t) \ro{\psi(t)}{2}
\ro{\theta}{1} \v_2 + d(t) \v_1. 
\end{align}

\begin{prop}
	Let $\Q$ be a hyperbolic space form with spacelike $\q$. Let 
$\f:\Sigma^2 \to \Q$ be a CGC $K\neq 0$ parabolic rotational surface, parametrised 
as in \eqref{eq:fParamP}. 
	
	Then, with a suitable choice of speed for the profile curve, the 
coordinate functions $r, \psi$ and $h$ satisfy
\begin{align}
r'^2 &= \left(-\tfrac{1}{\kappa}\right)
	\left(C-Kr^2\right)\left((K+\kappa)r^2-C\right), \label{eq:JEEP}\\
\psi'&= \left(-\tfrac{1}{\kappa}\right)
	\left(K -\tfrac{C}{r^2}\right), \label{eq:JEEpsiP} \\
d&=\left(-\tfrac{1}{\kappa}\right)\tfrac{1}{2r},
\end{align}
	where $C$ denotes a suitable constant.
\end{prop}

\begin{proof}
Using the Moutard lift 
\begin{align*}
\m_\f(t,\theta)=R(t)\q+\ro{\psi(t)}{2}\ro{\theta}{1}\v_1+D(t)\v_1,
\end{align*}
we have seen in \eqref{eq:GCParaRot} that 
\begin{align*}
K =\kappa^2\psi'R\tfrac{\psi'R''-\psi''R'}{(\psi'^2 - \kappa R'^2)^2}
+(-\kappa)\tfrac{\psi'^2}{\psi'^2 - \kappa R'^2}.
\end{align*}
Similarly to the proof of \Cref{prop:NonIsoProp}, we set
\begin{align*}
	g := \tfrac{-\kappa R'^2}{\psi'^2 - \kappa R'^2}
\end{align*}
to obtain
\begin{align*}
	g = (1+\tfrac{1}{\kappa} K)+ \widetilde{C}R^2.
\end{align*}
This time, $0 \leq g \leq 1$ implies
\begin{align*}
	\tfrac{1}{\kappa} K + \widetilde{C}R^2 \leq 0,
\end{align*}
hence we derive the following differential equation for 
$r=\tfrac{1}{\kappa R}$:
\begin{align*}
r'^2 = \left(-\tfrac{1}{\kappa}\right)
\left(C-Kr^2\right)\left((K+\kappa)r^2 - C\right),
\end{align*}
by setting $C=-\tfrac{\widetilde{C}}{\kappa}$ and choosing 
\begin{align}\label{eq:ParaPsi}
	\psi' = v^2= \left(-\tfrac{1}{\kappa}\right)
\left(K -\tfrac{C}{r^2}\right).
\end{align}
Naturally, $d$ is given by $(\f, \f)=0$, and we are done.
\end{proof}

Further, $C$ satisfies (because the function $g$ in the previous 
proof satisfies $0\leq g\leq 1$)
\begin{align*}
	r^2(K+\kappa) \leq C \leq K r^2.
\end{align*}
Accordingly, for $r$ to be a real function,  
\begin{itemize}
	\item $K+\kappa>0$ dictates $C \geq 0$, 
	\item $K<0$ dictates $C \leq 0$, whereas
	\item $K+\kappa \leq 0 \leq K$ does not restrict $C$. 
\end{itemize}
As before, we define functions 
$y_1 = \sqrt{\tfrac{K+\kappa}{C}}r$ and $y_2 = \sqrt{\tfrac{K}{C}} r$
and use them to obtain the solutions
\begin{align}\label{eq:ParabSolR}
\begin{split}
r_1(s) &= \sqrt{\tfrac{C}{K+\kappa}}\jac{cn}{p}
 \left(\sqrt{C}~s\right),~\textrm{with}~p^2 = -\tfrac{K}{\kappa}, \\
r_2(s) &=\sqrt{\tfrac{C}{K}}~\jac{cn}{p}
 \left(\sqrt{-C}~s\right),~\textrm{with}~p^2=\tfrac{K+\kappa}{\kappa}.
\end{split}
\end{align}
These solutions are real as soon as their coefficients are real (i.e., 
as soon as the respective $y$-function is real). This takes care of 
all cases: in the first two cases, the respective function provides
a solution.
If $K+\kappa \leq 0 \leq K$, depending on the sign of $C$, one of $r_1, r_2$ is 
real and hence a feasible solution. By \eqref{eq:JEEpsiP},
$\psi$ is determined as before:
\begin{align}\label{eq:ParabSolPsi}
\begin{split}
\psi_1(s) &= -\tfrac{K}{\kappa}~s+
\tfrac{K+\kappa}{\kappa\sqrt{C}}~\Pii{1}{p}{\sqrt{C}~s}, \\ 
\psi_2(s) &= -\tfrac{K}{\kappa}~s-
\tfrac{K}{\kappa}\sqrt{-\tfrac{\kappa}{(K+\kappa)C}}~
\Pii{p^2}{p}{\sqrt{-\tfrac{(K+\kappa)C}{\kappa}}~s},
\end{split}
\end{align}
with an appropriate transformation applied if $p \notin [0,1]$. 

\subsubsection{Surfaces of revolution in \texorpdfstring{$\R^3$}{R3}}
Finally, we consider an elliptic rotational surface in Euclidean 
$3$-space, i.e., a common surface of revolution. This case was fully analysed
in \cite{hertrich-jeromin2015}, thus we just show how our setting leads
to the same differential equations.
Equation \eqref{eq:GCEuclid} can be rewritten as
\begin{align*}
	\tfrac{2D'}{D^3}\kappa_1 K = g'
\end{align*}
with
\begin{align*}
	g := \tfrac{\kappa_1 D'^2}{\kappa_1 D'^2 + \psi'^2 D^4}.
\end{align*}
Integration of the equation then yields
\begin{align*}
D'^2 =\left((1-C)D^2+\kappa_1 K\right)\left(C D^2-\kappa_1 K\right),
\end{align*}
after reparameterising the profile curve so that
\begin{align*}
	v^2 = \kappa_1\left((1-C)D^2+\kappa_1 K\right).
\end{align*}
In this case, we obtain the space form lift of $\m_\f$ via rescaling
by $r=1/D$, and arrive (assuming $\kappa_1 =1$) at the standard 
parametrisation in a suitable orthonormal basis $\{e_1, e_2, e_3\}$ 
of $\R^3$:
\begin{align*}
f(t,\theta) = r(t)\cos \theta e_1 + r(t) \sin \theta e_2 + \psi(t)e_3,
\end{align*}
with $r$ satisfying
\begin{align*}
	r'^2 = \left((1-C)+Kr^2\right)(C-Kr^2),
\end{align*}
which is precisely Equation (6) of \cite{hertrich-jeromin2015}.
Thus, for a solution, we refer the interested reader to this 
publication.

\section{Rotational CGC surfaces in \texorpdfstring{$\S^3$}{S3} and %
\texorpdfstring{$\H^3$}{H3}}\label{sec:5}
%%%%%%%%%%%%%%%%%%%%%%%%%%%%%%%%%%%%%%%%%%%%%%%%%%%%%%%%%%%%%%%%%%%%%%%%
In this section we discuss rotational CGC surfaces in $\S^3$ and 
$\H^3$, which are modelled in the classical 
way: consider $\S^3$ as the unit sphere in Euclidean $\R^4$ and 
$\H^3$ as the upper sheet of the two sheeted hyperboloid in 
$\R^{3,1}$ with the Lorentz metric. In Minkowski space $\R^{3,1}$, 
we choose the basis according to the type of rotation: for surfaces 
of hyperbolic and elliptic rotation, we choose an \emph{orthonormal 
basis}, so that the metric takes the form
\begin{align*}
	(x,y) = -x_0y_0 + x_1y_1 + x_2y_2 + x_3y_3; 
\end{align*}
when considering surfaces of parabolic rotation we use a 
\emph{pseudo-orthonormal basis} and the metric is computed as
\begin{align*}
	(x,y) = -x_0y_1 - x_1y_0 + x_2y_2 + x_3y_3.
\end{align*}
These choices result in the parametrisations for surfaces of 
elliptic and hyperbolic (or parabolic) rotation given in
\eqref{eq:fParam} (or \eqref{eq:fParamP}) for $|\kappa|=|\kappa_1|=1$ 
(or $\kappa = -1$) with solutions to the equations \eqref{eq:JEE} and 
\eqref{eq:JEEpsi} (or \eqref{eq:JEEP} and \eqref{eq:JEEpsiP}).
Three 
different cases emerged in the solution of the differential equation 
satisfied by $r$. In this section, however, we want to consider the 
(slightly different) following three cases:
\begin{itemize}
	\item $K$ and $K+\kappa$ are positive,
	\item $K$ and $K+\kappa$ are negative, and
	\item $K$ and $K+\kappa$ have different signs.
\end{itemize}
These arise (algebraically) from the fact that at $K=0$ and 
$K+\kappa=0$ the polynomial on the right side of \eqref{eq:JEE} 
degenerates to degree $2$. This is another instance of the 
bifurcation that appears in the construction of constant curvature 
surfaces in space forms, see \cite[Sect 3.2]{dursun2020}, 
\cite{ferus1996} or \cite{lopez2009a} (\cite{hertrich-jeromin2015} 
for the Euclidean case). 

The boundary case $K=0$ yields tubular surfaces and is not 
considered in these notes. Surfaces with $K+\kappa = 0$ are
\emph{(intrinsically) flat}. The class of intrinsically flat 
surfaces has been widely studied and many examples are known 
(e.g., the Clifford tori in $\S^3$ or the peach front in $\H^3$). 
We start by considering this case.

\subsection{Intrinsically flat surfaces in
 \texorpdfstring{$\H^3$}{H3} and \texorpdfstring{$\S^3$}{S3}}
Flat surfaces of non-parabolic rotation are (mostly) a boundary case 
of the solution \eqref{eq:rSolGaussNeg} to \eqref{eq:JEE} given in 
Subsection \ref{punkt:NegCurv}:
\begin{align*}
 r(s) &=  \sqrt{\tfrac{(1-C)}{\kappa_2}} ~
  \jac{cn}{p}\left(\sqrt{\tfrac{C}{\kappa_2}}~s\right)
  ~\textrm{with}~
  p^2 = \tfrac{(K+\kappa)(C-1)}{\kappa C-(K+\kappa)} \\
      &=\sqrt{\tfrac{(1-C)}{\kappa_2}} ~ 
  \cos\left(\sqrt{\tfrac{C}{\kappa_2}}~s\right),
\end{align*}
where we used $\jac{cn}{0}=\cos$.
Note that $\kappa_2 = \kappa\kappa_1$ from \eqref{eq:MultiTable}.
Further, from \eqref{eq:JEEpsi} we get
\begin{align*}
	\psi(s) &= s - \sqrt{\tfrac{1}{\kappa_2}}~ 
\operatorname{arctan}\left(\sqrt{\tfrac{1}{C}}~
\operatorname{tan}(\sqrt{\kappa_2 C}~s)\right).
\end{align*}
From \eqref{eq:BoundsD} we then learn that
\begin{align*}
	0 \leq \kappa_1 C \leq \kappa_1 - \kappa\,r^2,
\end{align*}
and obtain the following results:
\begin{itemize}[-]
\item For surfaces of elliptic rotation in $S^3$
 ($\kappa = \kappa_1 = 1$), $C \in [0,1]$ which implies
 that $r$ and $\psi$ are real functions. For $C=1$, the surface 
 degenerates as $r=0$. For $C=0$, however, the differential 
 equation satisfied by $r$ degenerates to $r'=0$, the solution 
 to which are the Clifford tori.  
\item For surfaces of hyperbolic rotation in $\H^3$ 
 ($\kappa = \kappa_1 =-1$), $C\in[-1,0]$ is negative and $r$ as well 
 as $\psi$ are again real (hyperbolic) functions.
 These surfaces are called \emph{peach fronts}
 (see \cite{kokubu2005b}, \cite{kokubu2003}). 
\item For surfaces of elliptic rotation in $\H^3$
 ($\kappa = -\kappa_1=-1$), $C$ is positive and 
 \begin{align*}
	r(s) = \sqrt{C-1} ~ \cosh\left(\sqrt{C}~s\right)
 \end{align*}
 is real for $C>1$, which yields \emph{snowman fronts},
 but
 is imaginary for $C<1$;
 in this case, we obtain as solution the real form 
 of the solution \eqref{eq:SNSol}:
 \begin{align*}
	r(s) &= \sqrt{1-C}~\sinh\left(\sqrt{C}~s\right),
 \end{align*}
 which yields an \emph{hourglass front} 
 (see \cite{kokubu2005b}, \cite{kokubu2003}).
\end{itemize}

\noindent For flat fronts of parabolic rotation in $\H^3$, on the 
other hand, we employ the parametrisation \eqref{eq:fParamP}
\begin{align*}
 \f(t, \theta)= -\q +r(t)\ro{\psi(t)}{2}\ro{\theta}{1}\v_1 +d(t)\v_2,
\end{align*}
with the coordinate function
\begin{align*}
 r(s) =&=\sqrt{\tfrac{C}{K}}~\jac{cn}{p}
 \left(\sqrt{-C}~s\right),~\textrm{with}~p^2=1-K.
\end{align*}
We see that for $K-1=0$,
\begin{align*}
		r(s) &=\sqrt{C}~\cosh(\sqrt{C}~s) 
\end{align*}
and integrate \eqref{eq:JEEpsiP} to obtain
\begin{align*}
		\psi(s) &=s-\tfrac{\tanh(\sqrt{C}~s)}{\sqrt{C}},
\end{align*}
for any $C>0$. 

\subsection{Rotational CGC surfaces in \texorpdfstring{$\S^3$}{S3}} 	
	For rotational surfaces in $\S^3$, we assume 
$\kappa = \kappa_1 = \kappa_2 = 1$ in \eqref{eq:MultiTable}. 
In this case, \eqref{eq:JEEd} reduces to $d^2=1-r^2>0$, which can be used to 
refine \eqref{eq:BoundsD} to yield the following cases:
	\begin{itemize}
		\item $K > 0$, implying $0\leq C \leq K+1$, 
		\item $K+1 < 0$, implying $K+1\leq C \leq 1$, and
		\item $K \in (-1,0)$, implying $0\leq C \leq 1$. 
	\end{itemize}
	
\Cref{thm:Sphere} below, combined with \Cref{prop:Bonnet}, then provides a 
complete classification of non-tubular channel linear Weingarten 
surfaces in $\S^3$. 

\begin{bem}
	In \cite{arroyo2019} the authors consider Delaunay type surfaces, 
that is, rotational CMC surfaces in $\S^3$. Such surfaces are parallel 
to rotational surfaces of constant positive Gauss curvature.
Thus explicit parametrisations of these Delaunay surfaces in a sphere
can be obtained by applying a suitable parallel transformation to
the parametrisations given in \Cref{thm:Sphere}.
\end{bem}
	
\begin{thm}\label{thm:Sphere}
	Every rotational constant Gauss curvature $K\neq 0$ 
surface in $\S^3\subset \R^4$ is given by
	\begin{align*}
		(s,\theta) \mapsto \left(r(s)\cos\theta, r(s)\sin\theta, 
\sqrt{1-r^2(s)}~\cos\psi(s), \sqrt{1-r^2(s)}~\sin\psi(t)\right), 
	\end{align*}
	where $r, \psi$ are listed in \Cref{tab:SphericalCase}.
\end{thm}

\begin{proof}
	According to Subsections \ref{punkt:PosCurv} and \ref{punkt:NegCurv}, 
we obtain as potential solutions to \eqref{eq:JEE}
\begin{align*}
 r_1(s) &= \sqrt{\tfrac{C}{K+1}}~\jac{cn}{p}(\sqrt{K+1-C}~s)
  ~\textrm{with}~ p^2 = \tfrac{KC}{K+1-C},
  ~\textrm{or} \\
 r_2(s) &= \sqrt{\tfrac{C-1}{K}}~\jac{cn}{p}(\sqrt{C-(K+1)}~s)
  ~\textrm{with}~ p^2 = \tfrac{(K+1)(C-1)}{C-(K+1)}.
\end{align*}

For $K>0$, the function $r_1$ is real and $p^2\in [0,\infty)$,
while $r_2$ has an imaginary argument;
this can be rectified by means of the Jacobi transformations
given in \Cref{app:JacobiFunctions}.
We obtain a bifurcation of 
the solution space into $\jac{dn}{}$- and $\jac{cn}{}$-type solutions. 

For $K+1<0$, where $r_2$ is real and $p^2\in[0,\infty)$,
a similar analysis applies.

For $K\in(-1,0)$, both functions, $r_1$ and $r_2$, are real even though
the arguments and $p$ may be imaginary
(see \Cref{app:JacobiFunctions} again).
Writing, for example, $r_1$ in its real form yields
\begin{align*}
 r_1(s) &= \sqrt{\tfrac{C}{K+1}}~\jac{cd}{p}(\sqrt{(K+1)(1-C)}~s)
  ~\textrm{with}~ p^2 = \tfrac{(K+1)(C-1)}{KC}
  ~\textrm{for}~ C<K+1 ~\textrm{and}\\
 r_1(s) &= \sqrt{\tfrac{C}{K+1}}~\jac{dc}{p}(\sqrt{(-K)C}~s)
  ~\textrm{with}~p^2=\tfrac{(K+1)(C-1)}{KC}
  ~\textrm{when}~C>K+1.
\end{align*} 
	However, we still have the restriction $r^2<1$, which shows that 
the second form is not a feasible solution. Accordingly, we need to
use the real form of $r_2$ in the case $C>K+1$. 
	
	\Cref{tab:SphericalCase} summarises the solutions, written in
terms of the modulus $p$ instead of the integration constant $C$. 
The function $\psi$ is obtained as in Subsection \ref{punkt:PosCurv}
for each case. 
\end{proof}

\begin{table}[ht]
\centering
\begin{tabular}{l|ll}
$K<-1$ 
&$r(s)=\sqrt{\tfrac{p^2}{p^2-(K+1)}}~\jac{cn}{p}\left(\FAC~s\right)$ 
&$\psi(s)=\tfrac{K}{\FAC}~\Pii{\tfrac{p^2}{K+1}}{p}{\FAC~s}-Ks$ \cr
&\quad $\FAC=\sqrt{\tfrac{K(K+1)}{p^2-(K+1)}}$, &$p\in [0,1]$ \cr
&$r(s)=\sqrt{\tfrac{1}{1-(K+1)p^2}}~\jac{dn}{p}\left(\FAC~s\right)$ 
&$\psi(s)=\tfrac{K}{\FAC}~\Pii{\tfrac{1}{K+1}}{p}{\FAC~s}-Ks$ \cr
&\quad $\FAC=\sqrt{\tfrac{K(K+1)}{1-(K+1)p^2}}$ &$p\in [0,1]$ \cr
\hline
$K=-1$ 
&$r(s)=\sqrt{1-p^2}~\cos~(ps)$ 
&$\psi(s)=s-\operatorname{arctan}\left(\tfrac{\tan(ps)}{p}\right)$ \cr
&\quad $p\in (0,1)$ &\cr
\hline 
$K\in(-1,0)$ 
&$r(s)=\sqrt{\tfrac{p^2}{(K+1)p^2-K}}~\jac{cd}{p}\left(\FAC~s\right)$ 
&$\psi(s)=\tfrac{1}{\FAC}~\Pii{\tfrac{(K+1)p^2}{K}}{p}{\FAC~s}$ \cr
&\quad $\FAC=\sqrt{\tfrac{(-K)(K+1)}{(K+1)p^2-K}}$, &$p\in[0,1]$ \cr
&$r(s)=\sqrt{\tfrac{p^2}{K+1-Kp^2}}~\jac{cd}{p}\left(\FAC~s\right)$ 
&$\psi(s)=s-\tfrac{1}{\FAC}~\Pii{\tfrac{Kp^2}{K+1}}{p}{\FAC~s}$ \cr
&\quad $\FAC=\sqrt{\tfrac{(-K)(K+1)}{K+1-Kp^2}}$, &$p\in[0,1]$ \cr
\hline
$K>0$ 
&$r(s)=\sqrt{\tfrac{p^2}{K+p^2}}~\jac{cn}{p}\left(\FAC~s\right)$ 
&$\psi(s)=\tfrac{K+1}{\FAC}~\Pii{-\tfrac{p^2}{K}}{p}{\FAC~s}-Ks$\cr
&\quad $\FAC=\sqrt{\tfrac{K(K+1)}{K+p^2}}$, &$p\in [0,1]$ \cr
&$r(s)=\sqrt{\tfrac{1}{Kp^2+1}}~\jac{dn}{p}\left(\FAC~s\right)$ 
&$\psi(s)=\tfrac{K+1}{\FAC}~\Pii{-\tfrac{1}{K}}{p}{\FAC~s}-Ks$ \cr
&\quad $\FAC=\sqrt{\tfrac{K(K+1)}{Kp^2+1}}$ &$p\in [0,1]$
\end{tabular}
\caption{Parameter functions of a rotational CGC 
surface in \texorpdfstring{$\S^3$}{S3}.}
\label{tab:SphericalCase}
\end{table}

\begin{figure}%
\centering
\subfigure[][$K=-2$, $\jac{cn}{}$-type]{%
\label{fig:Curves_S3_Neg1}%
\includegraphics[width=3\columnwidth/10]{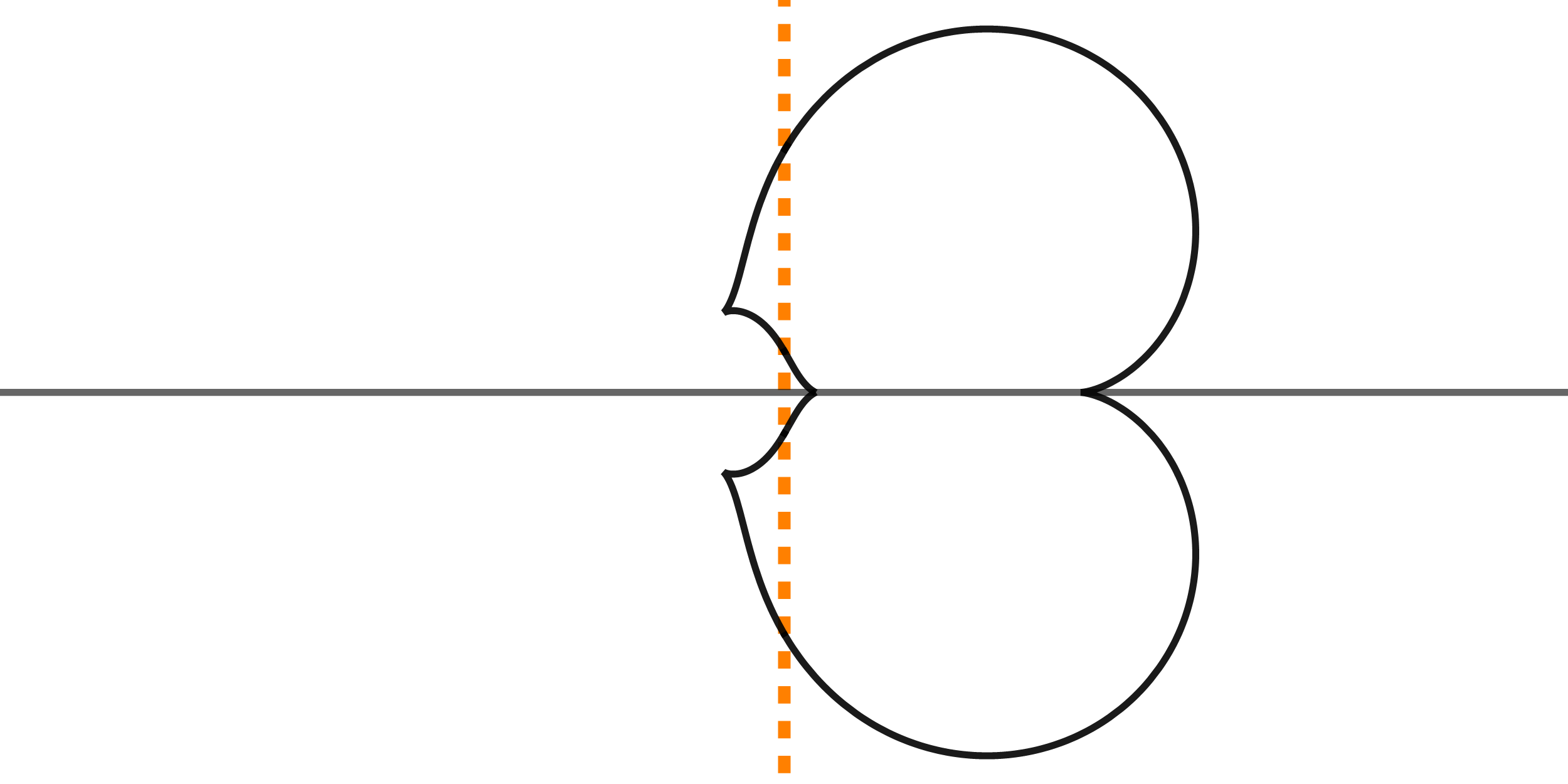}%
}%
\hfill 
\subfigure[][$K=-2$, $\jac{dn}{}$-solution]{%
\label{fig:Curves_S3_Neg2}%
\includegraphics[width=3\columnwidth/10]{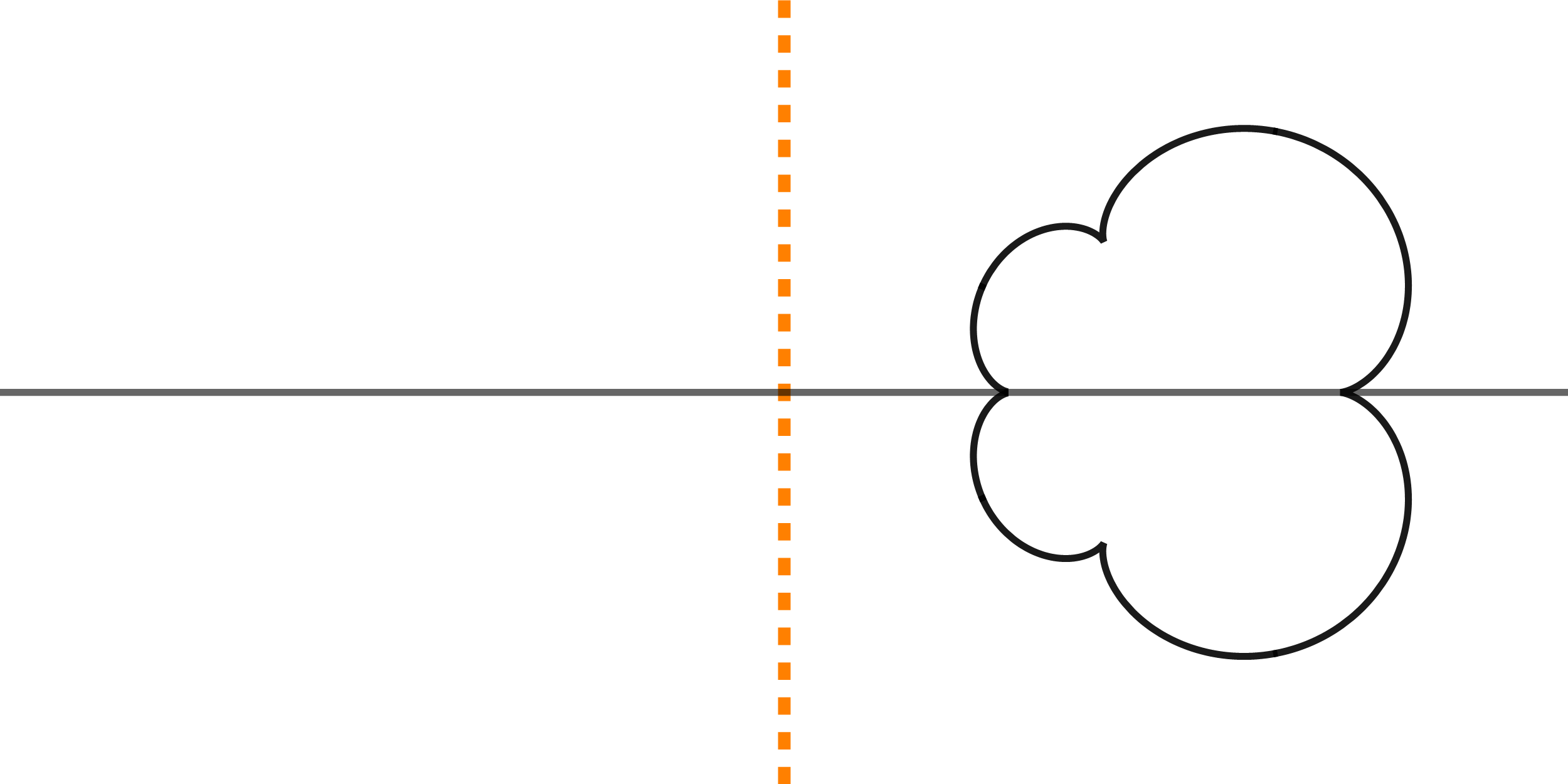}%
}% 
\hfill
\subfigure[][$K=-0.4$,first $\jac{cd}{}$-solution]{%
\label{fig:Curves_S3_Mix1}%
\includegraphics[width=3\columnwidth/10]{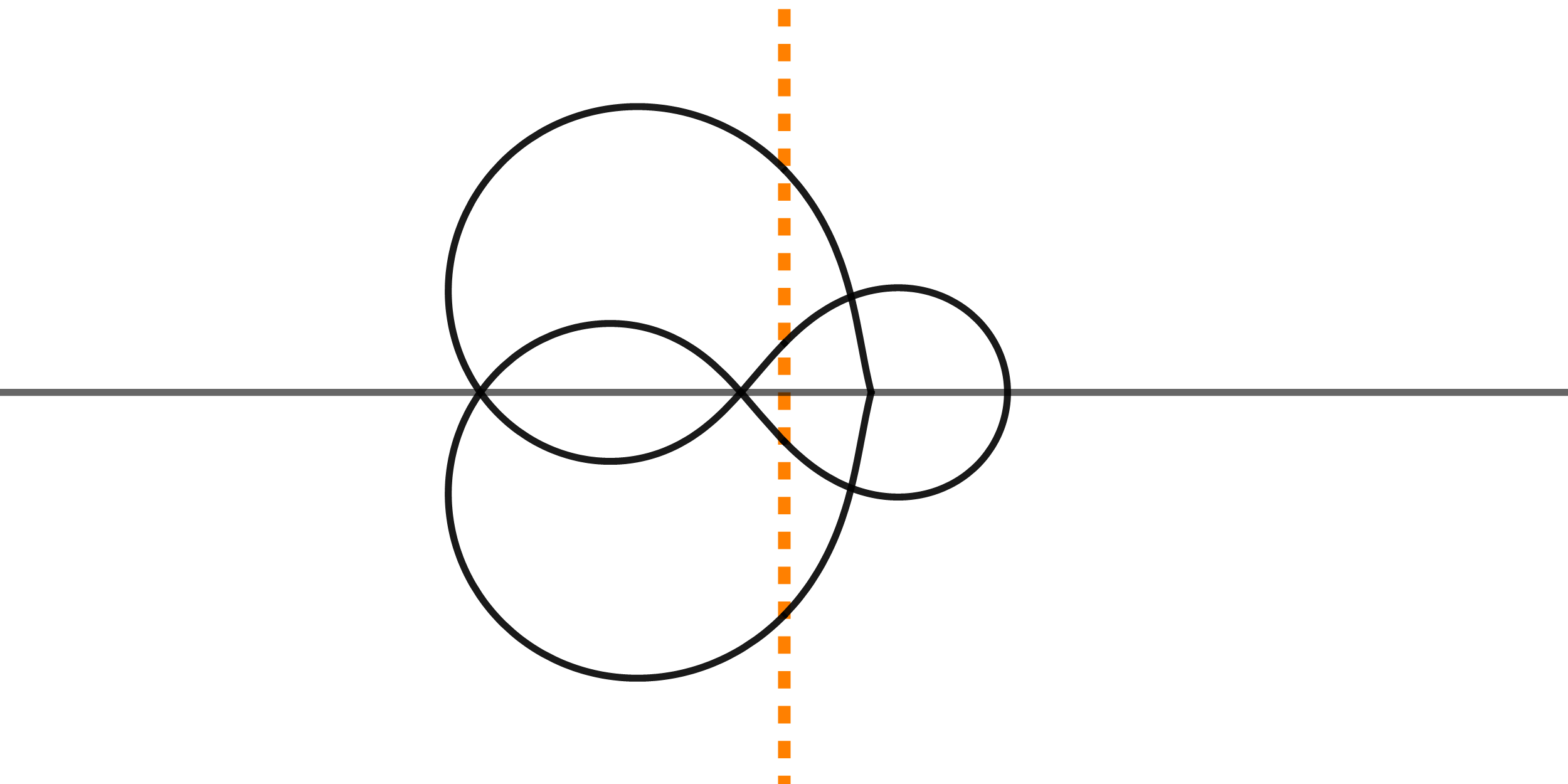}%
}% 
\\
\subfigure[][$K=-0.4$, second $\jac{cd}{}$-solution]{%
\label{fig:Curves_S3_Mix2}%
\includegraphics[width=3\columnwidth/10]{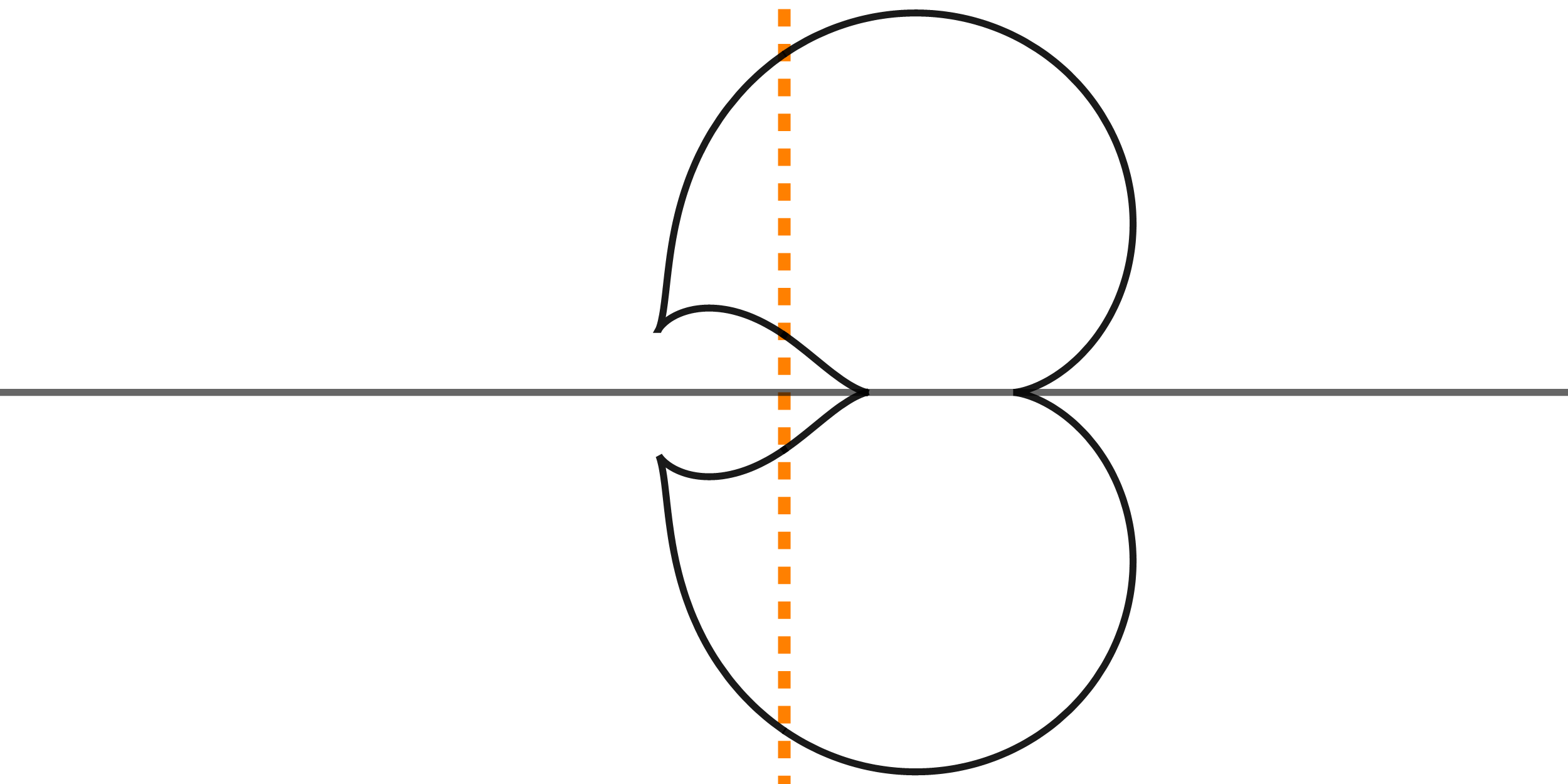}%
}%
\hfill
\subfigure[][$K=1$, $\jac{cn}{}$-solution]{%
\label{fig:Curves_S3_Pos1}%
\includegraphics[width=3\columnwidth/10]{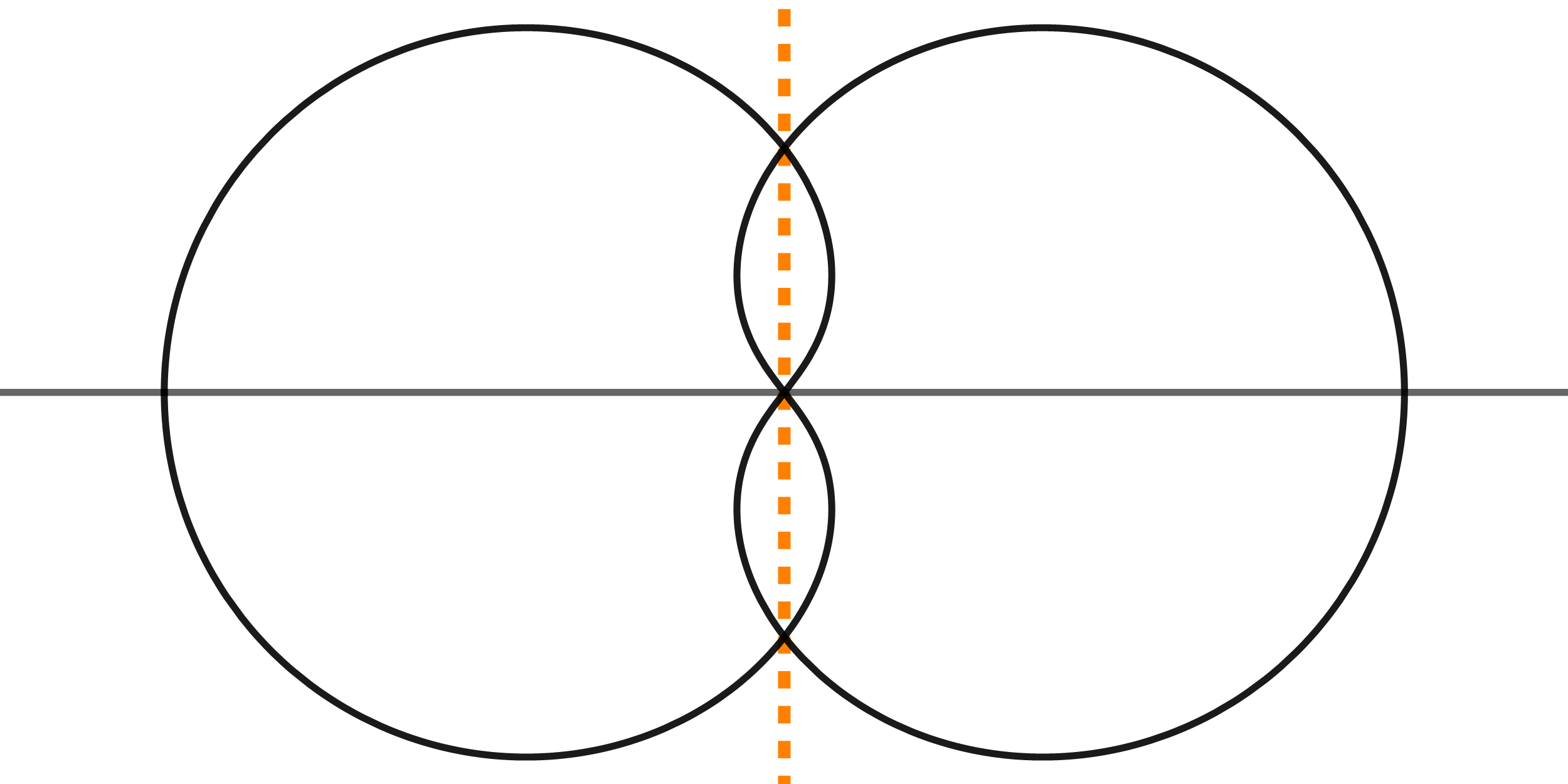}%
}% 
\hfill
\subfigure[][$K=1$, $\jac{dn}{}$-solution]{%
\label{fig:Curves_S3_Pos2}%
\includegraphics[width=3\columnwidth/10]{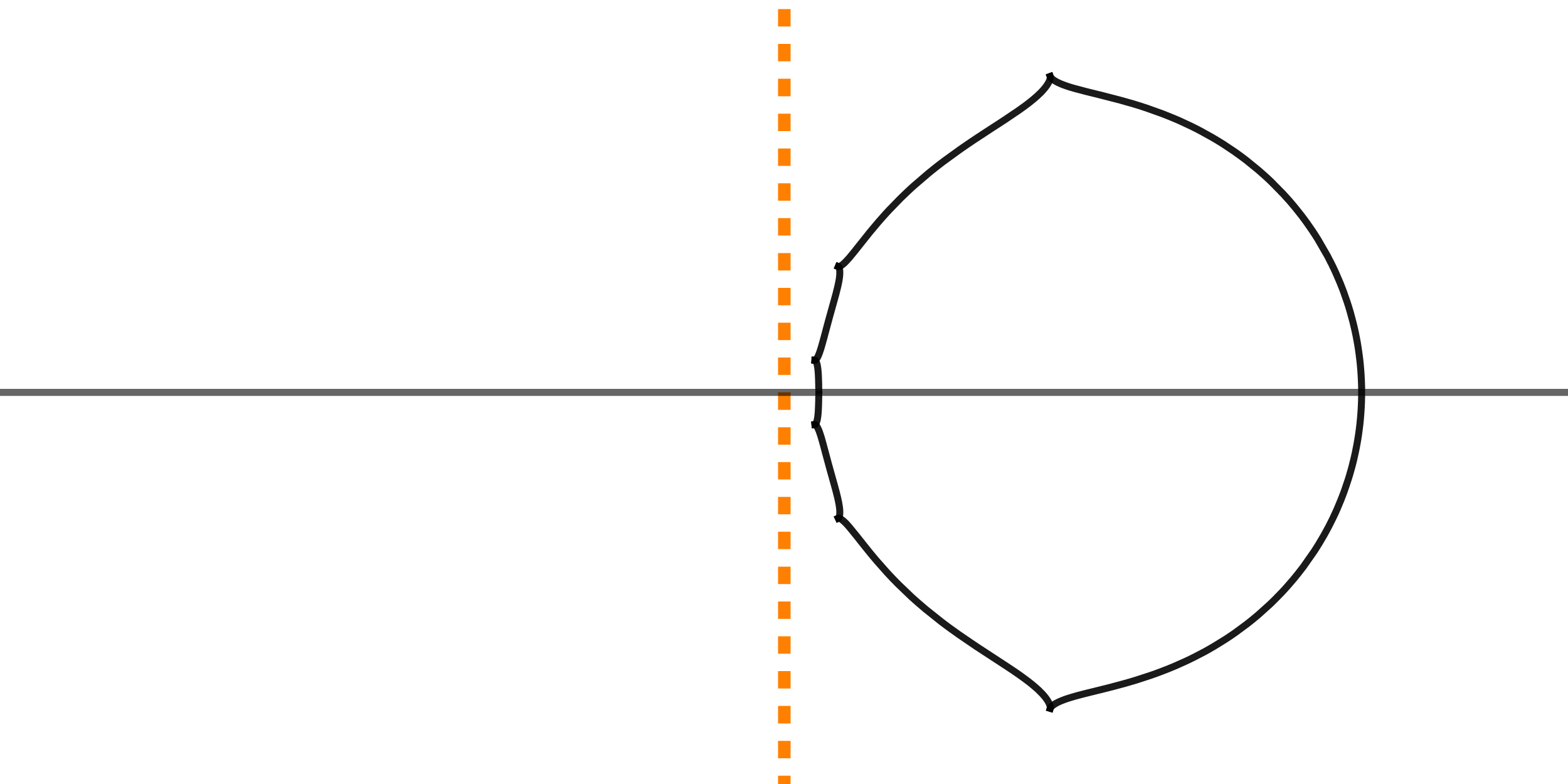}%
}% 
\caption{Profile curves of CGC surfaces of 
elliptic rotation in $\S^3$ after stereographic projection: the 
dashed line represents the axis of rotation. The profile curves 
are obtained from the solutions in 
\Cref{tab:SphericalCase} and the surfaces obtained via 
rotation are displayed in \Cref{fig:SphereSurfaces}.}%
\label{fig:Curves_S3}%
\end{figure}

\begin{figure}%
	\centering
	\subfigure[][$K=-2$, $\jac{cn}{}$-solution]{%
		\label{fig:S3_Neg1}
		\includegraphics[width=\columnwidth/4]{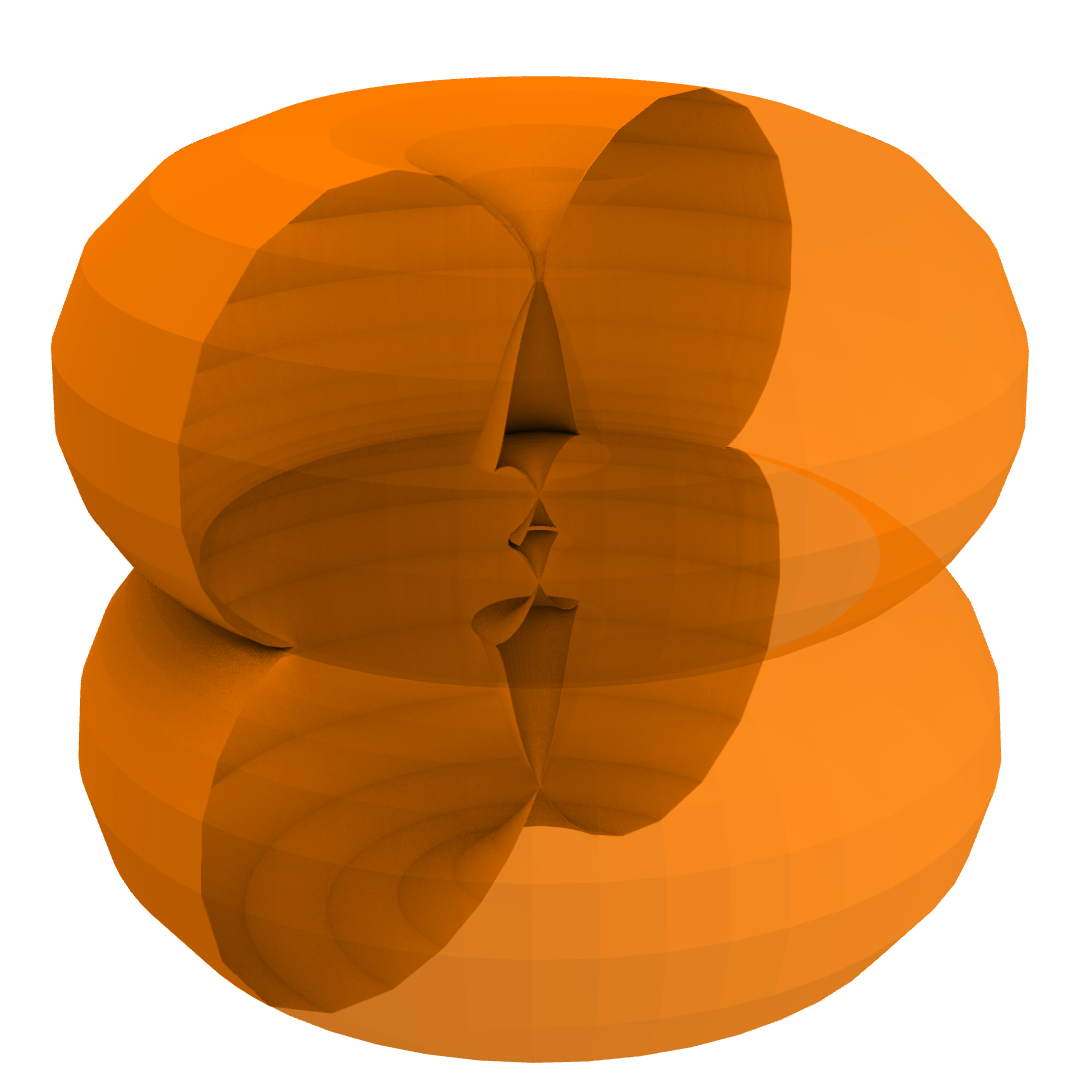}%
	}%
	\hfill 
		\subfigure[][$K=-2$, $\jac{dn}{}$-solution]{%
				\label{fig:S3_Neg2}
		\includegraphics[width=\columnwidth/4]{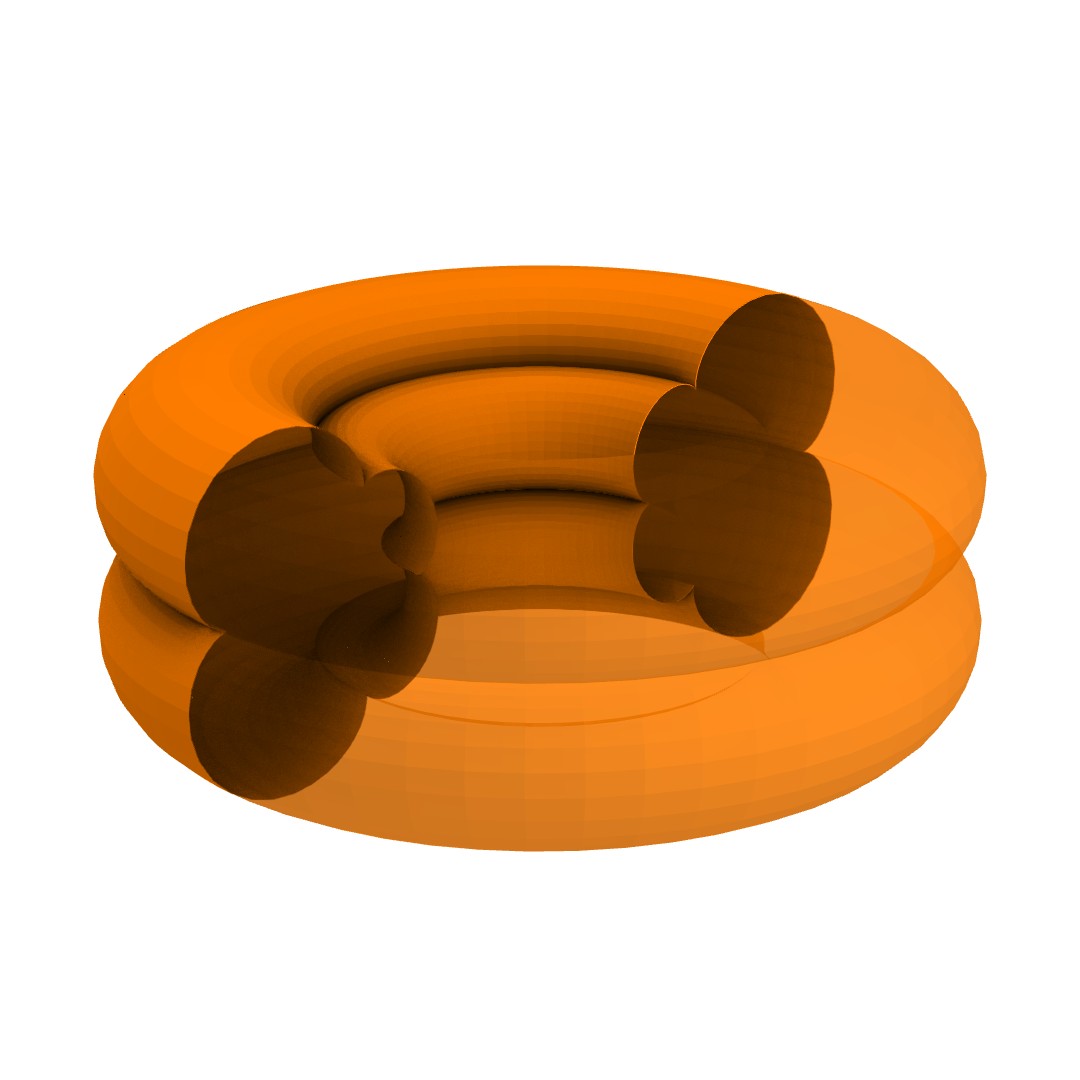}%
	}% 
	\hfill
		\subfigure[][$K=-0.4$, first solution]{%
		\label{fig:S3_Mix1}
		\includegraphics[width=\columnwidth/4]{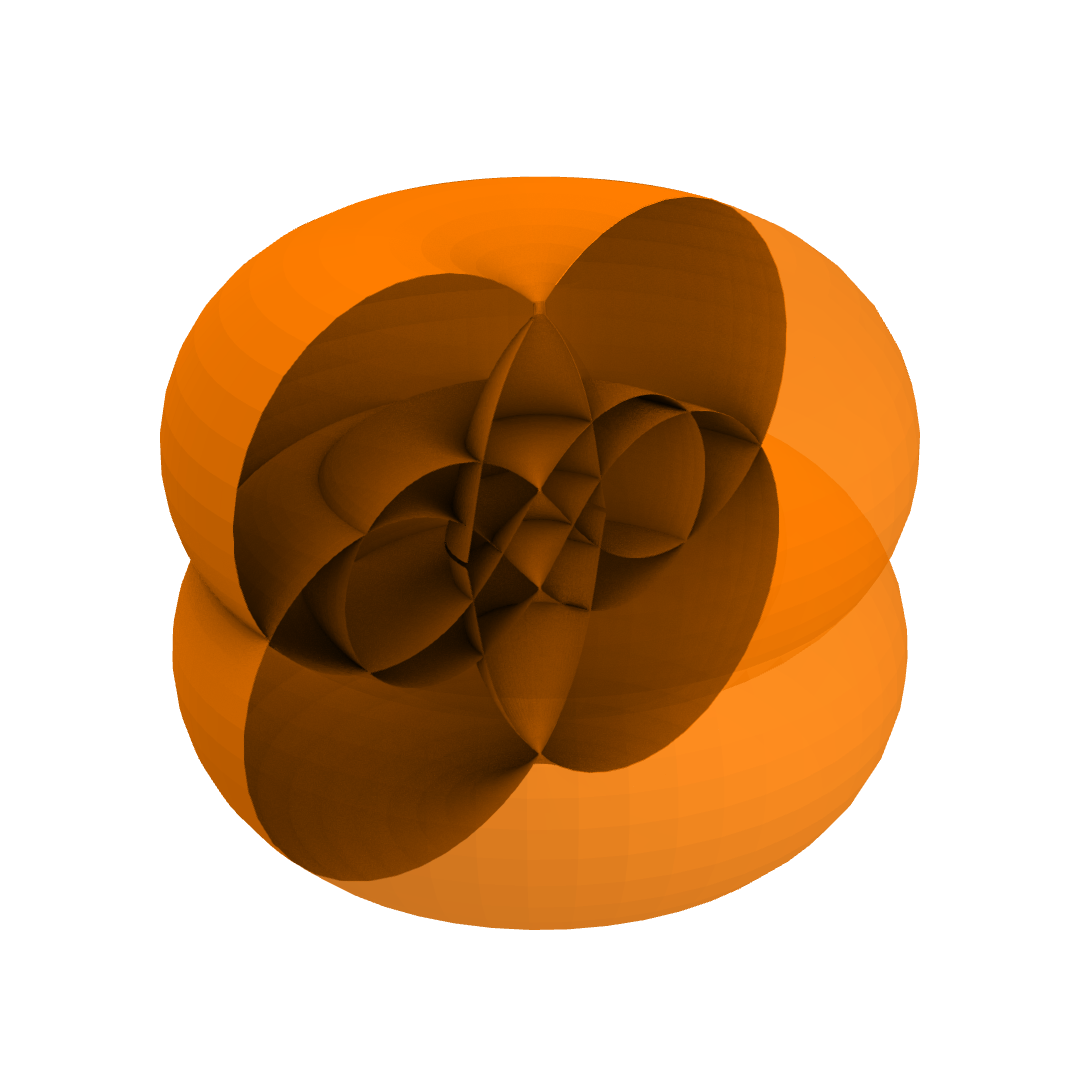}%
	}% 
	\\
		\subfigure[][$K=-0.4$, second solution]{%
		\label{fig:S3_Mix2}
		\includegraphics[width=\columnwidth/4]{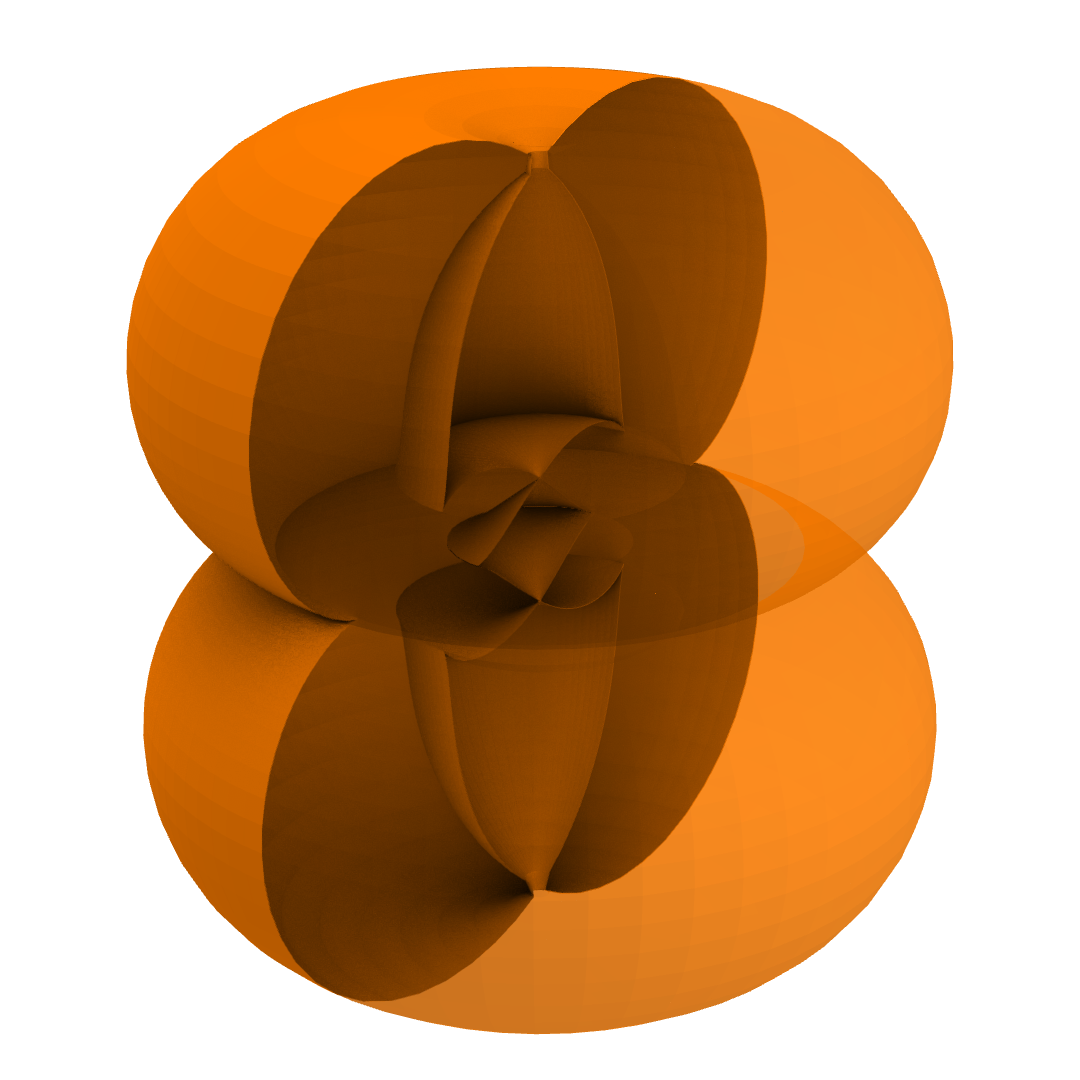}%
	}%
	\hfill 
		\subfigure[][$K=1$, $\jac{cn}{}$-solution]{%
				\label{fig:S3_Pos1}
		\includegraphics[width=\columnwidth/4]{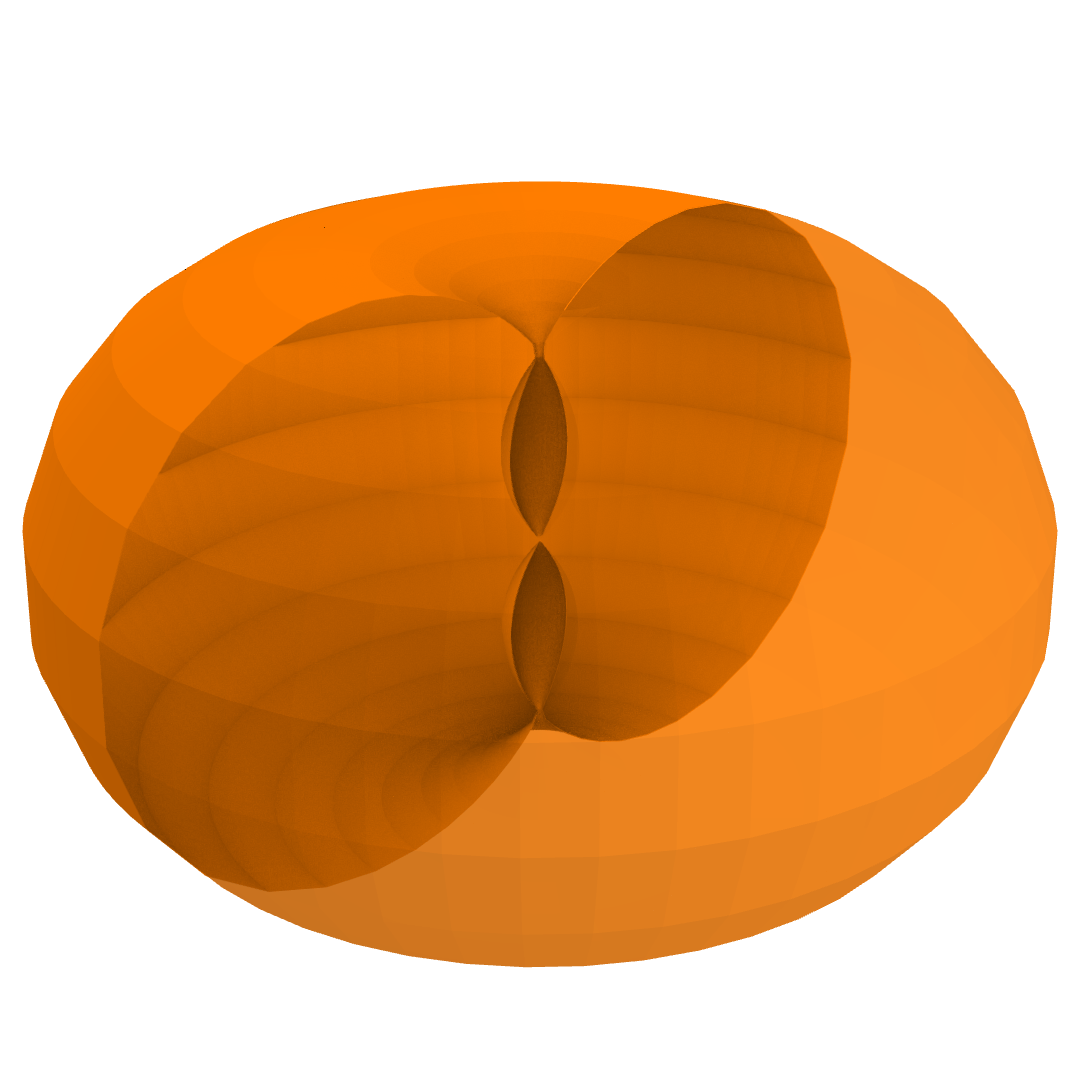}%
	}% 
	\hfill
		\subfigure[][$K=1$, $\jac{dn}{}$-solution]{%
		\label{fig:S3_Pos2}
		\includegraphics[width=\columnwidth/4]{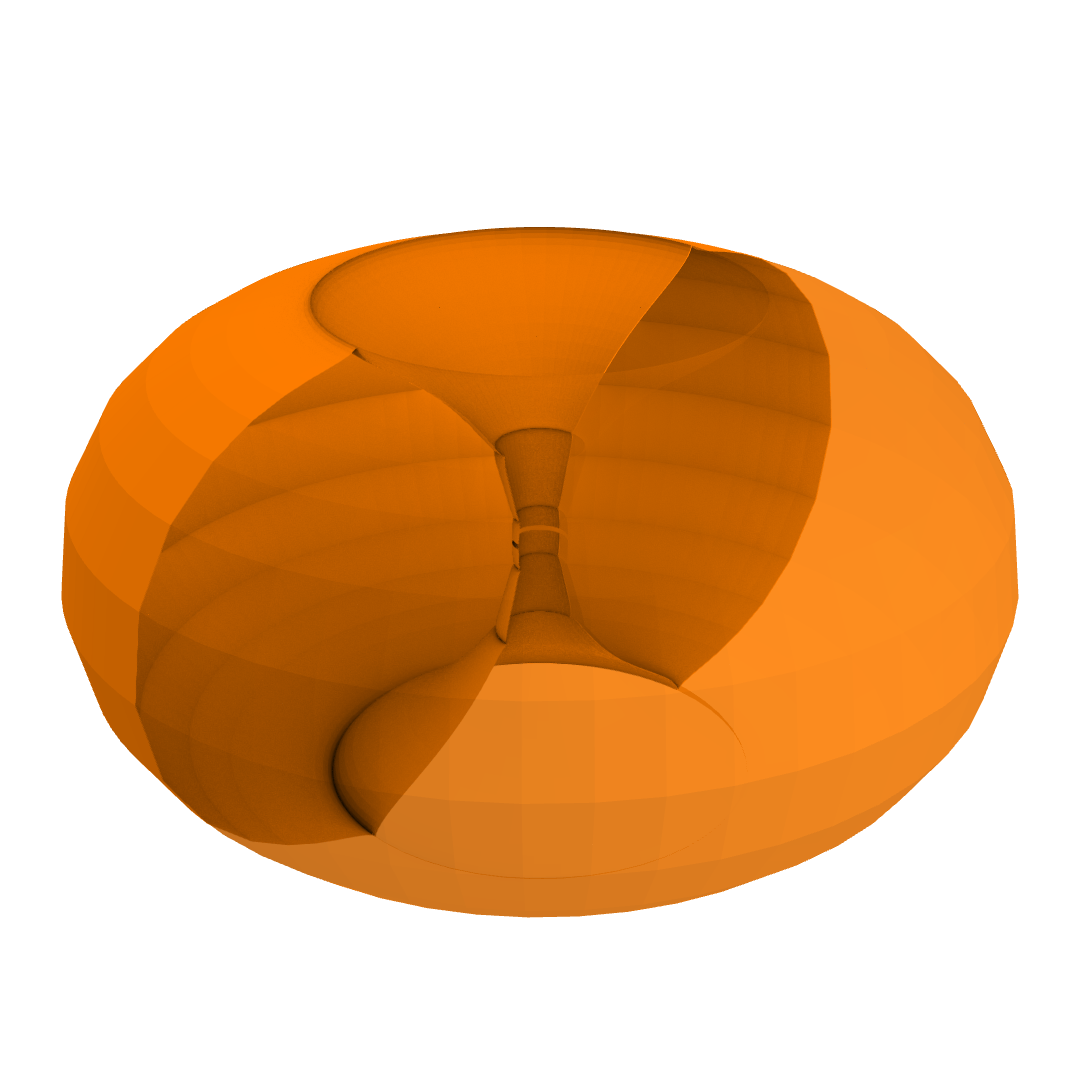}%
	}%
\caption{Stereographic projections of rotational surfaces of
different Gauss curvatures in $\S^3$. The figures are obtained
from the parametrisation given in \Cref{thm:Sphere}, with the
functions $r$ and $\psi$ from \Cref{tab:SphericalCase};
the meridian curves are shown in Figure \ref{fig:Curves_S3}.}%
	\label{fig:SphereSurfaces}%
\end{figure}

\subsection{CGC surfaces of elliptic rotation in \texorpdfstring{$\H^3$}{H3}}
 For surfaces of elliptic rotation in $\H^3$, we assume 
${\kappa_1 = -\kappa_2 = 1}$ and $\kappa=-1$ in \eqref{eq:MultiTable}, 
which amounts to choosing an orthonormal basis in the $4$-dimensional 
space ${\{v\in \R^{4,2}| (v,\q)=-1,~(v,\p)=0\}}$, viewed as a copy of 
$\R^{3,1}$.
In this case \eqref{eq:JEEd} reads $d^2 = r^2+1$ which poses no additional 
condition. Thus \eqref{eq:BoundsD} yields the cases
\begin{itemize}
	\item $K-1>0$ with $C>0$, 
	\item $K<0$ with $C-1<0$ and
	\item $K \in (0,1)$ with $C$ unrestricted. 
\end{itemize}
	
	By \Cref{prop:Bonnet}, every linear Weingarten 
surface with $\left|\tfrac{a+c}{2}\right| > |b|$ is parallel to 
a CGC surface with $K\neq 0$.
On the other hand, if $ac-b^2 <0$, then its parallel 
family also contains a pair of constant mean curvature $H>1$ or 
constant harmonic mean curvature surfaces $K/H <1$\footnote{Surfaces 
with $H\equiv 1$ or $K/H\equiv 1$ are \emph{linear Weingarten of 
Bryant type}, as are flat fronts with $K-1 \equiv 0$.}.  
			
	Constant mean curvature $H$ surfaces of elliptic rotation are 
considered in \cite{gomes1987}. In the case $H>1$, these arise in 
$1$-parameter families that are similar to that of Delaunay surfaces
in $\R^3$. By 
\Cref{prop:Bonnet}, these surfaces are parallel to the CGC surfaces 
considered here.
		
\begin{thm}\label{thm:Elliptic}
Every constant Gauss curvature $K\neq 0$ surface of 
elliptic rotation in $\H^3 \subset \R^{3,1}$ is given in an 
orthonormal basis by
\begin{align*}
(s, \theta) \mapsto \left( \sqrt{1+r^2(s)}~\cosh \psi(s), 
\sqrt{1+r^2(s)}~\sinh \psi(s), r(s) \cos \theta, r(s) \sin \theta 
\right), 
\end{align*}
where $r$, $\psi$ are listed in \Cref{tab:HyperbolicCaseElliptic}. 
\end{thm}

\begin{table}[ht]%
\centering
\begin{tabular}{l|ll}
$K<0$ 
&$r(s)=\sqrt{\frac{p^2}{1-K-p^2}}~\jac{cn}{p}\left(\FAC~s\right)$
&$\psi(s)=\tfrac{K}{\FAC}~\Pii{\frac{p^2}{1-K}}{p}{\FAC~s}-Ks$ \cr
&\quad $\FAC=\sqrt{\frac{(-K)(1-K)}{1-K-p^2}}$, &$p\in [0,1]$ \cr
&$r(s)= \sqrt{\frac{1}{(1-K)p^2-1}}~\jac{dn}{p}\left(\FAC~s\right)$
&$\psi(s)=\tfrac{K}{\FAC}~\Pii{\frac{1}{1-K}}{p}{\FAC~s}-Ks$ \cr
&\quad $\FAC=\sqrt{\frac{(-K)(1-K)}{(1-K)p^2-1}}$, &$p\in\left[\sqrt{\frac{1}{1-K}},1\right]$ \cr
\hline 
$K\in(0,1)$
&$r(s)=\sqrt{\frac{1-p^2}{K-1+p^2}}~\jac{nc}{p}\left(\FAC~s\right)$
&$\psi(s)= \tfrac{\FAC(1-p^2)}{K}~\Pii{\tfrac{1-K}{\FAC^2}}{p}{\FAC~s}-\tfrac{\FAC^2 p^2}{1-K}s$ \cr
&\quad $\FAC=\sqrt{\frac{K(1-K)}{K-1+p^2}}$, &$p\in[\sqrt{1-K},1]$ \cr
&$r(s)=\sqrt{\frac{1-p^2}{1+p^2(K-1)}}~\jac{sc}{p}\left(\FAC~s\right)$
&$\psi(s)= \tfrac{\FAC(1-p^2)}{K}~\Pii{\tfrac{\FAC^2p^2}{1-K}}{p}{\FAC~s}-s$ \cr
&\quad $\FAC=\sqrt{\frac{K(1-K)}{1+p^2(K-1)}}$, &$p\in[0,1]$ \cr
&$r(s)=\sqrt{\frac{1-p^2}{1-Kp^2}}~\jac{sc}{p}\left(\FAC~s\right)$
&$\psi(s)= \tfrac{\FAC(1-p^2)}{K-1}~\Pii{\tfrac{\FAC^2p^2}{K}}{p}{\FAC~s}$ \cr
&\quad $\FAC=\sqrt{\frac{K(1-K)}{1-Kp^2}}$, &$p\in[0,1]$ \cr
&$r(s)= \sqrt{\frac{1-p^2}{p^2-K}}~\jac{nc}{p}\left(\FAC~s\right)$
&$\psi(s)= \tfrac{\FAC(1-p^2)}{K-1}~\Pii{\tfrac{K}{\FAC^2}}{p}{\FAC~s}+\tfrac{\FAC^2(1-p^2)}{1-K}s$ \cr
&\quad $\FAC=\sqrt{\frac{K(1-K)}{p^2-K}}$, &$p\in [\sqrt{K},1]$ \cr 
\hline
$K=1$ 
&$r(s)=\sqrt{\tfrac{1-p^2}{p^2}}~\cosh~\left(\tfrac{s}{p}\right)$ &$\psi(s)=s-\operatorname{artanh}\left(p\tanh\left(\tfrac{s}{p}\right)\right)$ \cr
&\quad $p\in (0,1)$ &Snowman fronts\cr
&$r(s)=\sqrt{1-p^2}~\sinh~(ps)$ &$\psi(s)=\tfrac{1}{1-p^2}\left(s- \operatorname{artanh}(p~\tanh(ps)\right)$ \cr
&\quad $p\in(0,1)$ &Hourglass fronts \cr
\hline
$K>1$
&$r(s)=\sqrt{\frac{p^2}{K-p^2}}~\jac{cn}{p}\left(\FAC~s\right)$
&$\psi(s)=\tfrac{K-1}{\FAC}~\Pii{\frac{p^2}{K}}{p}{\FAC~s}-Ks$ \cr
&\quad $\FAC=\sqrt{\frac{K(K-1)}{K-p^2}}$, &$p\in[0,1]$ \cr
&$r(s)=\sqrt{\frac{1}{Kp^2-1}}~\jac{dn}{p}\left(\FAC~s\right)$ 
&$\psi(s)=\tfrac{K-1}{\FAC}~\Pii{\frac{1}{K}}{p}{\FAC~s}-Ks$ \cr
&\quad $\FAC=\sqrt{\frac{K(K-1)}{Kp^2-1}}$, &$p\in\left[\sqrt{\frac{1}{K}},1\right]$ 
\end{tabular}
\caption{Parameter functions of CGC surfaces of elliptic rotation in \texorpdfstring{$\H^3$}{H3}.}
\label{tab:HyperbolicCaseElliptic}
\end{table}

\begin{bem}
The bifurcation in the mixed case $K\in (0,1)$ is of a slightly 
different flavor than before. Since $C$ is unbounded, three cases 
emerge: for $C<0$ the solution of Subsection \ref{punkt:PosCurv} is real 
with $p>1$ and imaginary argument, which yields an $\jac{nc}{}$-type
solution. Similarly, for $C>1$, the solution of Subsection \ref{punkt:NegCurv} 
takes an $\jac{nc}{}$-form. If, however, $C\in [0,1]$ the solution to 
\eqref{eq:JEE} is given in Subsection \ref{punkt:MixCurv}, with 
$p^2 \in [0, \infty)$ which splits again in two cases with moduli 
in $[0,1]$, which yields the $\jac{sc}{}$-type solutions. 
\end{bem}
	
\begin{figure}%
\centering
\subfigure[][$K=-1$, $\jac{cn}{}$-solution]{%
\label{fig:Curves_H3_ell_NegCN}%
\includegraphics[width=\columnwidth/5]{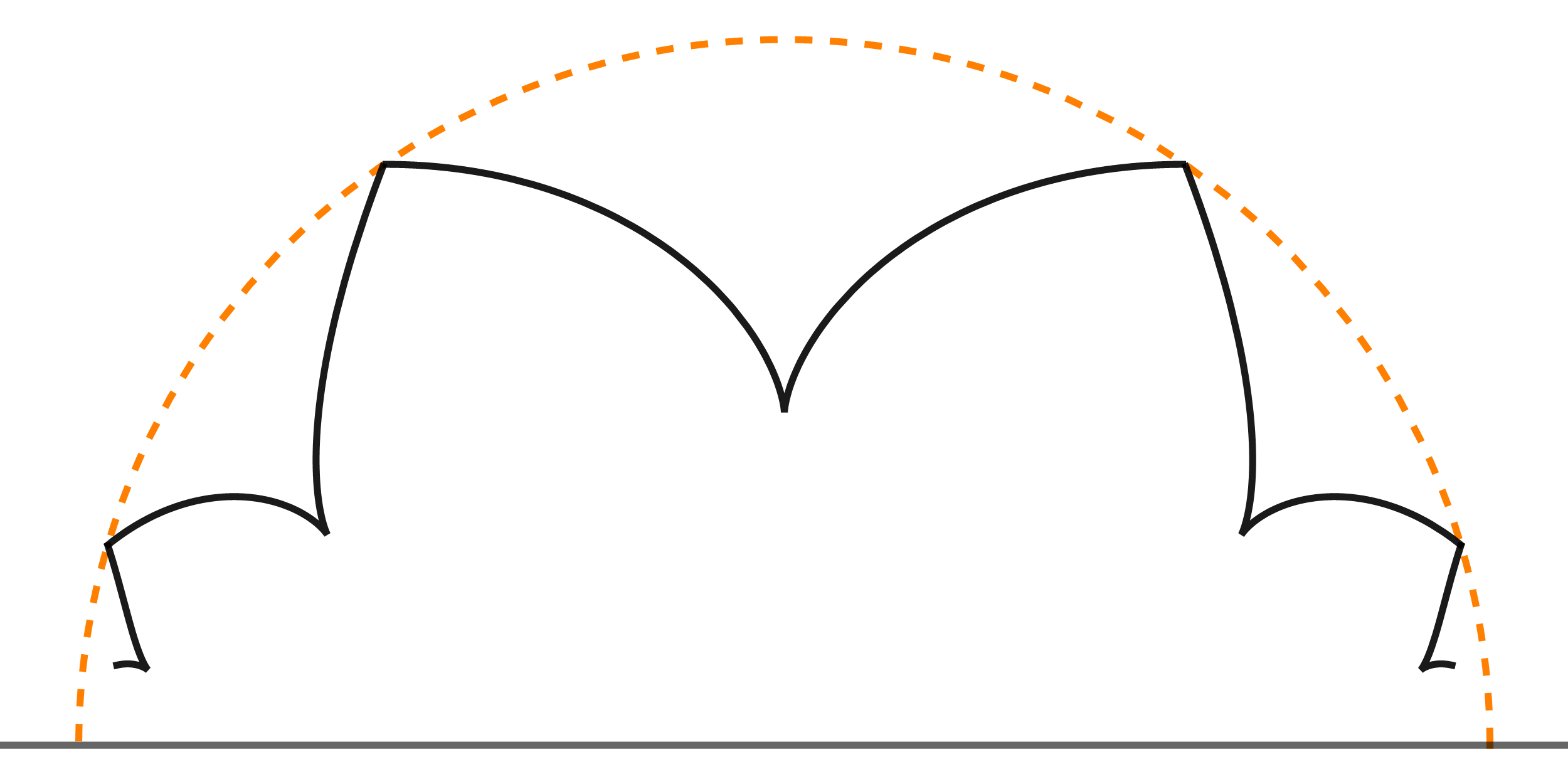}%
}%
\hfill 
\subfigure[][$K=-1$, $\jac{dn}{}$-solution]{%
\includegraphics[width=\columnwidth/5]{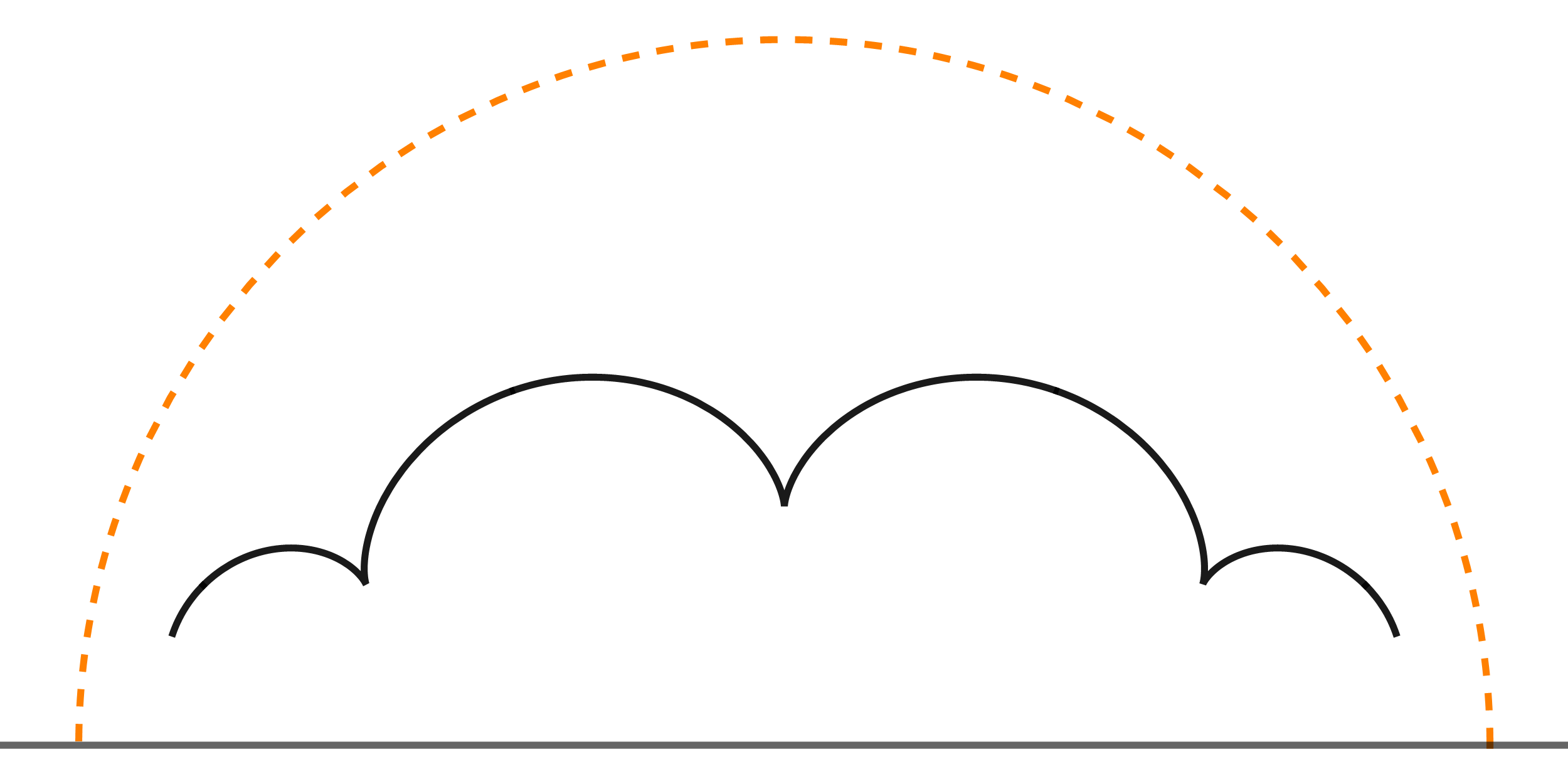}%
}% 
\hfill
\subfigure[][$K=0.4$, first $\jac{sc}{}$-solution]{%
\label{fig:Curves_H3_ell_MixSC}%
\includegraphics[width=\columnwidth/5]{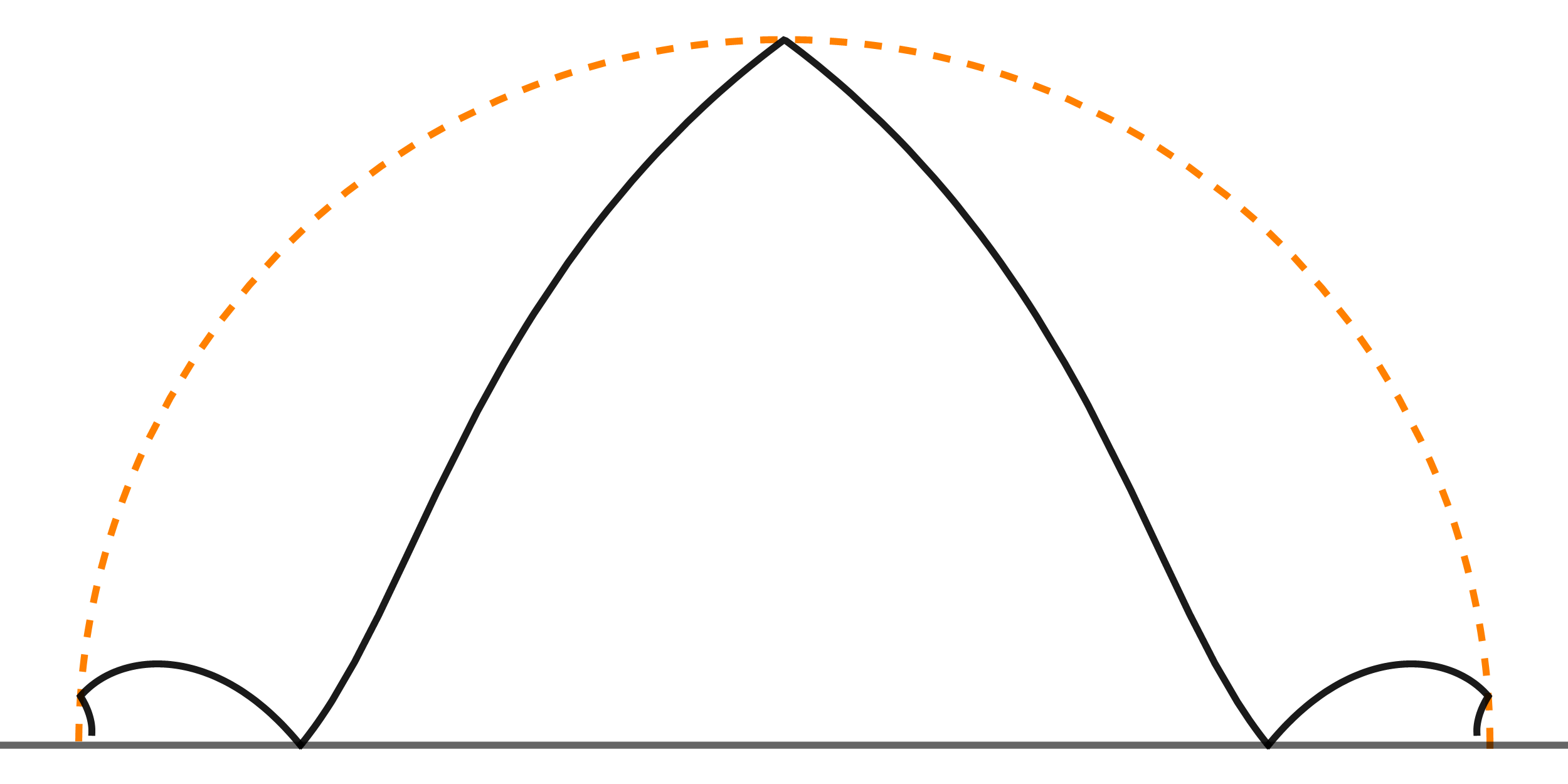}%
}% 
\hfill
\subfigure[][$K=0.4$, second $\jac{sc}{}$-solution]{%
\includegraphics[width=\columnwidth/5]{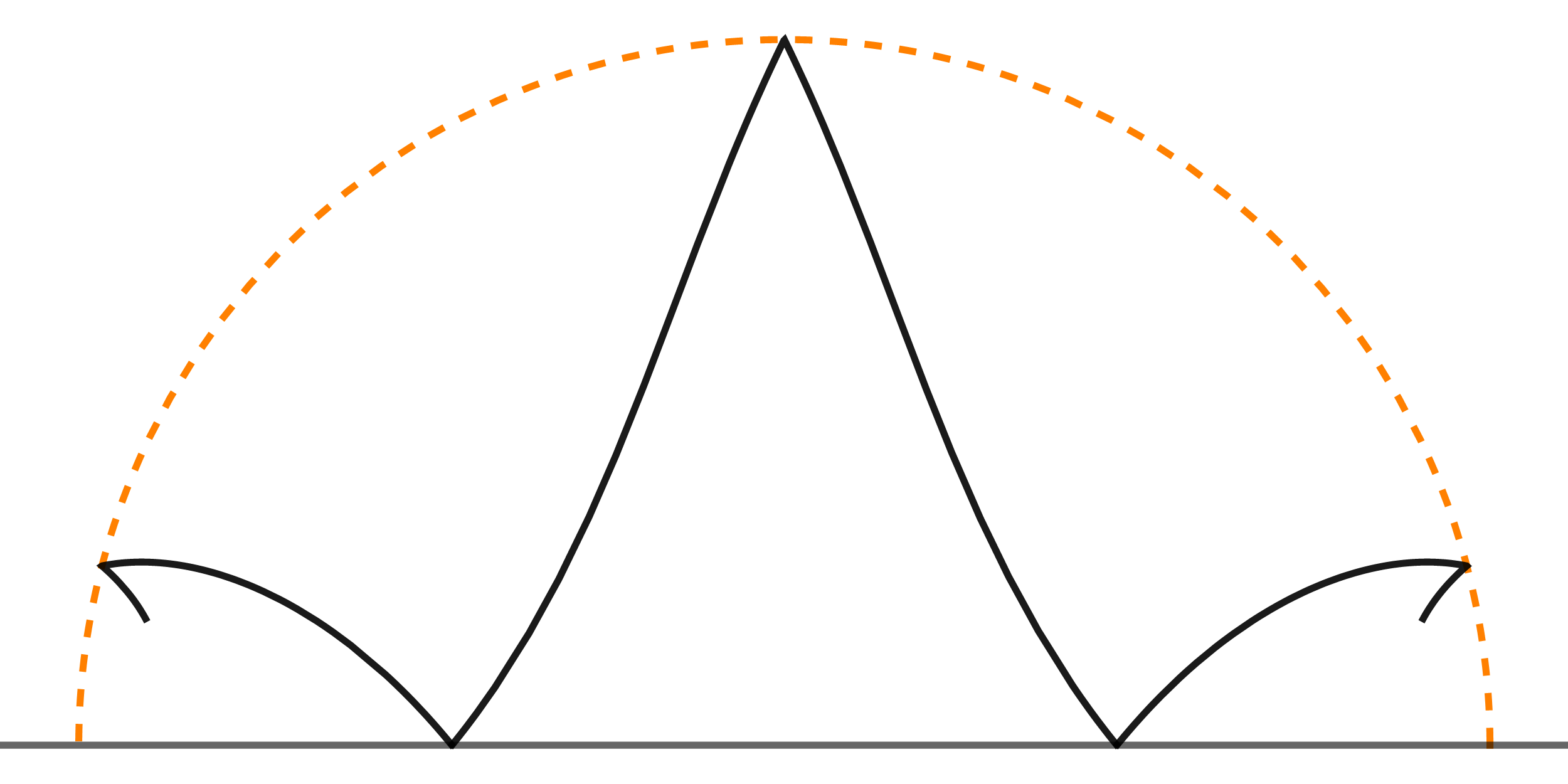}%
}%
\\
\subfigure[][$K=0.4$, first $\jac{nc}{}$-solution]{%
\includegraphics[width=\columnwidth/5]{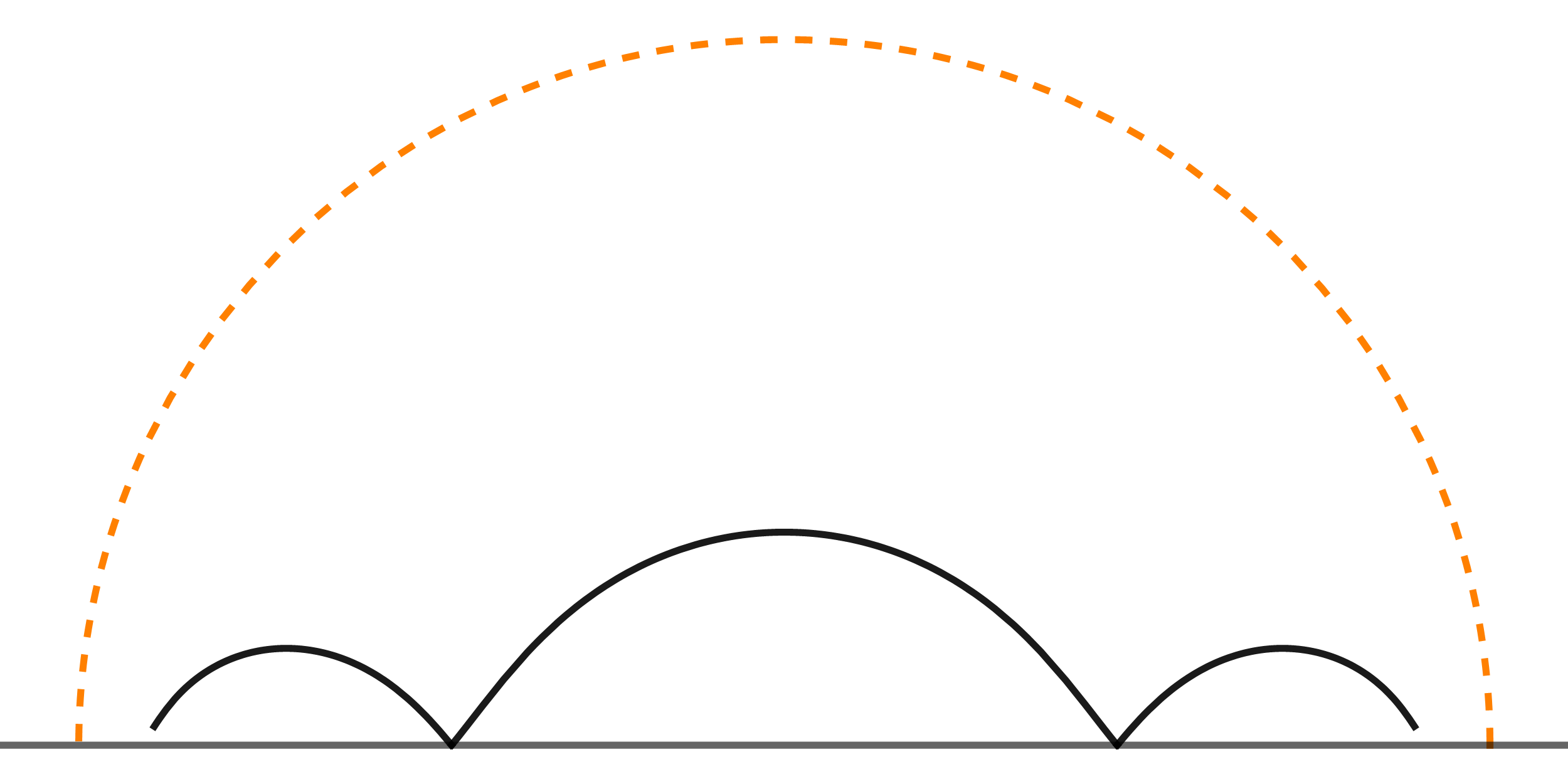}%
}%
\hfill
\subfigure[][$K=0.4$, second $\jac{nc}{}$-solution]{%
\includegraphics[width=\columnwidth/5]{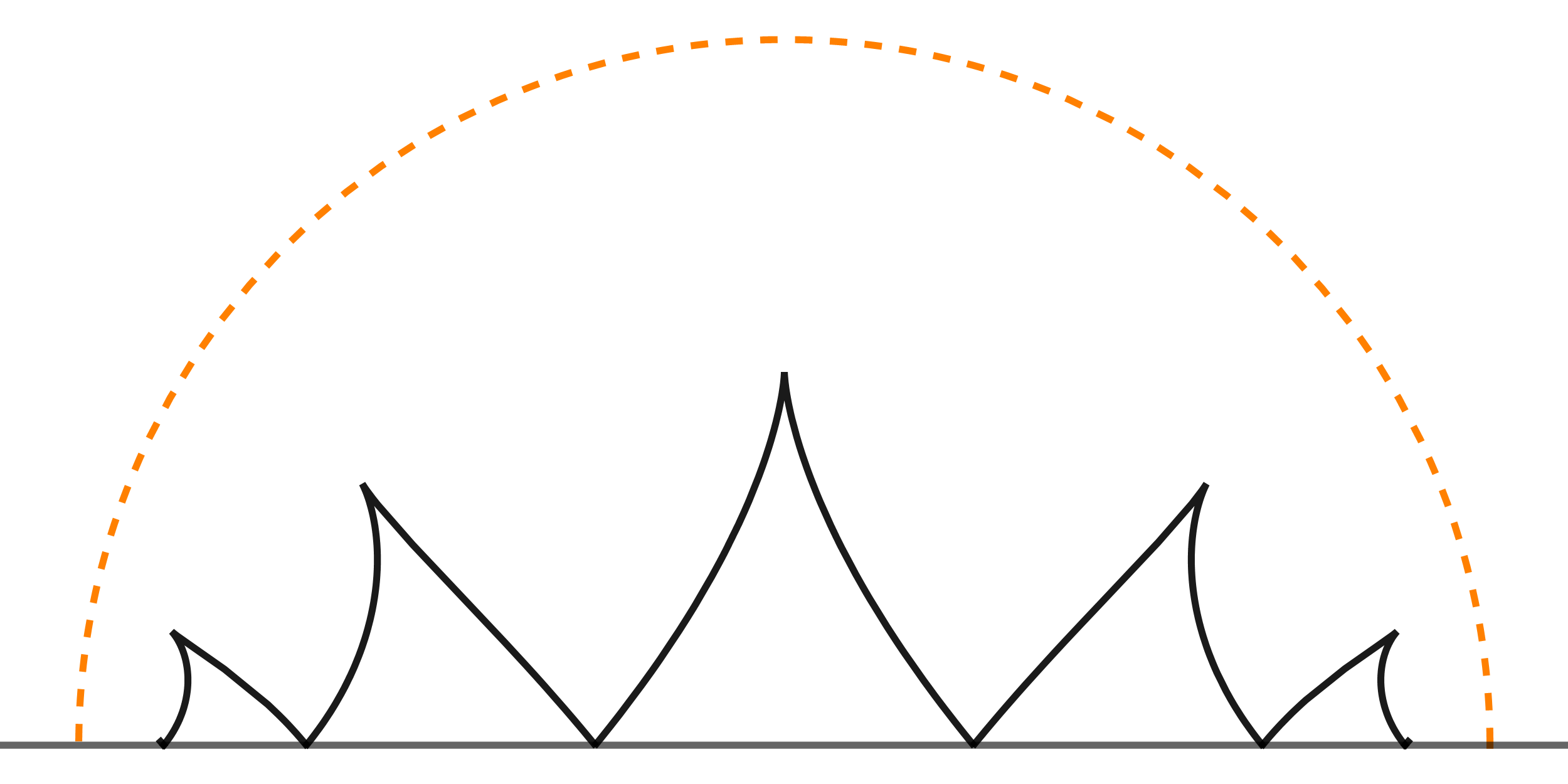}%
}% 
\hfill
\subfigure[][$K=2$, $\jac{cn}{}$-solution]{%
\label{fig:Curves_H3_ell_PosCN}
\includegraphics[width=\columnwidth/5]{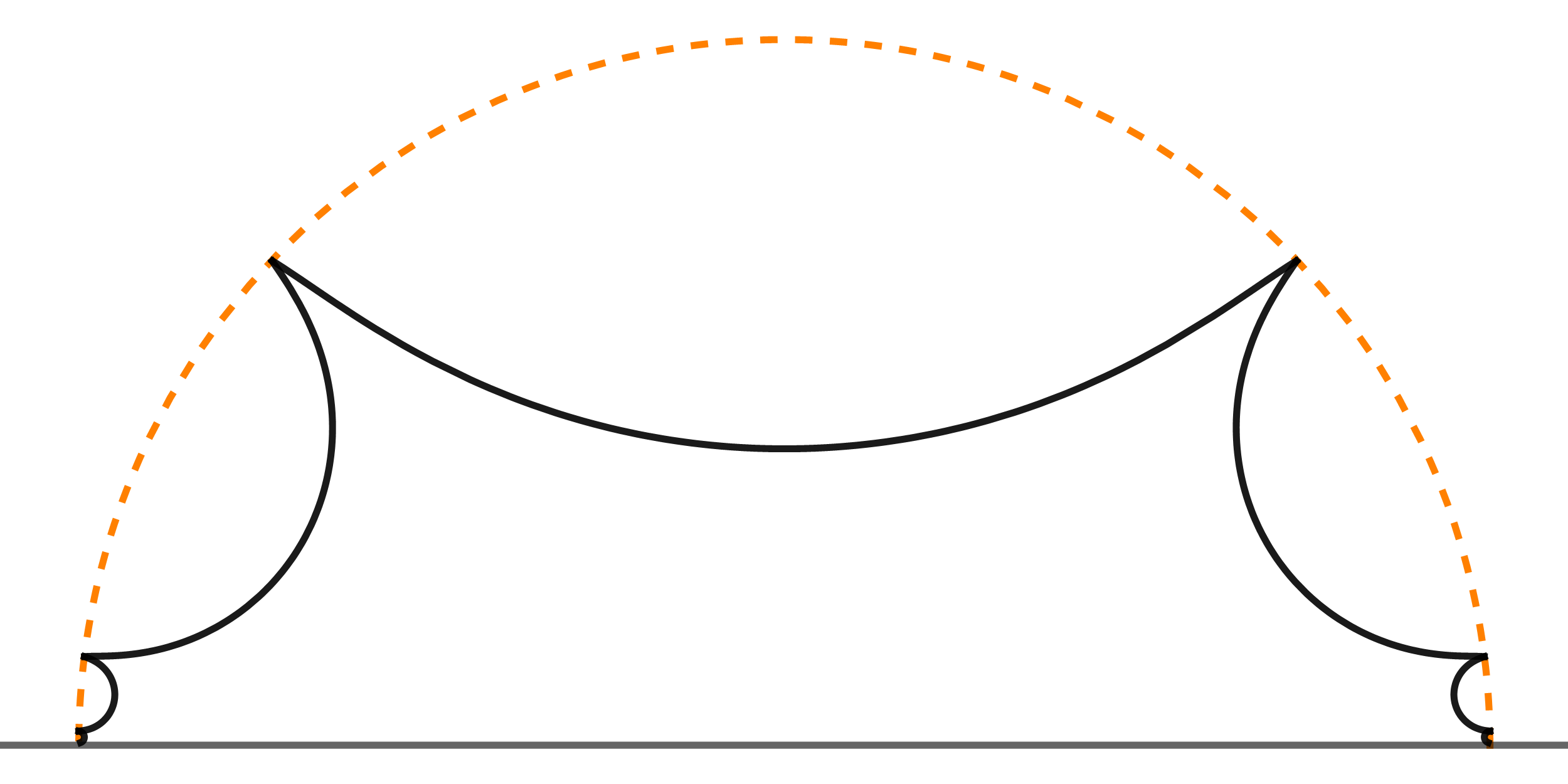}%
}% 
\hfill
\subfigure[][$K=2$, $\jac{dn}{}$-solution]{%
\includegraphics[width=\columnwidth/5]{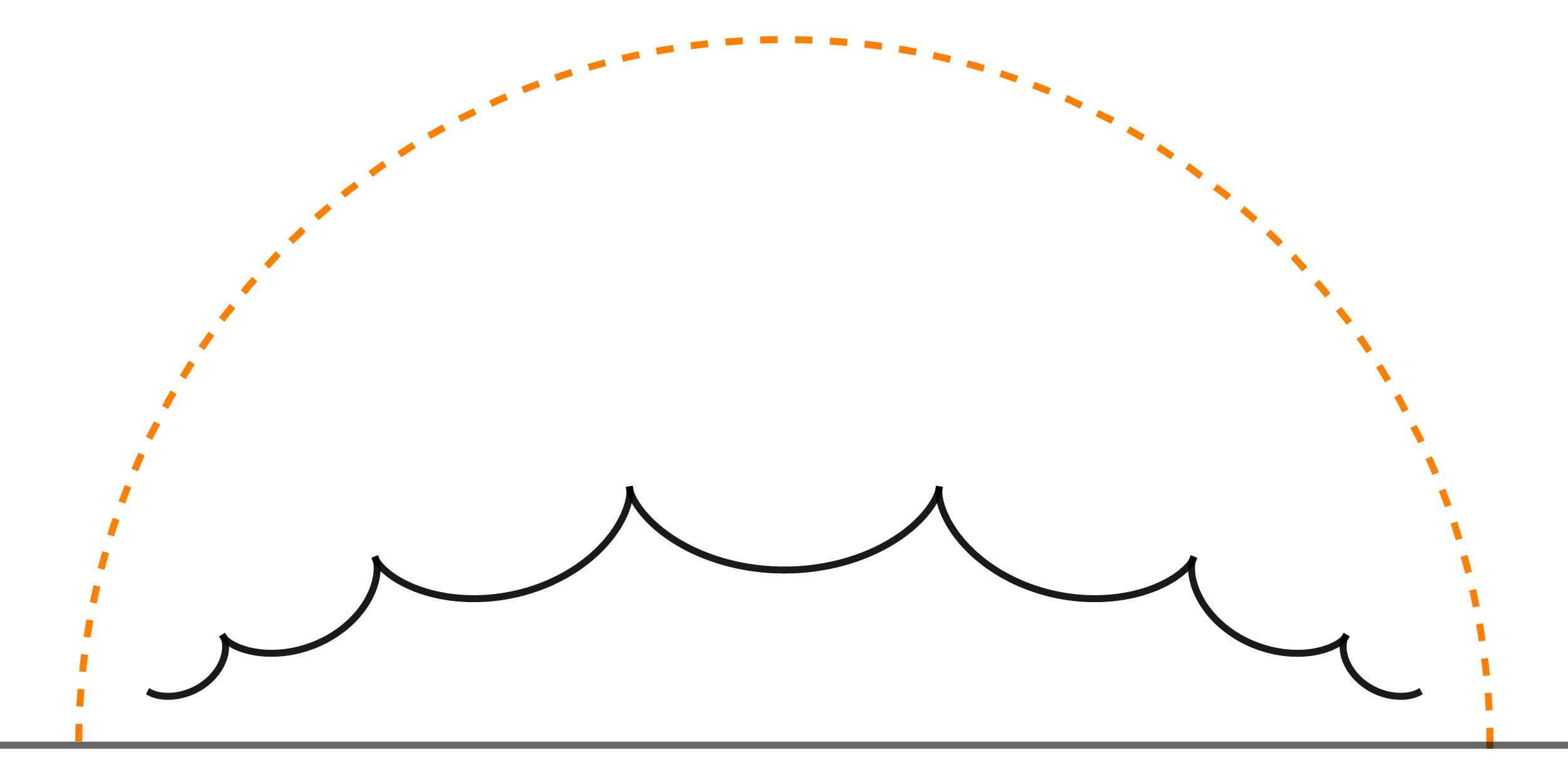}%
}%
\caption{Profile curves of CGC
 surfaces of elliptic rotation in the half plane model:
 the dashed line represents the axis of rotation. The
 profile curves are obtained using the solutions
 in \Cref{tab:HyperbolicCaseElliptic}. The elliptic
 rotational surfaces corresponding to the profile curves in Figures
 \ref{fig:Curves_H3_ell_NegCN}, \ref{fig:Curves_H3_ell_MixSC}
 and \ref{fig:Curves_H3_ell_PosCN} are displayed in Figures
  \ref{fig:Ell_Neg}, \ref{fig:Ell_Mix} and \ref{fig:Ell_Pos},
 respectively.}%
\label{fig:Curves_H3_ell}%
\end{figure}

\subsection{CGC surfaces of hyperbolic rotation in \texorpdfstring{$\H^3$}{H3}}	
For surfaces of hyperbolic rotation we assume 
$ \kappa_1 = -\kappa_2 = -1$ and, as before, $\kappa =-1$ in 
\eqref{eq:MultiTable}.
In this case, because $d^2 = r^2 - 1$ according to \eqref{eq:JEEd}, 
we are restricted to solutions of \eqref{eq:JEE} with $r^2>1$. 
Thus \eqref{eq:BoundsD} yields the three cases
\begin{itemize}
\item $K-1>0$, which implies $-C>K-1$, 
\item $K<0$, implying $-C<K-1$, and
\item $K \in (0,1)$ with $C$ unrestricted. 
\end{itemize}

\begin{thm}\label{thm:Hyperbolic}
Every constant Gauss curvature $K\neq 0$ surface of hyperbolic 
rotation in $\H^3\subset \R^{3,1}$ is given in an orthonormal basis by
\begin{align*}
(s, \theta) \mapsto \left(r(s) \cosh \theta, r(s) \sinh \theta, 
\sqrt{r^2(s)-1}~\cos \psi(s), \sqrt{r^2(s)-1}~\sin\psi(s)\right), 
\end{align*}
where $r$, $\psi$ are listed in \Cref{tab:HyperbolicCaseHyperbolic}. 
\end{thm}

\begin{table}[ht]
\centering
\begin{tabular}{l|lll}
$K<0$ 
&$r(s)=\sqrt{\frac{1}{1+(K-1)p^2}}~\jac{dn}{p}\left(\FAC~s\right)$ &$\psi(s)=\tfrac{K}{\FAC}~\Pii{\frac{1}{1-K}}{p}{\FAC~s}-Ks$ \cr
&\quad $\FAC=\sqrt{\tfrac{K(K-1)}{1+(K-1)p^2}}$, &$p\in\left[0,\sqrt{\frac{1}{1-K}}\right]$  \cr
\hline
$K\in (0,1)$ 
&$r(s)=\sqrt{\tfrac{q^2}{K-p^2}}~\jac{nc}{p}\left(\FAC~s\right)$ &$\psi(s)=\tfrac{\FAC^2 q^2}{1-K}~s+\tfrac{\FAC q^2}{K-1}~\Pii{-\tfrac{K}{\FAC^2}}{p}{\FAC~s}$ \cr
&\quad $\FAC=\sqrt{\tfrac{K(1-K)}{K-p^2}}$, &$p\in\left[0,\sqrt{K}\right]$  \cr
&$r(s)=\sqrt{\tfrac{1}{K-p^2(K-1)}}~\jac{dc}{p}(\FAC~s)$ &$\psi(s)=\tfrac{1}{\FAC}~\Pii{\frac{K}{K-1}}{p}{\FAC~s}-s$ \cr
&\quad $\FAC=\sqrt{\tfrac{K(1-K)}{K-p^2(K-1)}}$, &$p\in[0,1]$  \cr
&$r(s)=\sqrt{\tfrac{1}{1- K q^2}}~\jac{dc}{p}(\FAC~s)$ &$\psi(s)=\tfrac{1}{\FAC}~\Pii{\frac{K-1}{K}}{p}{\FAC~s}$ \cr
&\quad $\FAC=\sqrt{\tfrac{K(1-K)}{1-K q^2}}$, &$p\in[0,1]$  \cr
&$r(s)= \sqrt{\tfrac{q^2}{1-K-p^2}}~\jac{nc}{p}\left(\FAC~s\right)$ &$\psi(s)= \tfrac{\FAC^2 p^2}{K-1}~s + \tfrac{\FAC q^2}{K}~\Pii{\tfrac{K-1}{\FAC^2}}{p}{\FAC~s}$\cr
&\quad $\FAC=\sqrt{\tfrac{K(1-K)}{1-K-p^2}}$, &$p\in\left[0,\sqrt{1-K}\right]$  \cr
\hline
$K=1$ 
&$r(s)=\sqrt{1+p^2}~\cosh~\left(ps\right)$ &$\psi(s)=s-\operatorname{arctan}\left(\tfrac{1}{p}\tanh\left(ps\right)\right)$ \cr
&\quad $p\in (0,\infty)$ &Peach fronts\cr
\hline
$K>1$ 
&$r(s)=\sqrt{\frac{1}{1-Kp^2}}~\jac{dn}{p}\left(\FAC~s\right)$ &$\psi(s)=\tfrac{K-1}{\FAC}~~\Pii{\frac{1}{K}}{p}{\FAC~s}-Ks$ \cr
&\quad $\FAC=\sqrt{\frac{K(K-1)}{1-Kp^2}}$, &$p\in\left[0,\sqrt{\frac{1}{K}}\right]$
\end{tabular}
\caption{Parameter functions of CGC surfaces of hyperbolic rotation in \texorpdfstring{$\H^3$}{H\^3}.}
\label{tab:HyperbolicCaseHyperbolic}
\end{table}

\begin{bem}
The bifurcation in the mixed case $K\in (0,1)$ is similar to the case 
of elliptic rotations. Again, for $C<0$ and $C>1$, the solution from
Subsections
\ref{punkt:PosCurv} and \ref{punkt:NegCurv} are real
(which yields the $\jac{nc}{}$-type solutions). If $C\in [0,1]$, as for the 
spherical case, both solutions become real of either
$\jac{cd}{}$- or $\jac{dc}{}$-types. Since we necessarily have 
$r^2>1$, the $\jac{dc}{}$-type solutions apply.
\end{bem}

\begin{figure}%
\centering
\setlength{\unitlength}{0.1\textwidth}
\begin{picture}(10,4)
\put(2,0){\includegraphics[width=0.6\textwidth]{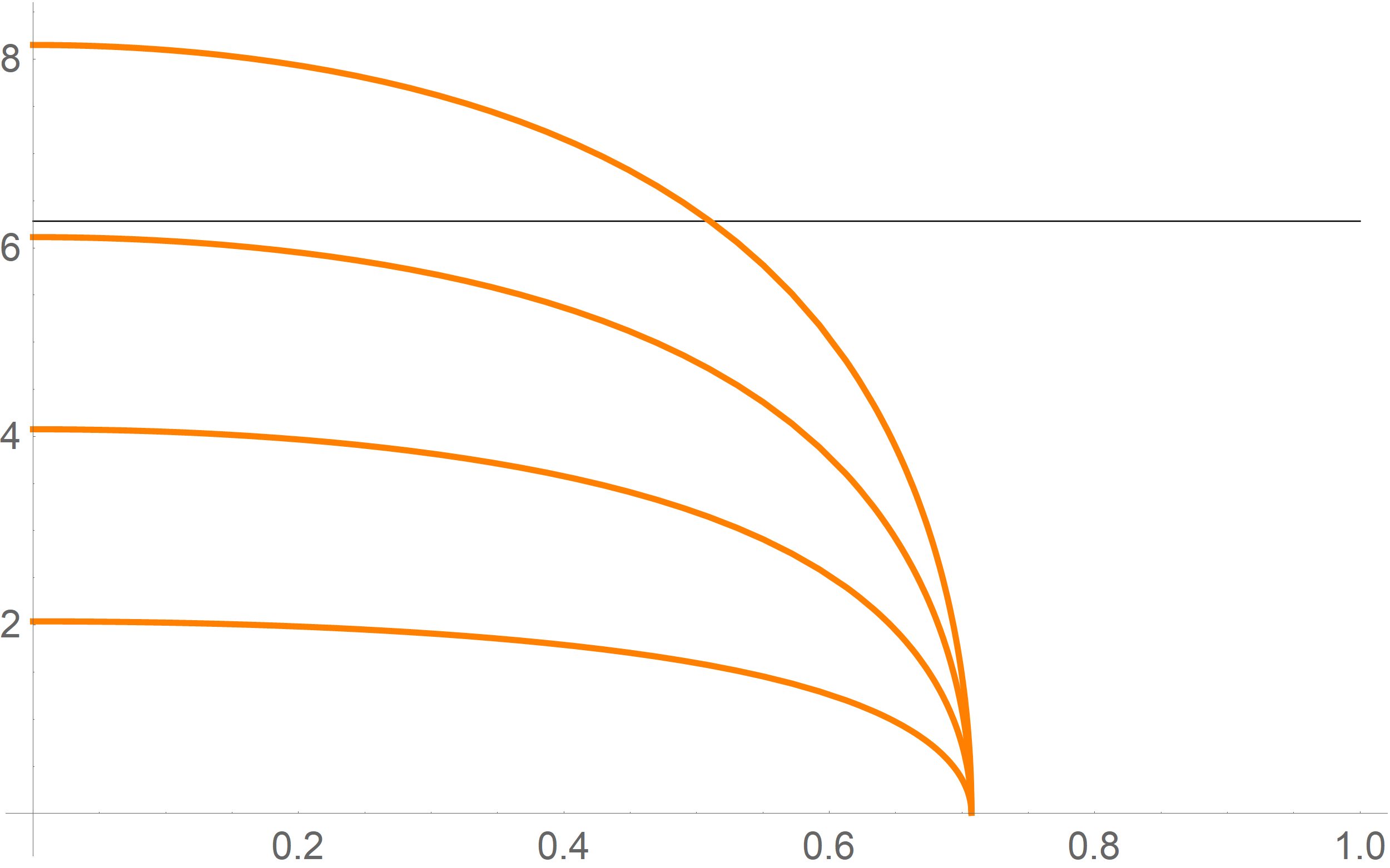}}
\put(2.2,3.35){$n=40$}
\put(7.6,2.9){$2\pi$}
\put(2.2,2.5){$n=30$}
\put(2.2,1.7){$n=20$}
\put(2.2,0.8){$n=10$}
\put(2,3.85){{\color{gray} $P(n,p)$}}
\put(8.1,0.175){{\color{gray} $p$}}
\end{picture}
\caption{The function $P$ in \eqref{eq:Period} is continuous in $p$ 
 and, for suitably large $n\in \N$, greater than $2\pi$ at $0$ with a 
 zero at $p=\sqrt{\tfrac{1}{1-K}}$ (here $K=-1$).
 Thus, by the intermediate value theorem, there is a unique solution
 $p$ of \eqref{eq:Period}.}
\label{fig:Period}%
\end{figure}
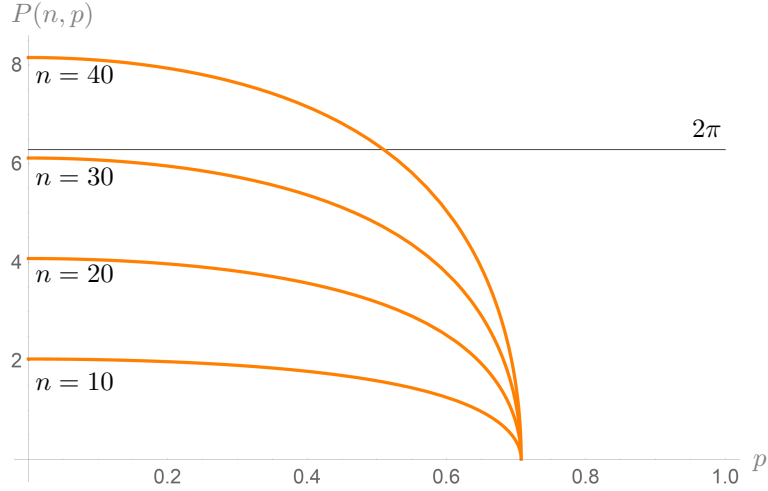

\begin{bsp}\label{exp:ClosedCurves}
Since we obtained explicit parametrisations of all CGC surfaces of 
hyperbolic rotation, it is easy to prove that there are periodic 
surfaces in this class, as we will prove for the case $K=-1$: the 
\emph{complete elliptic integral of third kind (with modulus $p$ 
and parameter $k$)} is defined as
\begin{align*}
\Pi^k_p :=\int_0^{\tfrac{\pi}{2}}\tfrac{1}{1-k\sin^2(u)}\tfrac{du}{\sqrt{1-p^2\sin^2(u)}}.
\end{align*}
For $k=0$ this coincides with the \emph{complete elliptic integral 
of first kind $F_p$} (see \Cref{app:JacobiFunctions}). The elliptic 
function $\jac{dn}{}$ is periodic (see \cite[Sect 16]{abramowitz1972}),
its period being $2F_p$.
Thus the functions $r$ and $\psi'$ from 
\Cref{tab:HyperbolicCaseHyperbolic} for the case $K<0$ are 
periodic with period
\begin{align*}
\pi_p := 2\tfrac{F_p}{\FAC}.
\end{align*}
Of course, $\psi$ is not periodic, but $\cos\psi$ is for any 
$X$ such that
\begin{align*}
\psi(s+X) = \psi(s) + 2\pi.
\end{align*}
Since
\begin{align*}
\Pii{k}{p}{s+n F_p} = \Pii{k}{p}{s} + n\Pi^k_p, 
\end{align*}
we have
\begin{align*}
\psi(s+n\pi_p) = \psi(s) + \tfrac{nK}{\FAC}\left(\Pi^k_p - F_p\right)~
\textrm{where}~k=\tfrac{1}{1-K},
\end{align*}
for any integer $n$. Therefore, the profile curve $\c(t)= \f(t,0)$ of 
a CGC $K<0$ surface is periodic if there exists 
$p \in \left[0,\sqrt{\tfrac{1}{1-K}}\,\right]$ such that 
\begin{align}\label{eq:Period}
P(n,p):=\tfrac{nK}{\FAC}\left(\Pi^k_p - F_p\right) = 2\pi. 
\end{align}
Note that $P$ vanishes at $p=\sqrt{\tfrac{1}{1-K}}$ for all $n$ 
($\FAC$ has a pole there) and, since $K<0$ and $\Pi^k_p<F_p$, is 
positive at $p=0$. Thus, for a suitably large $n \in \N$, 
\eqref{eq:Period} has a solution $p$ (which is also unique, see 
\Cref{fig:Period}), so that the corresponding profile curve
$\c$ is periodic.
Such 
a closed profile curve is displayed in \Cref{fig:Curves_H3_hyp_Neg},
and \Cref{fig:Hyp_Neg} shows the corresponding 
surface of hyperbolic rotation. 
\end{bsp}

\begin{figure}[b]%
\centering
\subfigure[][$K=-1$]{%
\label{fig:Curves_H3_hyp_Neg}%
\includegraphics[width=\columnwidth/5]{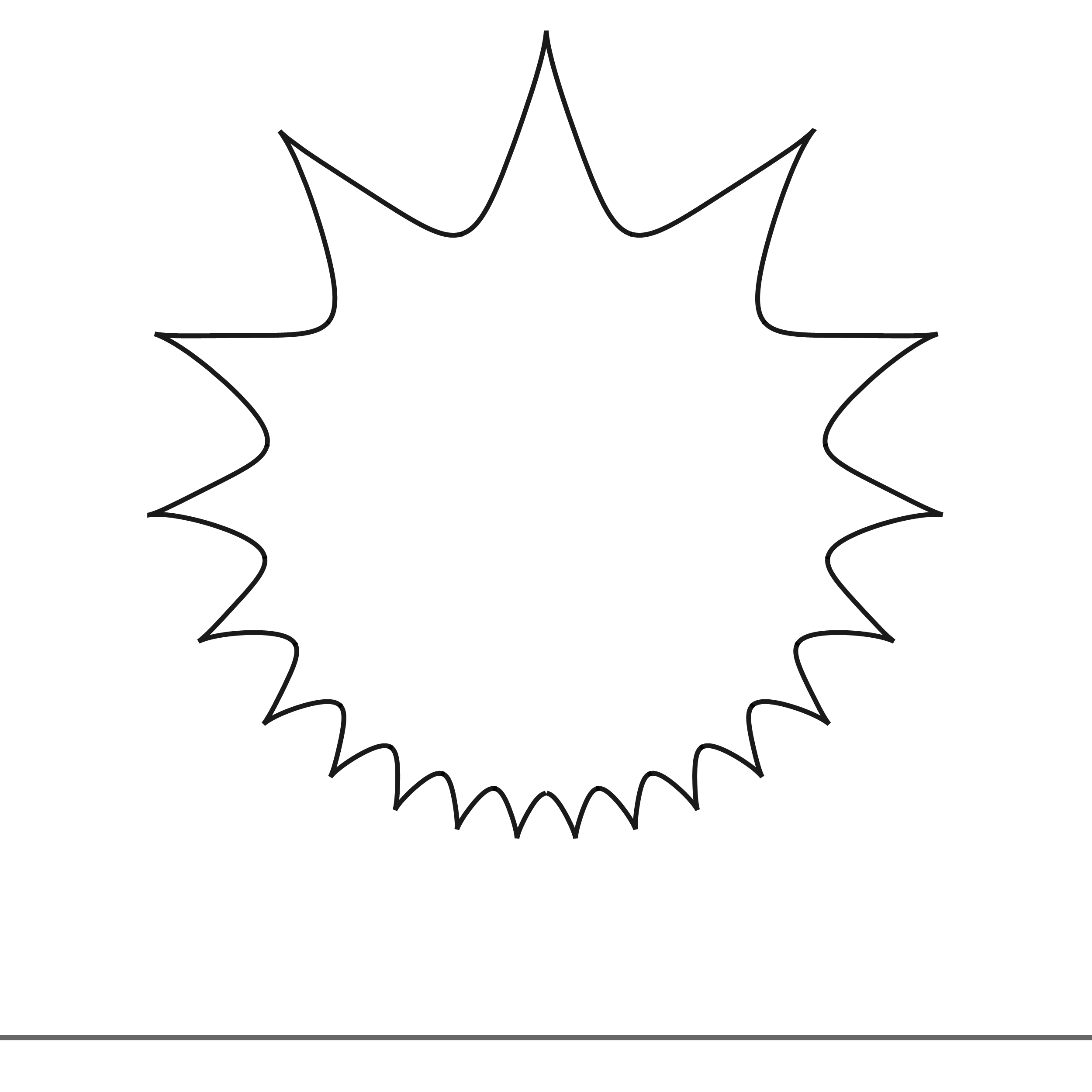}%
}%
\hfill 
\subfigure[][$K=0.4$]{%
\includegraphics[width=\columnwidth/5]{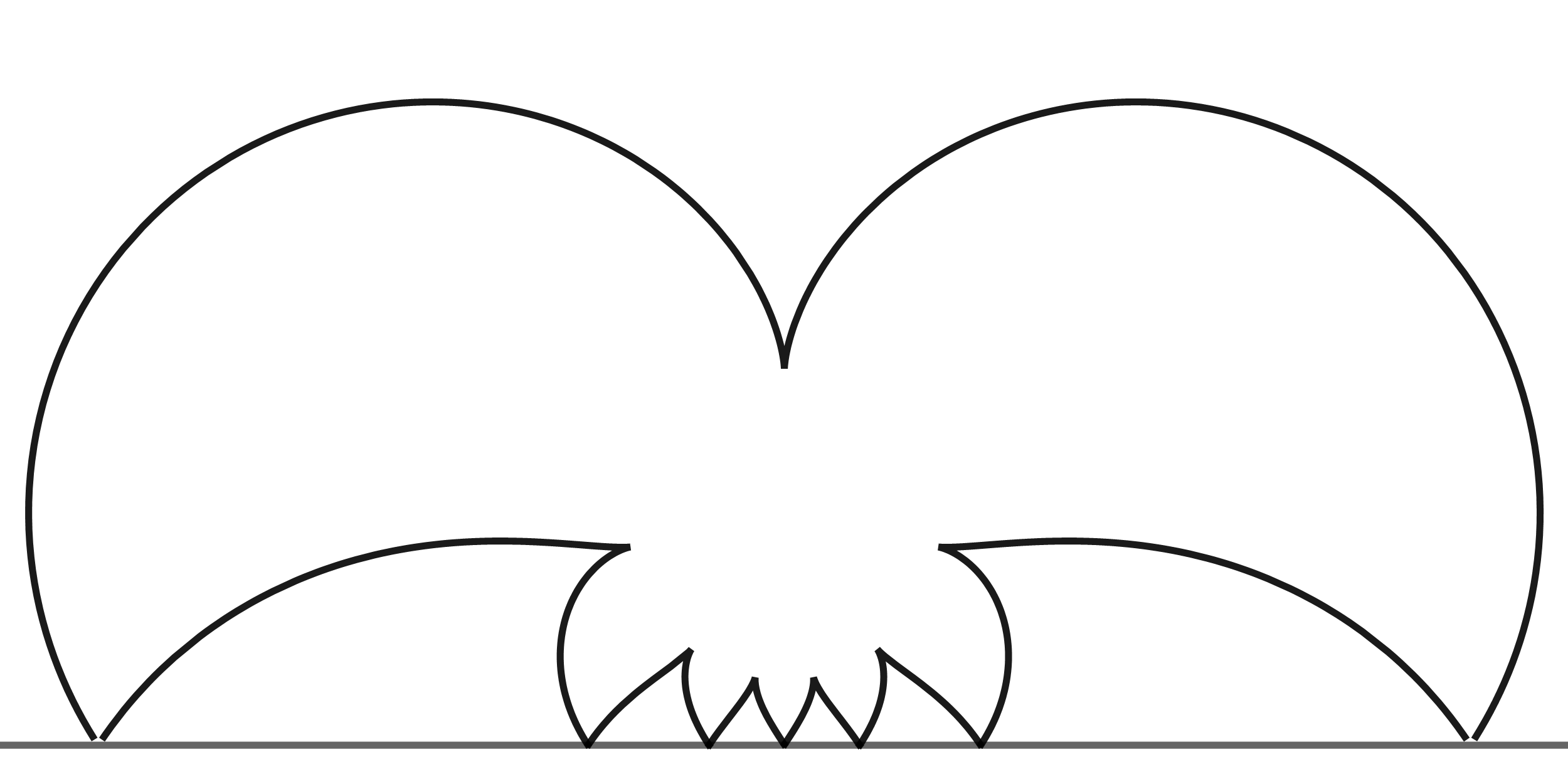}%
}% 
\hfill
\subfigure[][$K=0.4$]{%
\label{fig:Curves_H3_hyp_Mix}%
\includegraphics[width=\columnwidth/5]{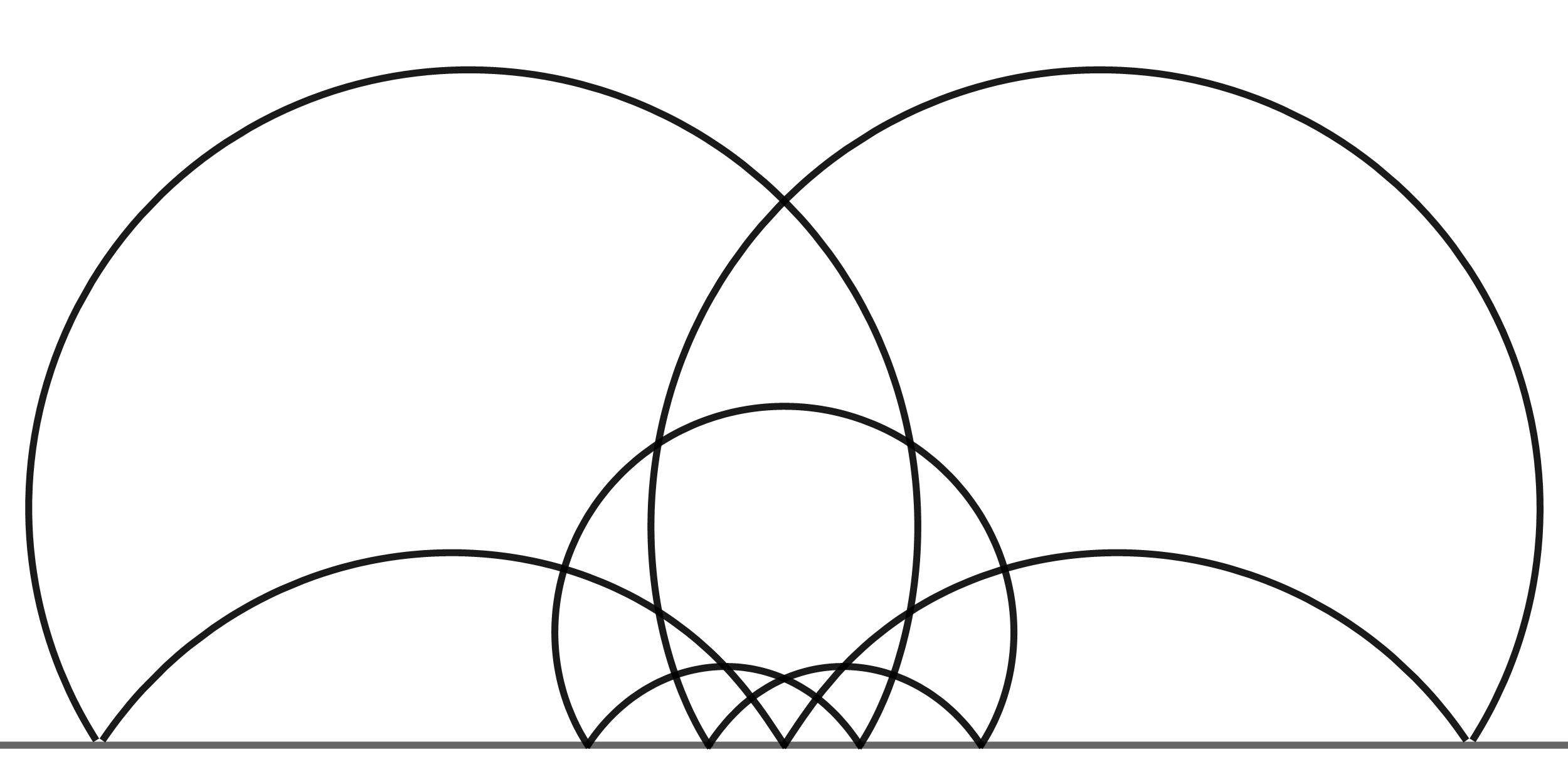}%
}% 
\hfill
\subfigure[][$K=2$]{%
\label{fig:Curves_H3_hyp_Pos}%
\includegraphics[width=\columnwidth/5]{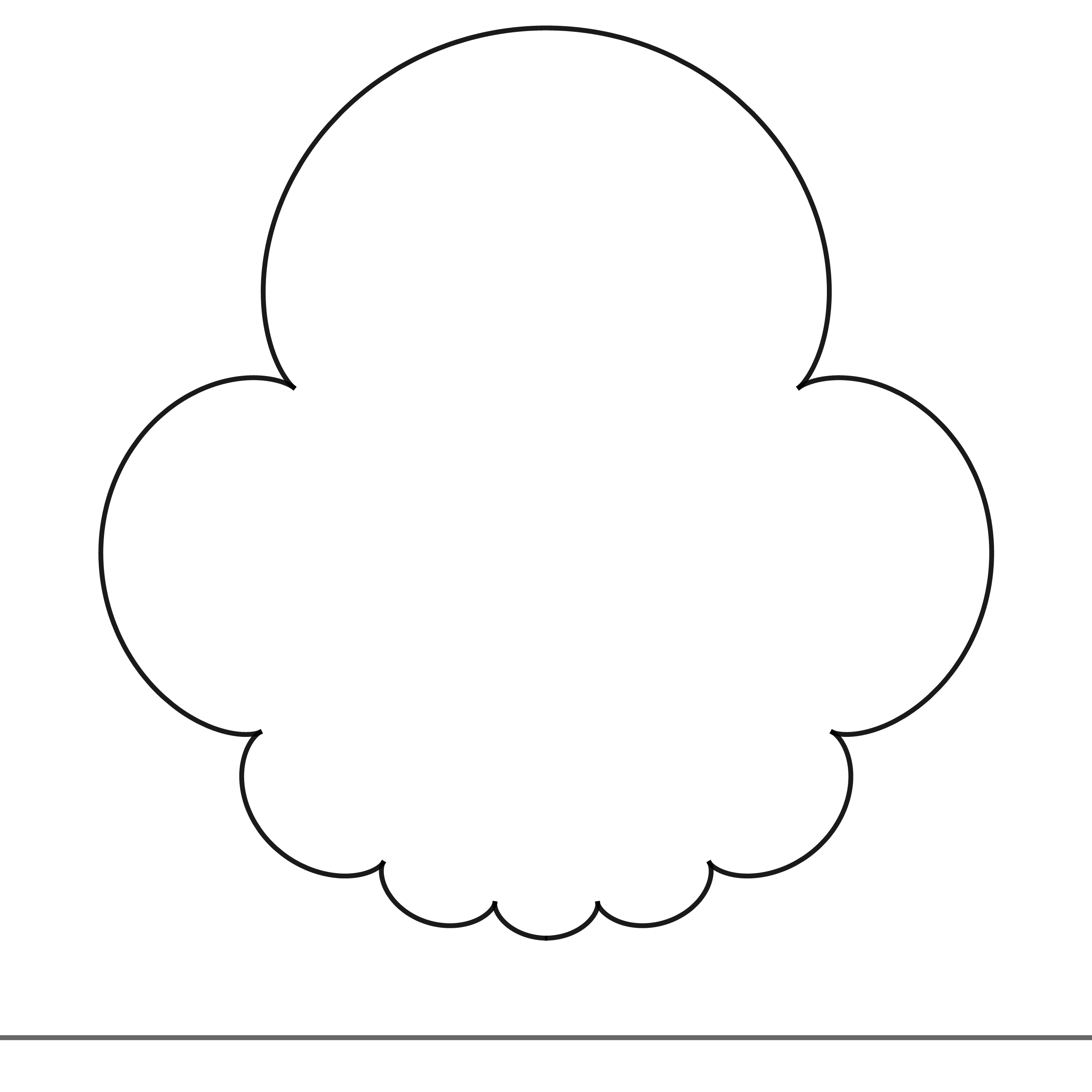}%
}%
\caption{Profile curves of CGC
 surfaces of hyperbolic rotation in the half plane model:
 for an appropriate choice of parameter, the solutions given
 in \Cref{tab:HyperbolicCaseHyperbolic} yield closed
 profile curves, see Example \ref{exp:ClosedCurves}. In
 the case $K=0.4$ the profile curves are obtained by the
 second and fourth solutions in the table. The surfaces
 obtained by hyperbolic rotation of the profile curves in Figures
 \ref{fig:Curves_H3_hyp_Neg}, \ref{fig:Curves_H3_hyp_Mix} and
 \ref{fig:Curves_H3_hyp_Pos}
 are shown in Figures \ref{fig:Hyp_Neg}, \ref{fig:Hyp_Mix},
 and \ref{fig:Hyp_Pos}, respectively.}%
\label{fig:Curves_H3_hyp}%
\end{figure}

\subsection{CGC surfaces of parabolic rotation in \texorpdfstring{$\H^3$}{H3}} 
For surfaces of parabolic rotation in $\H^3$, we equip 
${\{v\in \R^{4,2}| (v,\q)=-1,~(v,\p)=0\}\cong \R^{3,1}}$ with the 
pseudo-orthonormal basis $(\v_1, \v_2, \e_1, \e_2)$ as in the 
beginning of Subsection \ref{sec:Parb}. In that subsection, we gave formulas 
for the coordinate functions of parabolic rotational surfaces of 
constant Gauss curvature in hyperbolic spaces of arbitrary (negative) 
curvature. For $\kappa=-1$, the solutions given in 
\eqref{eq:ParabSolR} are real if
\begin{itemize}
\item $K-1\geq 0$ with $C\geq 0$,
\item $K<0$ with $C\leq 0$, or
\item $K\in (0,1)$ without restrictions on $C$.
\end{itemize}
In the last case, the solutions given in \eqref{eq:ParabSolR} are
real functions, though a Jacobi transformation has to be applied
for this fact to manifest itself.
The respective (real) solutions are listed
in \Cref{tab:HyperbolicCaseParabolic} and complement
the following classification theorem:

\begin{thm}\label{thm:Parabolic}
Every constant Gauss curvature $K\neq 0$ surface of parabolic 
rotation in $\H^3\subset \R^{3,1}$ is given, in terms of a pseudo-orthonormal 
basis, by 
\begin{align*}
(s, \theta) \mapsto 
\left(\tfrac{r^2(s)\left(\theta^2+\psi^2(s)\right)+1}{2r(s)}, 
r(s) ,~r(s)\theta,~r(s)\psi(s)\right),
\end{align*}
where $r$, $\psi$ are as listed in \Cref{tab:HyperbolicCaseParabolic}.
\end{thm}

\begin{bem}
In the third case, $K\in (0,1)$, the solution $\psi$ of \eqref{eq:ParabSolPsi}
comes with an imaginary argument, hence a Jacobi transformation needs to be
applied, see \Cref{app:JacobiFunctions}. The effect is that
$\psi$ then contains an \emph{elliptic integral of the second 
kind with modulus $p$}, denoted by $E_p$ and defined as
\begin{align*}
E_p(s) = \int_0^s \sqrt{1-p^2 \sin^2(u)} du.
\end{align*}
\end{bem}

\begin{figure}[t]%
\centering
\subfigure[][$K=-1$]{%
\label{fig:Curves_H3_par_Neg}%
\includegraphics[width=\columnwidth/5]{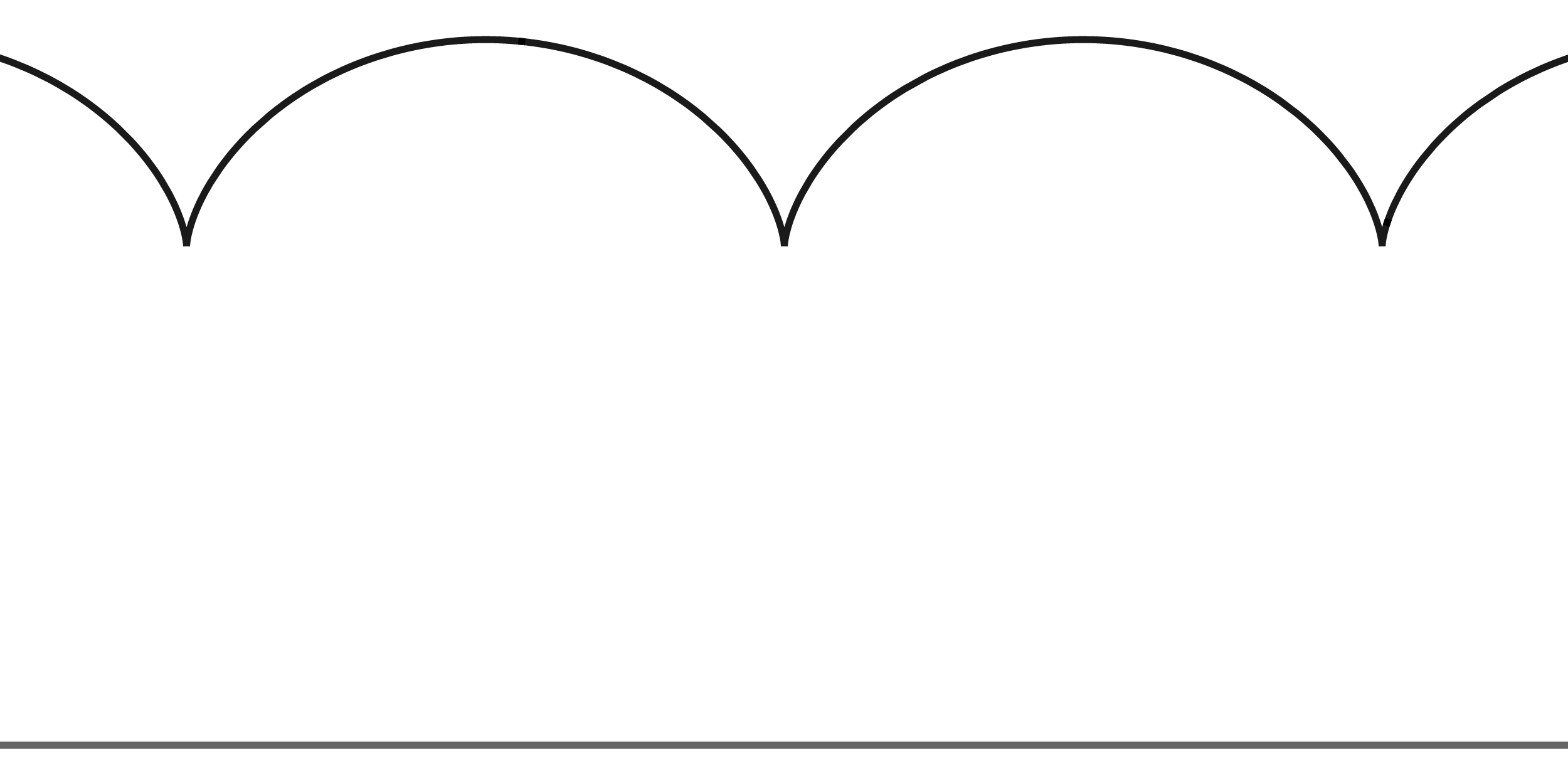}%
}%
\hfill 
\subfigure[][$K=0.4$, first $\jac{nc}{}$-solution]{%
\label{fig:Curves_H3_par_Mix}%
\includegraphics[width=\columnwidth/5]{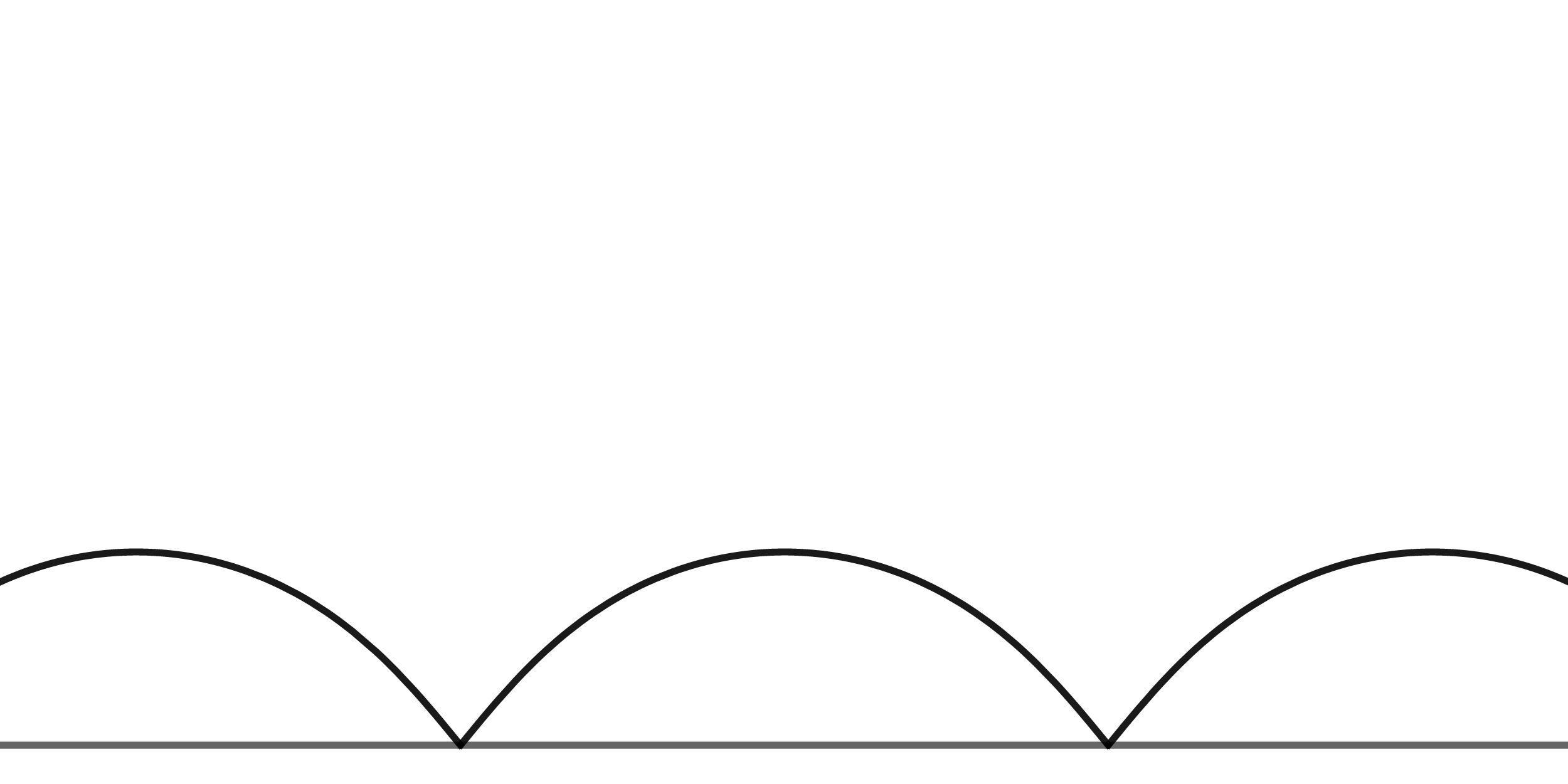}%
}% 
\hfill
\subfigure[][$K=0.4$, second $\jac{nc}{}$-solution]{%
\includegraphics[width=\columnwidth/5]{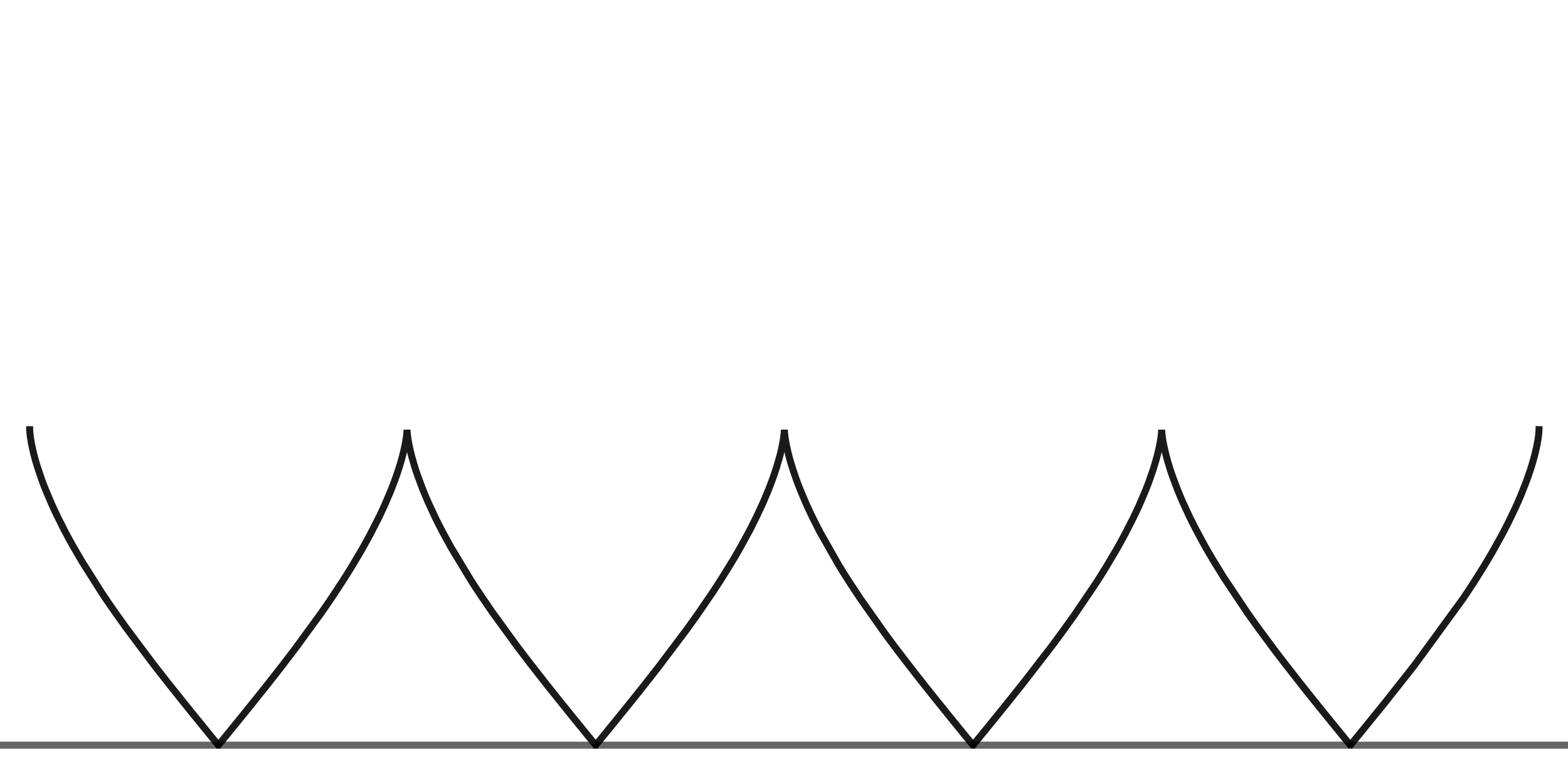}%
}% 
\hfill
\subfigure[][$K=2$]{%
\label{fig:Curves_H3_par_Pos}%
\includegraphics[width=\columnwidth/5]{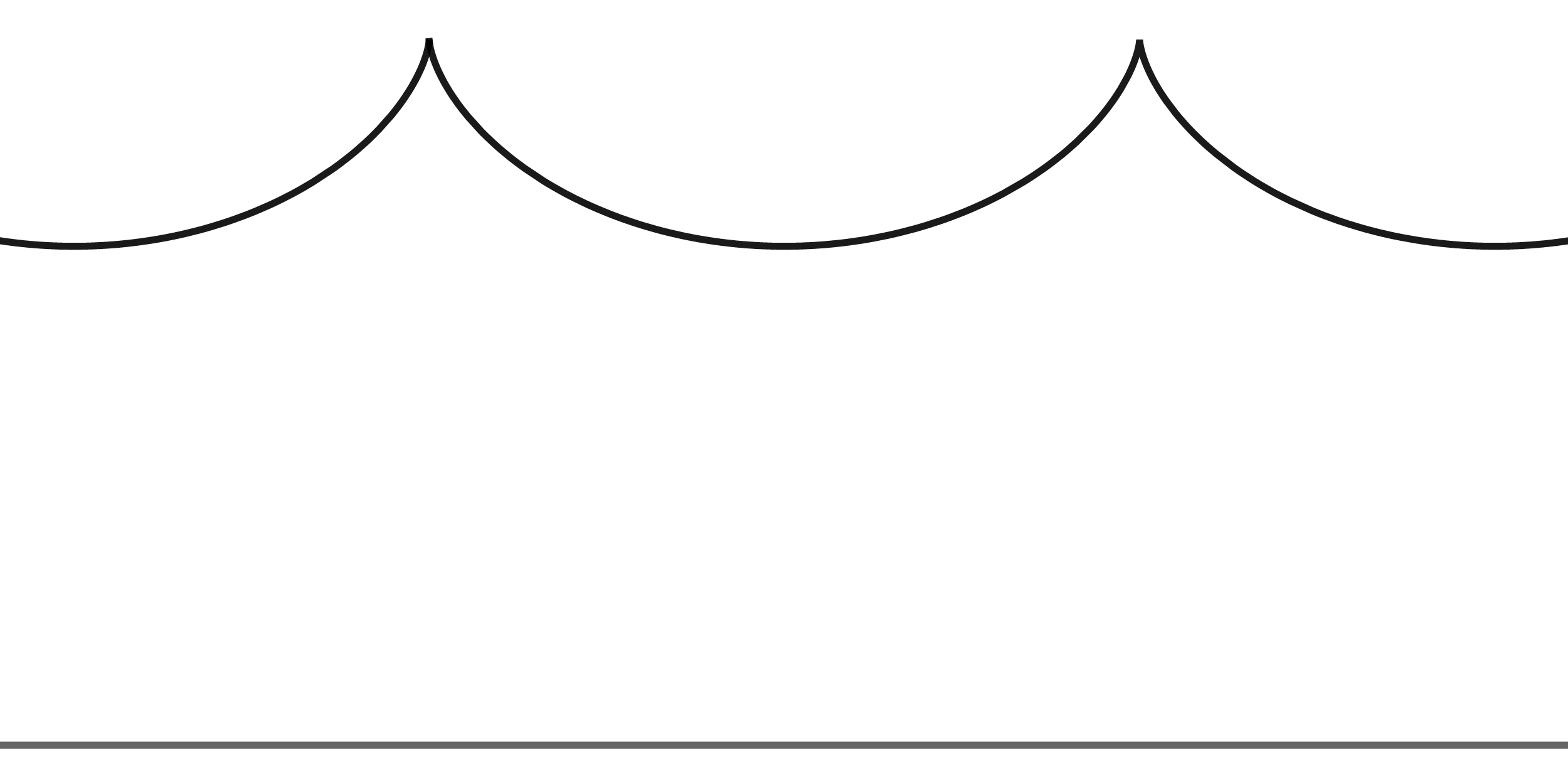}%
}%
\caption{Profile curves of CGC
 surfaces of parabolic rotation in the half plane model:
 the profile curves are obtained using the solutions
 in \Cref{tab:HyperbolicCaseParabolic}. In the Poincar\'e
 half space model, the parabolic rotation appears as a
 translation along the direction perpendicular to the half
 plane.
 In Figures \ref{fig:Par_Neg}, \ref{fig:Par_Mix}, and \ref{fig:Par_Pos},
 the surfaces obtained from the profile curves
 in Figures \ref{fig:Curves_H3_par_Neg}, \ref{fig:Curves_H3_par_Mix},
 and \ref{fig:Curves_H3_par_Pos}, respectively,
 are shown in the Poincar\'{e} ball model.}%
\label{fig:Curves_H3_par}%
\end{figure}

\begin{table}[ht]
\centering
\begin{tabular}{l|ll}
$K<0$  
&$r(s)=\sqrt{\frac{C}{K}}~\jac{dn}{p}\left(\FAC~s\right)$ 
&$\psi(s)=K~s-\tfrac{K}{\FAC}~\Pii{p^2}{p}{\FAC~s}$ \cr
&\quad $\FAC=\sqrt{(K-1)C}$, with $C\leq 0$, &$p=\sqrt{\tfrac{1}{1-K}}$ \cr
\hline
$K\in (0,1)$ 
&$r(s)=\sqrt{\tfrac{C}{K-1}}\jac{nc}{p}(\FAC~s)$
&$\psi(s)=\tfrac{1}{\FAC}~\left(E_p\circ\jac{am}{p}\right)(\FAC~s)$ \cr 
&\quad $\FAC=\sqrt{-C}$ with $C<0$, &$p=\sqrt{1-K}$ \cr
&$r(s)=\sqrt{\tfrac{C}{K}}\jac{nc}{p}(\FAC~s)$
&$\psi(s)=s - \tfrac{1}{\FAC}~\left(E_p\circ\jac{am}{p}\right)(\FAC~s)$, \cr 
&\quad $\FAC=\sqrt{C}$, with $C>0$, &$p=\sqrt{K}$  \cr
\hline
$K=1$ 
&$r(s)=p~\cosh~\left(ps\right)$ &$\psi(s)=s+\tfrac{\tanh(ps)}{p}$ \cr
&\quad $p\in (0, \infty)$ & \cr
\hline
$K-1>0$ 
&$r(s)=\sqrt{\frac{C}{K-1}}\jac{dn}{p}\left(\FAC~s\right)$
&$\psi(s)=K~s-\tfrac{K-1}{\FAC}~\Pii{p^2}{p}{\FAC~s}$\cr 
&\quad $\FAC=\sqrt{KC}$, with $C\geq 0$, &$p=\sqrt{\tfrac{1}{K}}$
\end{tabular}
\caption{Parameter functions of CGC surfaces of parabolic rotation in \texorpdfstring{$\H^3$}{H3}.}
\label{tab:HyperbolicCaseParabolic}
\end{table}

\begin{figure}%
	\centering
	\subfigure[][Elliptic rotation, ${K=-1}$, corresponds to the profile curve given in Figure \ref{fig:Curves_H3_ell_NegCN}.]{%
		\label{fig:Ell_Neg}
		\includegraphics[width=\columnwidth/5]{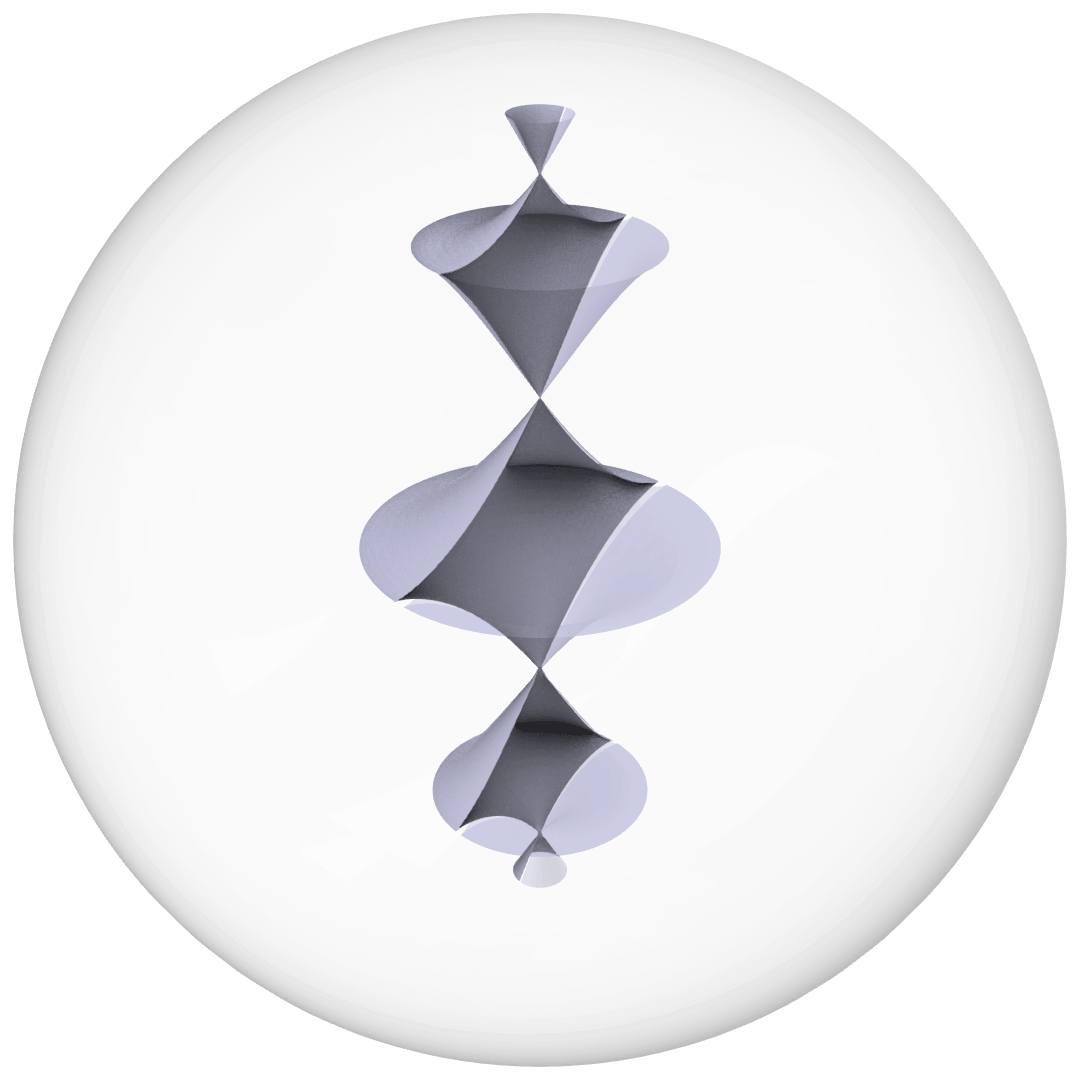}%
	}%
	\hfill 
		\subfigure[][Elliptic rotation, ${K=0.4}$, corresponds to the profile curve given in Figure \ref{fig:Curves_H3_ell_MixSC}.]{%
				\label{fig:Ell_Mix}
		\includegraphics[width=\columnwidth/5]{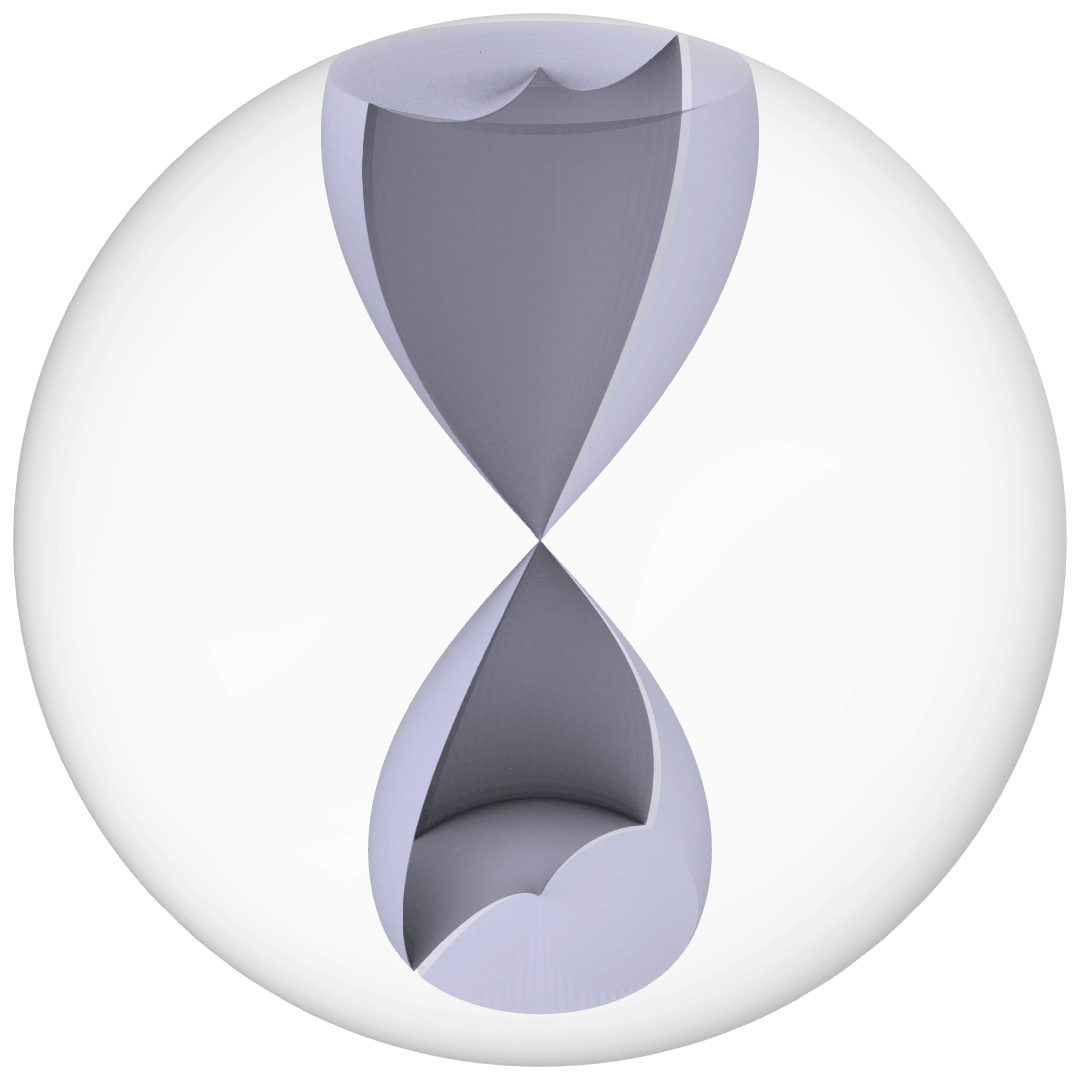}%
	}% 
	\hfill
		\subfigure[][Elliptic rotation, ${K=2}$, corresponds to the profile curve given in Figure \ref{fig:Curves_H3_ell_PosCN}.]{%
		\label{fig:Ell_Pos}
		\includegraphics[width=\columnwidth/5]{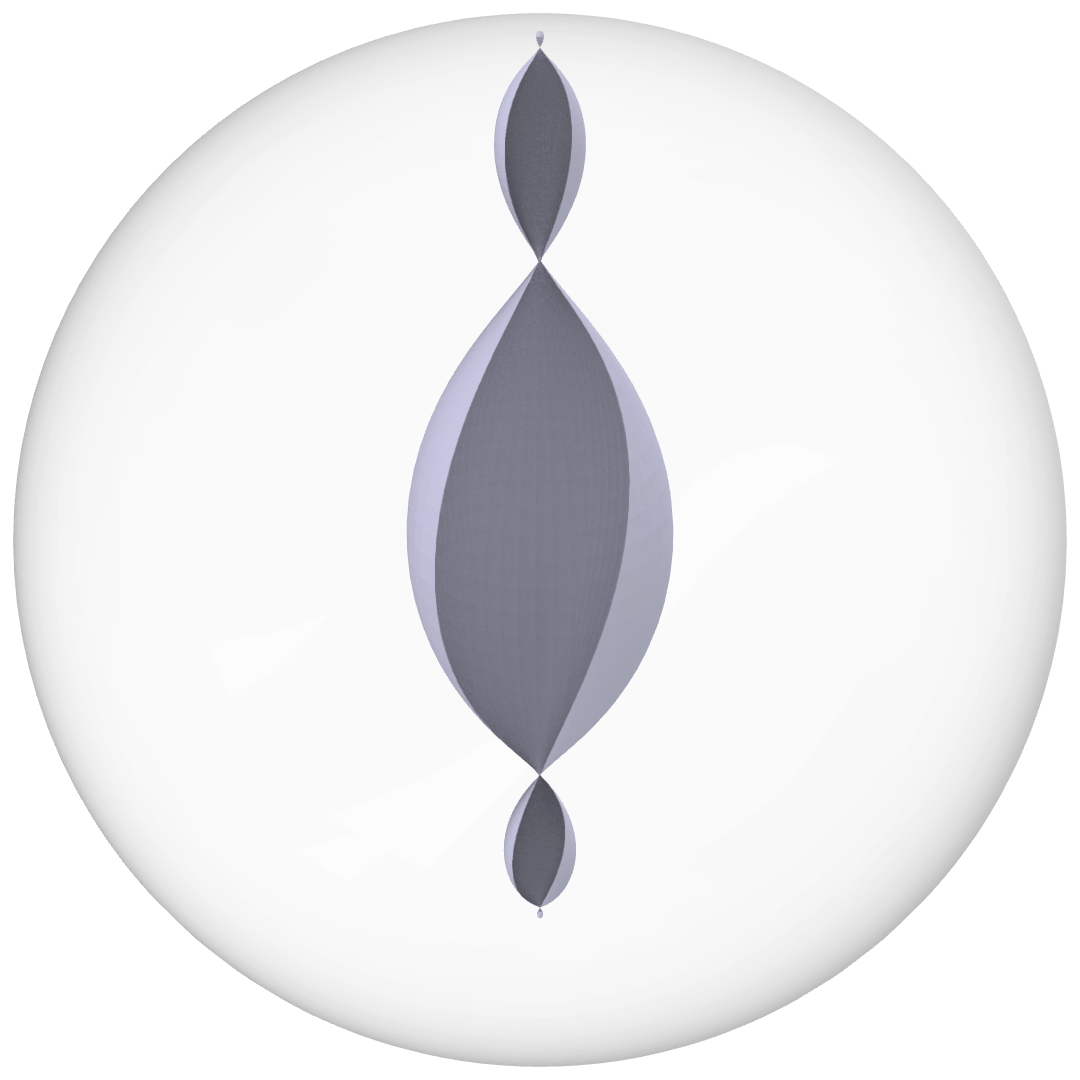}%
	}% 
	\\
	\subfigure[][Hyperbolic rotation, ${K=-1}$, corresponds to the profile curve given in Figure \ref{fig:Curves_H3_hyp_Neg}.]{%
		\label{fig:Hyp_Neg}
		\includegraphics[width=\columnwidth/5]{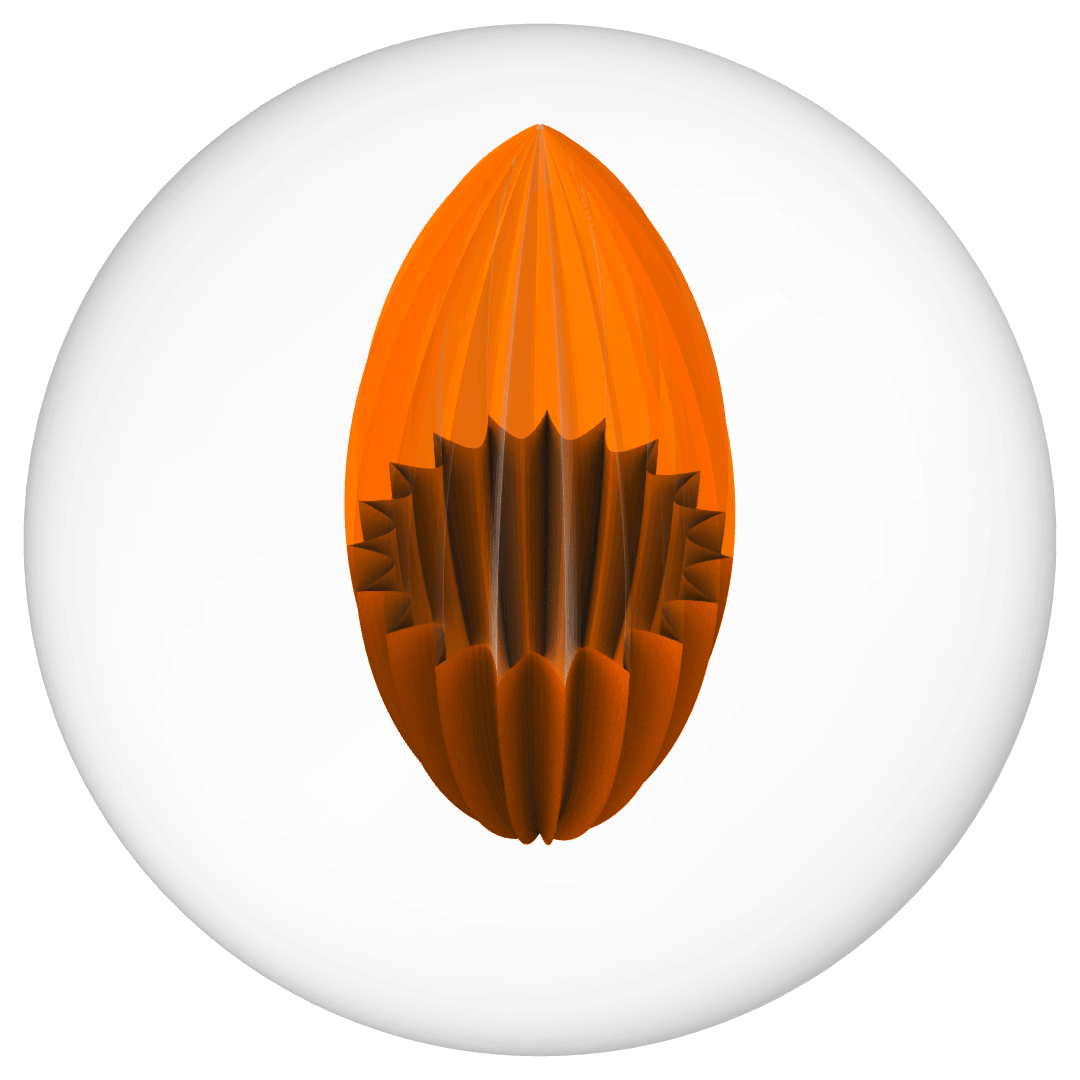}%
	}%
	\hfill
		\subfigure[][Hyperbolic rotation, ${K=0.4}$, corresponds to the profile curve given in Figure \ref{fig:Curves_H3_hyp_Mix}.]{%
		\label{fig:Hyp_Mix}%
		\includegraphics[width=\columnwidth/5]{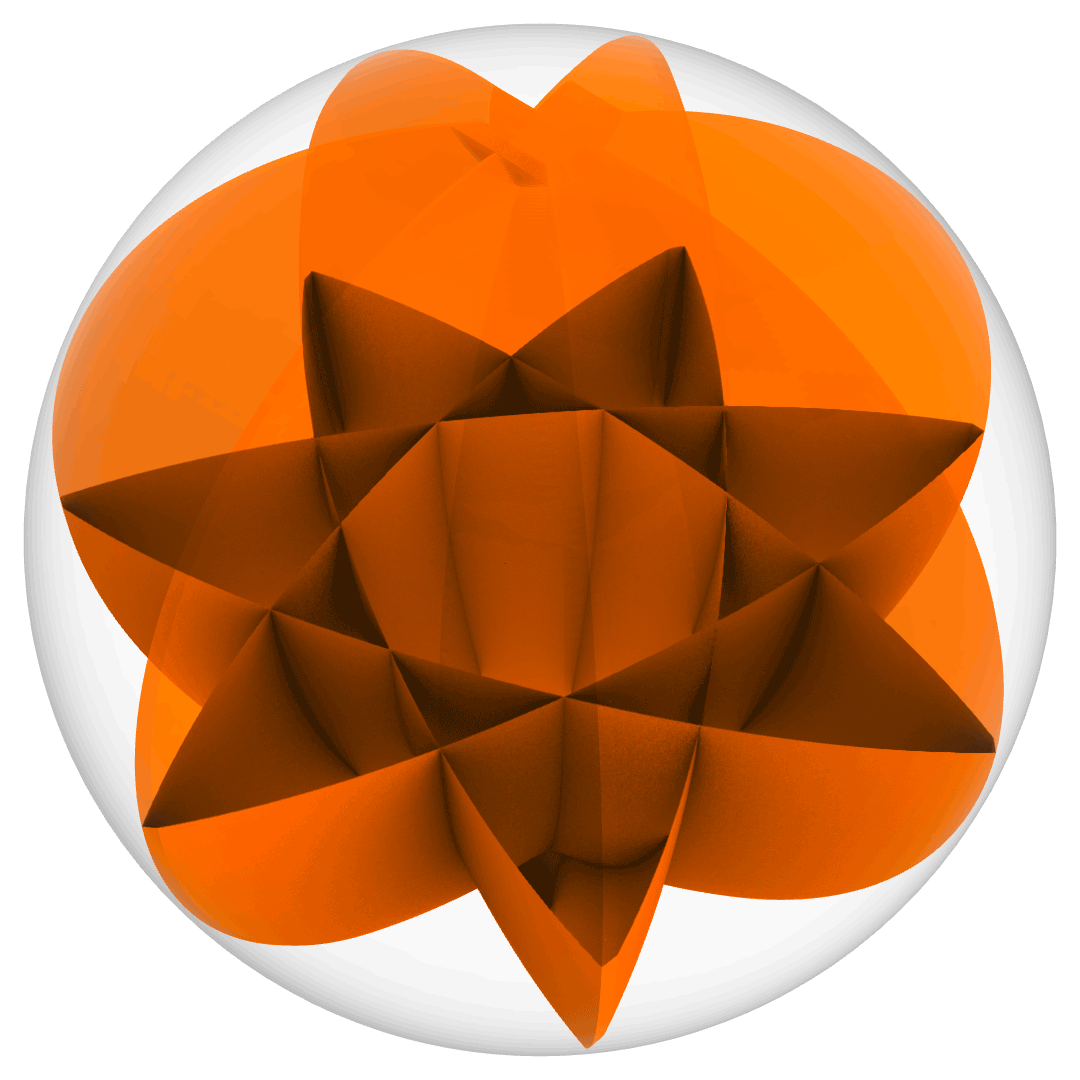}%
	}% 
	\hfill
		\subfigure[][Hyperbolic rotation, ${K=2}$, corresponds to the profile curve given in Figure \ref{fig:Curves_H3_hyp_Pos}.]{%
		\label{fig:Hyp_Pos}%
		\includegraphics[width=\columnwidth/5]{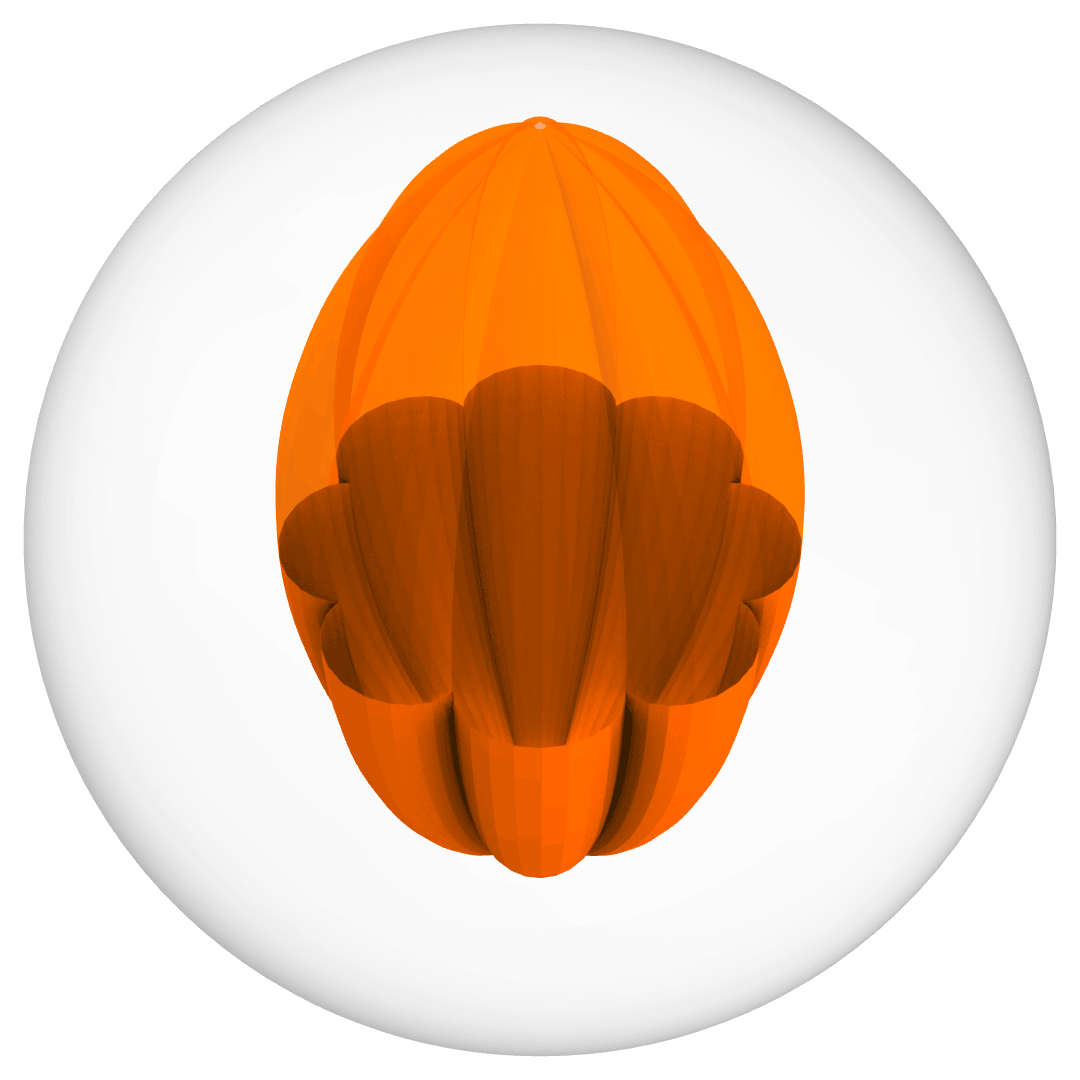}%
	}% 
	\\
	\subfigure[][Parabolic rotation, ${K=-1}$, corresponds to the profile curve given in Figure \ref{fig:Curves_H3_par_Neg}.]{%
		\label{fig:Par_Neg}%
		\includegraphics[width=\columnwidth/5]{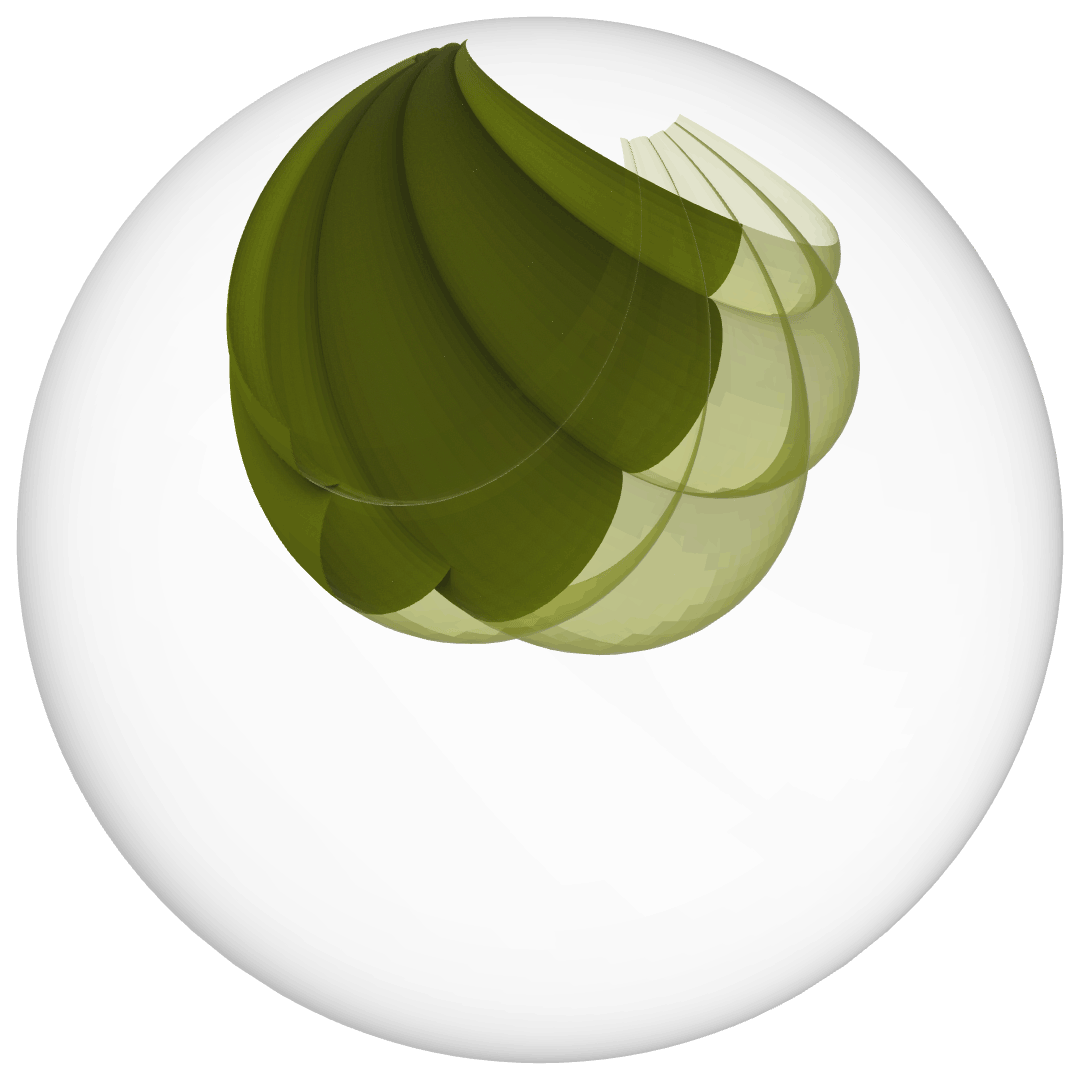}%
	}%
	\hfill
		\subfigure[][Parabolic rotation, ${K=0.4}$, corresponds to the profile curve given in Figure \ref{fig:Curves_H3_par_Mix}.]{%
		\label{fig:Par_Mix}%
		\includegraphics[width=\columnwidth/5]{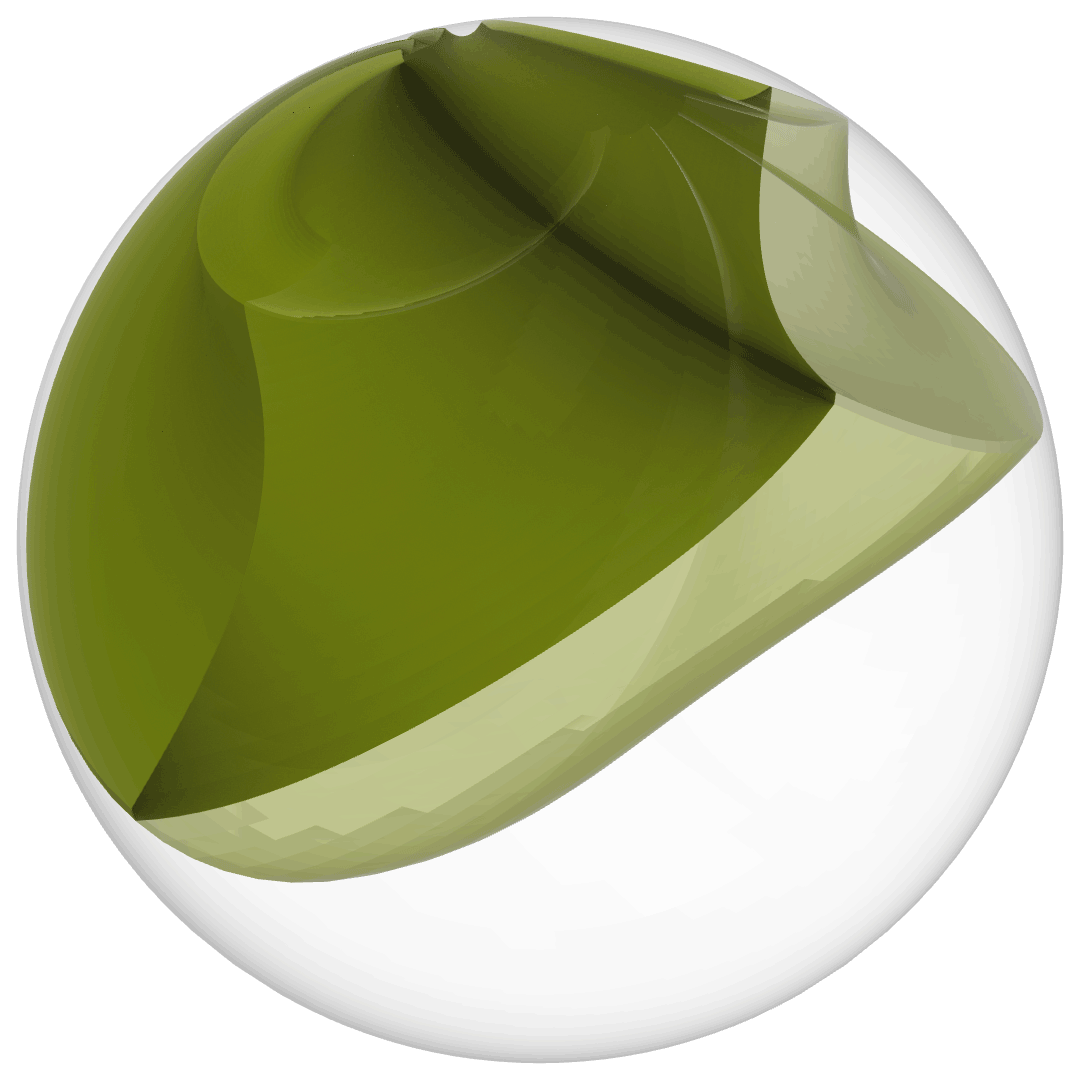}%
	}% 
	\hfill
		\subfigure[][Parabolic rotation, ${K=2}$, corresponds to the profile curve given in Figure \ref{fig:Curves_H3_par_Pos}.]{%
		\label{fig:Par_Pos}%
		\includegraphics[width=\columnwidth/5]{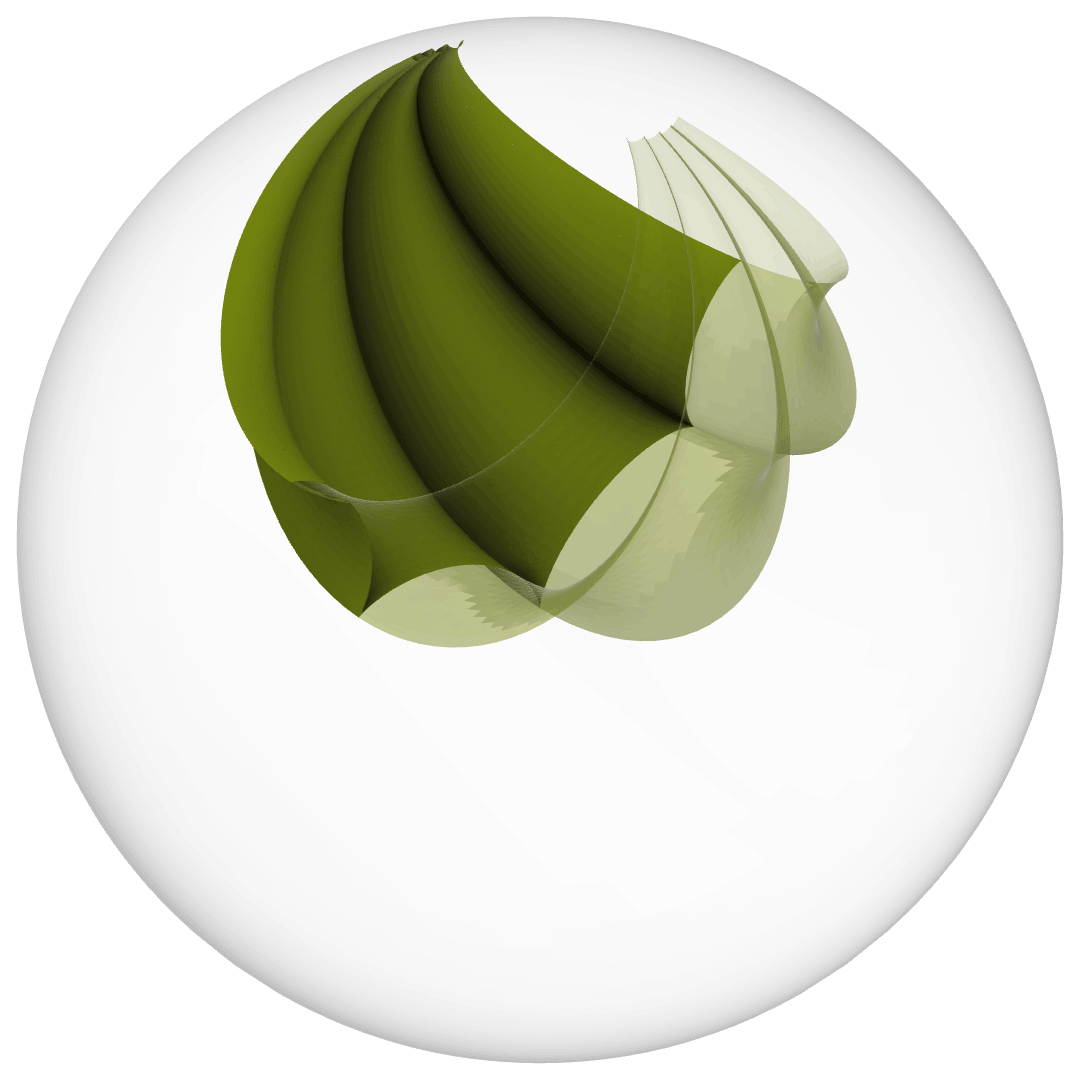}%	
		}
\caption{Surfaces of different Gauss curvatures and types of
rotation in the Poincar\'e ball model.}%
	\label{}%
\end{figure}

The Theorems \ref{thm:Elliptic}, \ref{thm:Hyperbolic} and
\ref{thm:Parabolic} provide explicit parametrisations of all 
rotational CGC surfaces in $\H^3$. 
Moreover, via parallel transformations, we may obtain 
parametrisations of all rotational linear Weingarten surfaces in 
the parallel families of a rotational CGC surface. Thus, 
\Cref{thm:CLWisRot} and \Cref{prop:Bonnet} lead to the following 
theorem. 

\begin{thm}
Every channel linear Weingarten surface in hyperbolic space 
$\H^3$, satisfying
\begin{align*}
 aK + 2bH + c =0 ~\text{with}~ \left|\tfrac{a+c}{2}\right|>|b|
\end{align*}
is parallel to a 
rotational surface parametrised by one of the parametrisations 
given in Theorems \ref{thm:Elliptic}, \ref{thm:Hyperbolic}, or 
\ref{thm:Parabolic}.
\end{thm}

\newpage

\appendix

\section{Jacobi elliptic functions}\label{app:JacobiFunctions}
%%%%%%%%%%%%%%%%%%%%%%%%%%%%%%%%%%%%%%%%%%%%%%%%%%%%%%%%%%%%%%%%%%%%%%%%
We gather some results and transformation formulas for Jacobi elliptic 
functions and elliptic integrals. For details see 
\cite[Chap 63]{spanier1987} and \cite[Chap 16 and 17]{abramowitz1972}.

\subsection{Elliptic functions}
The \emph{Jacobi elliptic functions of pole type $n$} may be given for 
a \emph{modulus} $p\in [0,1]$ by their characterising elliptic 
differential equations
\begin{align*}
 &y'^2(s)=(1-y^2(s))(1-p^2y^2(s))~ \Rightarrow ~y(s)=\jac{sn}{p}(s), \\
 &y'^2(s)=(1-y^2(s))(q^2+p^2y^2(s))~\Rightarrow~y(s)=\jac{cn}{p}(s), \\
 &y'^2(s)=(1-y^2(s))(-q^2 + y^2(s))~\Rightarrow~ y(s) = \jac{dn}{p}(s),
\end{align*}
where $q = \sqrt{1-p^2}$ is called the \emph{complementary modulus}. 
The \emph{Jacobi amplitude function} $\jac{am}{p}$ may be defined via
$\jac{am}{p}(s) = \operatorname{arcsin}\jac{sn}{p}(s) =
\operatorname{arccos}\jac{cn}{p}(s)$, then the characterising 
differential equations imply $\jac{am}{p}' = \jac{dn}{p}$. Further, 
we obtain the \emph{Pythagorean laws}
\begin{align*}
	\jac{sn}{p}^2+\jac{cn}{p}^2=\jac{dn}{p}^2+p^2\jac{sn}{p}^2=1.
\end{align*}
A wider class of Jacobi elliptic functions may be defined via algebraic
combinations of the three functions of pole type $n$:
\begin{align*}
 \jac{ef}{p}(s):=\tfrac{\jac{en}{p}(s)}{\jac{fn}{p}(s)}~%
 \textrm{and}~\jac{ne}{p}(s) = \tfrac{1}{\jac{en}{p}(s)},~%
 \textrm{for}~e,f \in \{c,d,s\}.
\end{align*}

The Jacobi elliptic functions take complex arguments: purely imaginary 
arguments are evaluated using \emph{Jacobi's imaginary transformations} 
\begin{align*}
 \jac{sn}{p}(is)=i\jac{sc}{q}(s),\quad\jac{cn}{p}(is)=\jac{nc}{q}(s),%
 \quad\jac{dn}{p}(is) = \jac{dc}{q}(s).
\end{align*}
Note that $\jac{cn}{}$ and $\jac{dn}{}$ are real functions of imaginary
arguments, whereas $\jac{sn}{}$ becomes imaginary. 

The restriction $p\in[0,1]$ on the modulus $p$ can be lifted
by means of \emph{Jacobi's real transformations}:
\begin{align*}
 \jac{sn}{\tfrac{1}{p}}(s)=p\jac{sn}{p}\left(\tfrac{s}{p}\right),\quad 
 \jac{cn}{\tfrac{1}{p}}(s)=\jac{dn}{p}\left(\tfrac{s}{p}\right),\quad
 \jac{dn}{\tfrac{1}{p}}(s)=\jac{cn}{p}\left(\tfrac{s}{p}\right),
\end{align*} 
allow evaluation for $p\geq 1$, and by definition, all Jacobi elliptic 
functions are even with respect to their modulus.
Also, with respect to an imaginary modulus $ip \in i \R$,
the Jacobi elliptic functions take real values for a real
argument: the corresponding transformations read
\begin{align*}
 \begin{array}{l}
  \jac{sn}{ip}(s)=q'\jac{sd}{p^\prime}%
   \left(s\sqrt{1+p^2}\right),\vspace{5pt}\\
  \jac{cn}{ip}(s)=\jac{cd}{p^\prime}%
   \left(s\sqrt{1+p^2}\right)\vspace{5pt},\\ 
  \jac{dn}{ip}(s)=\jac{nd}{p^\prime}\left(s\sqrt{1+p^2}\right).
 \end{array}~\textrm{with}~p^\prime = \tfrac{p}{\sqrt{1+p^2}},
\end{align*}

\subsection{Elliptic integrals}
The elliptic integrals are closely related to Jacobi's elliptic 
functions. According to \cite[Chap 17]{abramowitz1972} the 
\emph{(incomplete) elliptic integral of first, second and third 
kind}, denoted by $F$, $E$ and $\Pi$, respectively, is
\begin{align*}
 F(s,p) &= \int_0^s \tfrac{du}{\sqrt{1-p^2\sin^2(u)}}, \\
 E(s,p) &= \int_0^s \sqrt{1-p^2\sin^2(u)}~du, \\
 \Pi(k; s, p)&=\int_0^s \tfrac{1}{1-k\sin^2(u)}\tfrac{du}{%
               \sqrt{1-p^2\sin^2(u)}},
\end{align*}
where $p \in [0,1]$ as before. Evaluated at $s=\pi/2$, we obtain  
$F_p$ ($E_p$, $\Pi^k_p$) the \emph{complete elliptic integrals 
of first (second, third) kind}. $F_p$ is of particular interest: 
The functions $\jac{sn}{p}$ and $\jac{cn}{p}$ are periodic with 
period $4F_p$, whilst $\jac{dn}{p}$ is periodic with period $2F_p$. For this note, it is 
useful to introduce notation for the composition of elliptic integrals 
with the amplitude function $\jac{am}{}$. We will denote
\begin{align*}
	F(\jac{am}{p}(s), p) &=: F(s|p), \\
	E(\jac{am}{p}(s), p) &=: E(s|p), \\
	\Pi(k; \jac{am}{p}(s), p) &=: \Pii{k}{p}{s}. 
\end{align*}

In \Cref{sec:4}, we utilise transformation formulas for $\Pi$, which 
can be written as
\begin{align*}
	\Pii{k}{p}{s} = \int_0^s \tfrac{du}{1-k\jac{sn}{p}^2(u)}.
\end{align*}
In this form, we see that $\Pi$ is a real function for all $p\in \R$ 
and for imaginary arguments, since $\jac{sn}{}$ has real or imaginary 
values in these cases. Using Jacobi's transformations of the last 
subsection, we learn
\begin{align*}
 &\Pii{k}{\tfrac{1}{p}}{as} = p~\Pii{kp^2}{p}{\tfrac{as}{p}}, \\
 &\Pii{k}{p}{i~as} =\left\{
  \begin{array}{l}
      p^2\tfrac{i~as}{q^2}-\tfrac{i}{q^2}E(s|q),~\textrm{for}~k=1, \\
      \tfrac{i~as}{1-k} - \tfrac{ik}{1-k}~\Pii{1-k}{q}{as},~\textrm{for}~k\neq 1,
  \end{array}\right. \\
 &\Pii{k}{ip}{as}=\tfrac{ap'^2}{k'}~s+\tfrac{kq'^3}{k'}%
	\Pii{k'}{p'}{\tfrac{as}{q'}},
\end{align*}
for $a\in \R^\times$ with $p', q'$ as in the last section and $k' = p'^2 + kq'^2$.

\section*{Acknowledgements}
The authors would like to thank Feray Bayar, Fran Burstall, Joseph Cho, Shoichi Fujimori, Wayne Rossman 
and Yuta Ogata for fruitful and helpful discussions. 
Part of this work was done during a six months stay in Japan,
granted to the third author by
the FWF/JSPS Joint Project grant I3809-N32
 "Geometric shape generation".
Further, this work has been partially supported by
the FWF research project P28427-N35
 "Non-rigidity and symmetry breaking". The second author was also supported by GNSAGA of INdAM and the MIUR grant ``Dipartimenti di Eccellenza'' 2018 - 2022, CUP: E11G18000350001,
 DISMA, Politecnico di Torino.

\bibliographystyle{plain}
\bibliography{CLWfinal}

\end{document}